\documentclass[12pt,reqno]{amsart}
\usepackage{amsmath,amsfonts,amssymb}
\allowdisplaybreaks
\usepackage[utf8]{inputenc}
\usepackage[english]{babel}
\usepackage[T1]{fontenc}
\usepackage{lmodern}
\usepackage{mathtools}
\usepackage{scrextend}
\usepackage{enumitem}
\usepackage{graphicx}
\usepackage{stmaryrd}

\usepackage[all]{xy}
\usepackage[scr=boondox,frak=euler]{mathalfa}

\usepackage{tikz}
\usetikzlibrary{intersections}
\usetikzlibrary {calc}

\usepackage{bm}
\usepackage[normalem]{ulem}
\numberwithin{equation}{section}
\usepackage{xcolor}

\usepackage{centernot}

\usepackage[initials]{amsrefs}

\renewcommand{\leqslant}{\leq}
\renewcommand{\geqslant}{\geq}

\newcommand{\cb}{{\mathcal B}}
\newcommand{\cc}{{\mathcal C}}

\newcommand{\cf}{{\mathcal F}}

\newcommand{\cn}{{\mathcal N}}
\newcommand{\cm}{{\mathcal M}}

\newcommand{\E}{{\mathbb E}}

\newcommand{\mf}{\mathfrak{F}}
\newcommand{\MM}{{\mathbb M}}
\newcommand{\N}{{\mathbb N}}
\renewcommand{\P}{{\mathbb P}}

\newcommand{\R}{{\mathbb R}}

\newcommand{\nn}{n}
\newcommand{\f}{{\rm f}}

\newcommand{\vir}{,\,}

\newcommand{\Supp}{{\rm Supp}\;}

\newcommand{\expp}[1]{\mathop {\mathrm{e}^{ #1}}}

\title[Additive functionals of conditioned BGW trees]{Global regime 
  for general additive functionals
of conditioned Bienaymé-Galton-Watson trees}
\date{\today}

\usepackage[text={500pt,660pt},centering]{geometry}
\usepackage[colorlinks=true, pdfstartview=FitV, linkcolor=blue, citecolor=red, urlcolor=blue,pagebackref=false]{hyperref}
\usepackage[noabbrev]{cleveref}

\makeatletter
\def\l@subsection{\@tocline{2}{0pt}{2.5pc}{5pc}{}}
\makeatother

\usepackage{amsthm}

\theoremstyle{plain}
\newtheorem{thm}{Theorem}[section]
\newtheorem{corollary}[thm]{Corollary}
\newtheorem{proposition}[thm]{Proposition}
\newtheorem{lemma}[thm]{Lemma}
\newtheorem*{thm*}{Theorem}
\newtheorem*{proposition*}{Proposition}

\theoremstyle{definition}
\newtheorem{remark}[thm]{Remark}
\newtheorem{example}[thm]{Example}

\newtheorem*{remark*}{Remark}

\usepackage{amsopn}

   {\left\lbrace\begin{array}{@{}l@{}}}
   {\end{array}\right.}
\newcommand{\norm}[1]{\left\lVert#1\right\rVert}
\newcommand{\real}{\mathbb{R}}
\newcommand{\dd}{\mathrm{d}}
\newcommand{\ex}[1]{\operatorname{\mathbb{E}}\left[#1\right]}
\newcommand{\n}{\operatorname{\mathbf{N}}}
\newcommand{\pr}[1]{\operatorname{\mathbb{P}}\left(#1\right)}

\mathchardef\mhyphen="2D

\renewcommand{\epsilon}{\varepsilon}
\renewcommand{\phi}{\varphi}

\newcommand{\ghp}{d_\mathrm{GHP}}

\newcommand{\costmh}{\mathcal{A}^{\mathfrak{mh}}}

\newcommand{\costj}{\mathcal{A}^{\mathfrak{mh}, \circ}}

\newcommand{\psij}{\Psi^{\mathfrak{mh}}}
\newcommand{\M}{\mathcal{M}}

\newcommand{\MF}{\mathcal{M}_\mathfrak{F}}

\newcommand{\dF}{d_\mathfrak{F}}

\newcommand{\B}{\mathcal{B}}

\newcommand{\BF}{\mathcal{B}_{\mathfrak{F}}}
\newcommand{\Btr}{\mathcal{B}_{\mathrm{tr}}}
\newcommand{\law}{\overset{\scriptscriptstyle (d)}{=}} 
\newcommand{\cvlaw}{\xrightarrow[]{\scriptscriptstyle (d)}} 
\newcommand{\cvlawd}{\xrightarrow[n \to \infty]{(d)}}
\newcommand{\cvlawmean}{\xrightarrow[n \to \infty]{(d)+\mathrm{mean}}}
\newcommand{\lawd}{\overset{(d)}{=}}

\newcommand{\ind}{\mathbf{1}}
\newcommand{\tree}{T}
\newcommand{\rdtree}{\mathcal{T}}
\newcommand{\rddtree}{\tau}
\renewcommand{\L}[1]{\norm{#1}_{\mathrm{L}}}

\renewcommand{\*}[1]{#1^{\nu}}
\renewcommand{\H}{\mathfrak{h}}
\newcommand{\Hexc}{\mathfrak{h}}
\newcommand{\m}{\mathfrak{m}}
\renewcommand{\root}{\emptyset}
\newcommand{\T}{\mathbb{T}}
\newcommand{\TT}{\mathbb{T}_0}
\newcommand{\ttt}{\mathrm{T}}

\newcommand{\Gammaeuler}{\Gamma}
\newcommand{\e}{\mathrm{e}}
\newcommand{\Br}{B_{\mathrm{ex}}}

\newcommand{\excm}[1]{\operatorname{\mathbf{N}}^{(#1)}}
\newcommand{\spn}{\lambda_0}
\newcommand{\supp}{\Delta}

\newcommand{\height}{\mathsf{H}}
\newcommand{\mass}{\mathsf{S}}
\newcommand{\theight}{\mathrm{T}}
\newcommand{\tmass}{\mathrm{T}}
\newcommand{\ppp}{\mathfrak{T}}
\newcommand{\sigmaxi}{\sigma_\xi}

\author{Romain Abraham}
\address{Romain Abraham,
Institut Denis Poisson,
Universit\'{e} d'Orl\'{e}ans,
Universit\'e de Tours,
CNRS,
France}
\email{romain.abraham@univ-orleans.fr}

\author{Jean-Fran\c{c}ois Delmas}
\address{Jean-Fran\c{c}ois Delmas,
 CERMICS, Ecole des Ponts, France}
\email{delmas@cermics.enpc.fr}

\author{Michel Nassif}
\address{Michel Nassif,
CERMICS, Ecole des Ponts, France}
\email{michel.nassif@enpc.fr}

\begin{document}

\begin{abstract}
  We give  an invariance principle  for very general additive  functionals of
  conditioned Bienaymé-Galton-Watson trees in the global regime when the
  offspring distribution  lies in the  domain of attraction of  a stable
  distribution, the limit being an  additive functional of a stable Lévy
  tree.   This includes  the case  when the  offspring distribution  has
  finite variance (the Lévy tree being  then the Brownian tree). We also
  describe,  using  an  integral  test,  a  phase  transition  for  toll
  functions depending on the size and height.
\end{abstract}



\maketitle


\section{Introduction}

In view of the many applications of trees (in computer science, biology,
physics, ...), the  study of additive functionals on  large random trees
has seen a lot of development in recent years, see references below.  In
this  paper,  we  consider   asymptotics  for general  additive  functionals  on
conditioned  Bienaymé-Galton-Watson   (BGW  for  short)  trees   in  the
so-called global regime.

Recall that a  functional $F$ defined on finite  rooted ordered discrete
trees is said to be additive if it satisfies the recursion
\begin{equation}
F(\mathbf{t}) = \sum_{i=1}^d F(\mathbf{t}_i) + f(\mathbf{t}),
\end{equation}
where $\mathbf{t}_1,  \ldots, \mathbf{t}_d$  are the subtrees  rooted at
the $d$ children of the root of the tree $\mathbf{t}$ and $f$ is a given
toll function. Notice that this can also be written as
\begin{equation}\label{additive functional}
F(\mathbf{t}) = \sum_{w \in \mathbf{t}} f(\mathbf{t}_w),
\end{equation}
where $\mathbf{t}_w$ is the subtree of $\mathbf{t}$ above the vertex $w$
and rooted at $w$. Such  functionals are encountered in computer science
where  they  represent the  cost  of  divide-and-conquer algorithms,  in
phylogenetics where  they are used as  a rough measure of  tree shape to
detect imbalance  or in  chemical graph  theory where  they appear  as a
predictive tool  for some chemical  properties. Among these,  we mention
the total path length defined as the sum of the distances to the root of
all vertices, the Wiener index \cite{szekely2016problems} defined as the
sum  of  the  distances  between   all  pairs  of  vertices,  the  shape
functional,  the Sackin  index,  the Colless  index  and the  cophenetic
index,   see  \cite{shao1990tree}   for  their   definitions  and   also
\cite{delmas2018} for  their representation using  additive functionals,
and the references therein.  See also \cite{ralaivaosaona2020acentral}
for other functionals such that the number of matchings,  dominating
sets, independent sets for trees. 
We also  mention the Shao and Sokal's $B_1$
index \cite{agapow2002power, shao1990tree} defined by
\begin{equation}
B_1(\mathbf{t}) = \sum_{\substack{w \in \mathbf{t}^\circ\\w\neq \emptyset}} \frac{1}{\H(\mathbf{t}_w)},
\end{equation}
where for every finite rooted ordered tree $\mathbf{t}$,
$\H(\mathbf{t})$ is its height and $\mathbf{t}^\circ$ is the set of
internal vertices. It is used for assessing the balance of phylogenetic
trees, see \emph{e.g.}
\cite{fabre2012glimpse,jabot2009inferring,kirxpatrick1993searching,poon2015phylodynamic,scott2020inferring}. 

We shall consider in this paper random discrete trees $\rddtree^n$ which
are BGW trees conditioned to have $n$ vertices, and then study the limit
of  rescaled additive  functionals as  $n$  goes to  infinity.  One  can
distinguish between  local and global  regime. In the local  regime, the
toll function  is small or even  vanishes when the subtree  is large; so
the main  contribution to the  additive functional comes from  the small
subtrees. These being almost  independent, we understand intuitively why
the       limit        distribution       is        Gaussian.        See
\cite{janson2016asymptotic,ralaivaosaona2020acentral,wagner2015central}
for asymptotic  results in the local  regime. In the global  regime, the
toll  function  is  large  when  the  subtree  is  large;  so  the  main
contribution  comes from  large subtrees  which are  strongly dependent.
This intuitively explain why we expect  the limit to be non-Gaussian. As
far as we  know, asymptotic results in the global  regime deal with toll
functions depending only on the size.   In this paper, we shall focus on
the  global  regime for  general  toll  functions.  In  particular,  our
results apply to toll functions depending  on the size and height.  When
the   toll   function   is   monomial   in  the   size   of   the   tree
$f(\mathbf{t})  =   |\mathbf{t}|^{\alpha'}$,  with   $|\mathbf{t}|$  the
cardinal  of   $\mathbf{t}$,  Fill  and   Kapur  \cite{fill2004limiting}
observed a  phase transition at $\alpha'  = 1/2$ for binary  trees under
the  Catalan  model  (which  is   a  special  case  of  conditioned  BGW
trees): the global regime corresponds to $\alpha'>1/2$. This was later generalized  by Fill and
Janson \cite{fillsum} to 
BGW  trees with  critical  offspring distribution  with finite  variance
using techniques from complex analysis; they identified a local regime for $\alpha' < 0$ and an intermediate regime for $0 < \alpha' < 1/2$. When the offspring distribution
has infinite variance  but lies in the domain of  attraction of a stable
distribution  with  index  $\gamma   \in  (1,2]$,  Delmas,  Dhersin  and
Sciauveau  \cite{delmas2018}  proved  convergence  in  distribution  for
$\alpha' \geqslant  1$ using stable  Lévy trees and conjectured  a phase
transition  at $\alpha'  =1/\gamma$. We  shall prove this conjecture,  as a particular case  of our
main result, see Theorem \ref{thm:main}.

Let $\xi$ be a $\N$-valued random variable. We write BGW($\xi$) tree
for a  BGW tree  with offspring  distribution (the law of) $\xi$. We denote by $\rddtree^n$  a BGW($\xi$) tree conditioned to have
$n$ vertices and we assume that $\xi$ is critical, \emph{i.e.} $\ex{\xi} = 1$,
      nondegenerate, \emph{i.e.} $\pr{\xi = 0} >0$, and that it belongs to the domain of attraction of a stable
      distribution with index $\gamma \in (1,2]$, \emph{i.e.} there  exists a positive sequence    $(b_n,    \,    n     \geqslant    1)$    such    that    if
$(\xi_n,\, n \geqslant 1)$ is a sequence of independent random variables
with      the       same      distribution      as       $\xi$      then
$b_n^{-1}  \left(   \sum_{k=1}^n  \xi_k   -  n  \right)$   converges  in
distribution towards a  stable random variable whose Laplace  transform is given
by  $  \exp(\kappa  \lambda^\gamma)$  for $\lambda\geq  0$,  with  index
$\gamma  \in   (1,2]$  and   normalizing  constant  $\kappa   >0$  (the
 constant  $\kappa$   depends      on     the      choice     of the sequence
$(b_n, \,  n \geqslant 1)$.  Under  these assumptions, it is  also well
known  that, as  $n$ goes  to infinity,  $\rddtree^n$ properly  rescaled
converges in distribution with respect to the Gromov-Hausdorff-Prokhorov topology to
the  stable  Lévy tree  $\rdtree$  with  index $\gamma$  (and  branching
mechanism  $\psi(\lambda) =  \kappa \lambda^\gamma$)  which is  a rooted
random real tree (see Section \ref{levy tree} for a precise definition),
see Aldous  \cite{aldous1991continuum} for the finite  variance case and
Duquesne \cite{duquesne2003limit} for the  general case. The stable Lévy
tree is a generalization of Aldous' Brownian continuum random tree which
corresponds to $\gamma = 2$. We recall  that the stable Lévy tree is the
real tree  coded by the normalized  excursion of the height  process associated with a stable L\'evy process and
that it  codes the  genealogy of continuous-state  branching processes,
see \emph{e.g.} Le Gall and  Le Jan \cite{le1998branching}, Duquesne and
Le Gall \cite{duquesne2002random,  duquesne2005probabilistic}. We recall
that  any  real  tree  $\tree$   is  endowed  with  the  length  measure
$\ell(\dd  y)$ (which  roughly speaking is  the Lebesgue
measure on  the branches of the tree)  and that the Lévy  tree is naturally
endowed  with a  mass measure  (which  roughly speaking  is the  uniform
probability measure on the infinite  set of leaves). One of our main results can be stated as follows. We refer the reader to
Proposition \ref{general functional} and Theorem \ref{mass+height
  convergence} for  more general statements. Recall that $\mathbf{t}^\circ$
denotes the set of
internal vertices of the discrete tree  $\mathbf{t}$. 

\begin{thm}
\label{thm:main}
Let $\rddtree^n$ be a BGW($\xi$) tree conditioned to have $n$ vertices,
with $\xi$  being critical, nondegenerate and in the domain of attraction of a stable distribution with index $\gamma \in (1,2]$.   We suppose moreover that the sequence $(b_n,  \,  n \geqslant  1)$ defined  as above is such that $(b_n /n^{1/\gamma}, \, n \geqslant 1)$ is bounded away from zero and infinity. Let $\rdtree$   be   the  stable   Lévy   tree   with  branching   mechanism
$\psi(\lambda)  =  \kappa  \lambda^\gamma$. Let 
$\alpha',\beta \in \real$. 
	\begin{enumerate}[label=(\roman*),leftmargin=*]
		\item If $\gamma \alpha' + (\gamma-1)\beta >1$, we have the
          convergence in distribution and of the first moment 
\begin{equation}
\label{convergence length measures}
	\frac{b_n^{1+\beta}}{n^{1+\alpha'+\beta}} \sum_{w \in
      \rddtree^{n,\circ}} |\rddtree^n_{w}|^{\alpha'}\, 
    \H(\rddtree^n_{w})^\beta \cvlawmean \int_{\rdtree}
    \m(\rdtree_y)^{\alpha'}\,  \H(\rdtree_y)^\beta\, \ell(\dd y),
\end{equation}
where the right hand-side of
\eqref{convergence length measures} has finite mean and, for $y\in \rdtree$,  $\rdtree_y$ is the subtree of $\rdtree$
above $y$, $\m(\rdtree_y)$ 
is its mass, and $\H(\rdtree_y)$ its height. 
\item If $\gamma \alpha' + (\gamma-1)\beta \leqslant 1$, we have the convergence in distribution and of the first moment
		\begin{equation}
		\frac{b_n^{1+\beta}}{n^{1+\alpha' +\beta}} \sum_{w \in
          \rddtree^{n,\circ}} |\rddtree^n_{w}|^{\alpha'} \,
        \H(\rddtree^n_{w})^\beta \cvlawmean \infty.\end{equation} 
	\end{enumerate}
\end{thm}
 We complete the previous result with some comments.

\begin{remark}
\label{rem:main}
\begin{enumerate}[label=(\roman*),leftmargin=*]
\item 
From Theorem \ref{thm:main}, we obtain a phase change for functionals of
the mass and height at $\gamma \alpha' + (\gamma -
1)\beta=1$. Heuristically, the condition on $\alpha'$ and $\beta$ is due
to the fact that the height of a (unnormalized) stable Lévy tree scales
as its mass to the power $(\gamma-1)/\gamma$. Let us mention that this
phase change is specific to BGW trees, see Remark \ref{non generic phase
  transition} in this direction. 
   \item
   See conditions \ref{xi1} and \ref{xi2} in Section \ref{Sect BGW} for a more detailed discussion of the assumptions on the offspring distribution. The additional boundedness assumption on $(b_n/n^{1/\gamma}, \, n\geqslant 1)$ is also equivalent to \ref{xi3}. This latter can be dropped in (i) of Theorem \ref{thm:main} when $\alpha'\geqslant 1$ and $\beta \geqslant 0$ according to Proposition \ref{convergence for continuous functions discrete}.
   \item We also
have the
convergence (and finiteness) of the moments  of all order $p>1$ in
\eqref{convergence length measures} as soon as  $p(\gamma \alpha +
(\gamma-1)\beta)>1-\gamma$, with $\alpha=\alpha'-1$, see
Proposition~\ref{general functional}.  In particular for $\beta=0$, we
have the convergence of all nonnegative moments for $\alpha'\geq
1$. However, in the finite variance case, for $\alpha'\in (1/2, 1)$ (and $\beta=0$), our result is not
optimal, see (vi) below. 
\item Theorem \ref{thm:main} generalizes a result by Delmas, Dhersin and
  Sciauveau where only 
functionals of the mass are considered (\emph{i.e.} $\beta = 0$), see
\cite[Lemma 4.6]{delmas2018}. 
In particular, we prove the conjecture stated therein:  when $\beta=0$, there is a
phase transition  at $\alpha' = 1/\gamma$ (the parameter $\alpha$
therein corresponds to $\alpha'-1$ here). If we fix $\alpha' =0$ and
let $\beta$ 
vary, the phase transition  occurs at $\beta = 1/(\gamma-1) \geqslant 1$. In
particular, Shao and Sokal's $B_1$ index, which corresponds to $\alpha
=0$ and $\beta = -1$, lies in the local regime, whatever the value of
the index $\gamma$ and is therefore not covered by our
results. See also (v) below. 
\item If the offspring distribution has finite variance $\sigmaxi^2 \in
(0,\infty)$, one can take $b_n = b\sqrt{n}$ in which case $\rdtree$ is
distributed as the Brownian continuum random tree with branching
mechanism $\psi(\lambda) = \sigmaxi^2\lambda^2/(2b^2)$.  For $b=\sigmaxi$, the contour
process of $\rdtree$ is a standard Brownian motion under its normalized excursion
measure. 
\item\label{local regime  remark} Assume that the offspring  distribution has
  finite  variance   $\sigmaxi^2  \in  (0,\infty)$,  which   implies  that
  $\gamma=2$.  We consider  the asymptotics in the local
  regime  of $\sum_{w \in
      \rddtree^{n,\circ}} |\rddtree^n_{w}|^{\alpha'}\, 
    \H(\rddtree^n_{w})^\beta$, that is when $\alpha', \beta \in \R $
    such that $2\alpha' + \beta <0$. Denote by  $F_{\alpha',\beta}$ the additive
  functional \eqref{additive functional} associated with the toll function
  $f_{\alpha',\beta}(\mathbf{t})         =        |\mathbf{t}|^{\alpha'}
  \H(\mathbf{t})^\beta\ind_{\{|\mathbf{t}|                       >1\}}$.
  By  \cite[Theorem  1.5]{janson2016asymptotic}   and  Lemma  \ref{local
    regime lemma}, we have
\[
\frac{F_{\alpha',\beta}(\rddtree^n) - n\mu}{\sqrt{n}} \cvlawd
\mathcal{N}(0,\varsigma^2), 
\]
where $\mu, \varsigma^2$ are finite and given by $\mu =
\ex{f_{\alpha',\beta}(\rddtree)}$ and by $\varsigma^2 = 2 
\ex{f_{\alpha',\beta}(\rddtree)\left(F_{\alpha',\beta}(\rddtree) 
 -    |\rddtree|\mu\right)}  
 - \operatorname{Var}( f_{\alpha',\beta}(\rddtree)) - \mu^2/\sigmaxi^2$,
and  
$\rddtree$ is the corresponding unconditioned BGW tree. In particular,
this covers Shao and Sokal's $B_1$ index (where $\alpha'=0$ and
$\beta=-1$). Notice that this leaves a gap for $0\leqslant
2\alpha'+\beta \leqslant 1$. At least when $\beta = 0$, the situation is
well understood. Fill and Janson \cite{fillsum} identify three different regimes: the global regime for $\alpha' >1/2$, the local regime for $\alpha' < 0$ and an intermediate regime for $0<\alpha'<1/2$. The nontrivial asymptotic behavior of
$F_{\alpha',\beta}(\rddtree^n)$ for $\gamma\in (1, 2)$ and
$\gamma\alpha' + (\gamma-1) \beta \leq 1$ (that is the non global regime in
the non quadratic case) is an open question. 
\item When $\rddtree^n$  is uniformly distributed among the  set of full
  binary  ordered  trees  with  $n$ vertices  (which  corresponds  to  a
  conditioned BGW($\xi$) tree with $\pr{\xi =  0}= \pr{\xi = 2} = 1/2$),
  Fill and  Kapur \cite{fill2004limiting}  studied the local  and global
  regime  when  the  toll  function  is  a power  of  the  size  of  the
  tree. Concerning  the global  regime, they  showed the  convergence in
  distribution,  using  the  convergence  of  all  positive  moments  in
  \eqref{convergence   length  measures}   for  $\alpha'   >  1/2$   and
  $\beta =0$, see Eq. (3.14) and  Proposition 3.5 therein. In that case,
  one can take $b_n = \sqrt{n}$  and $\rdtree$ is the Brownian tree with
  branching  mechanism  $\psi(\lambda)   =  \lambda^2/2$. See also  Fill and
Janson \cite{fillsum} for general  critical  offspring distribution with
finite  variance. The  explicit
  formula for the first moment of the right  hand-side of \eqref{convergence
    length   measures}   are   given   by   the   right   hand-side   of
  \eqref{eq:moment-Psi-Br} with $\kappa=1/2$ and $\alpha=\alpha'-1$.
\item As an application, using 
  \eqref{convergence length 
      measures}, we obtain, when  
$\alpha' > 1/\gamma$,  in Example \ref{power log} (with
$\alpha'=\alpha+1$) an   
    asymptotic expansion in distribution for 
$ b_n\, n^{-(1+\alpha')}  \sum_{w \in
  \rddtree^{n,\circ}}\left|\rddtree^{n}_w\right|^{\alpha'} \log
\left|\rddtree^n_w\right|$. 
\end{enumerate}
\end{remark}

More generally, if  one views a discrete  tree as a real  tree, then the
left-hand side in \eqref{convergence length  measures} is related to the
discrete                          length                         measure
$\ell_n(\dd y)=\sum_{w\in  \rddtree^n} \delta_w(\dd y)$  of $\rddtree^n$
(after rescaling by  $b_n/n$). One way to interpret the  result would be
to       say       that       the       sequence       of       measures
$\int_{\rddtree^n} \delta_{\rddtree^n_y} \,  \ell_n(\dd y)$ converges in
distribution  to $\int_{\rdtree}  \delta_{\rdtree_y}\,  \ell(\dd y)$  in
some   sense.  One   might  then   hope  to   prove  that   the  mapping
$\tree \mapsto \int_\tree \delta_{\tree_y}\,  \ell(\dd y)$ is continuous
on the space of compact real trees. This is not true however, see Remark
\ref{non generic  phase transition}, one  problem being that  the length
measure  is not  finite in  general.  To overcome  this difficulty,  our
approach,  inspired by  \cite{delmas2018}, consists  in considering  the
length  measure biased  by  the  size of  the  subtree  above $y$,  thus
penalizing small subtrees.

More precisely let $\T$ be the space of (equivalent classes of) weighted
rooted   compact    real   trees   (\emph{i.e.}   the    set   of   quadruplets
$(\tree,\root,d,\mu)$ where  $(\tree,d)$ is a real  tree, $\root$ is a
distinguished vertex of $\tree$ called the root,  and the mass measure $\mu$
is a  finite measure on  $\tree$). 
We recall that the length measure $\ell$ on a real tree $(\tree, d)$ has
an intrinsic definition. 
For every  $(\tree,\root,d,\mu)\in\T$, we
define  a  measure  $\Psi_\tree$  on $\T\times  \R_+$  by:  for
every nonnegative  measurable function 
$f$ defined on $\T\times \R_+$, 
\begin{equation}
\label{eq:def-Psi-intro}
\Psi_\tree(f) = \int_\tree \mu(\tree_y) f\left(\tree_y, H(y)\right)\, \ell(\dd y),
\end{equation}
where $H(y)=d(\root, y)$ denotes the height of $y$ (\emph{i.e.} the
distance to the root) in $\tree$. 
We also consider the measure $\psij_\tree$ on $\R_+^2$ defined similarly
to $\Psi_\tree$ for functions depending only on the mass and height of
the tree, see \eqref{psi joint}.

If $\mathbf{t}$ is a finite rooted ordered tree and $a>0$, we denote by
$a\mathbf{t}$ the real tree associated with $\mathbf{t}$, rescaled so
that all edges have length $a$ and equipped with the uniform probability
measure on the set of vertices whose height is an integer multiple of
$a$, see Section \ref{real trees} for a precise definition. Furthermore, for
$w\in \mathbf{t}$, we write $aw$ for the corresponding vertex in
$a\mathbf{t}$ and  $a\mathbf{t}_w$ for the subtree of
$a\mathbf{t}$ above $aw$. The height of $w$  in $\mathbf{t}$
is denoted by
$H(w)$; and thus the height of $aw$ in $a\mathbf{t}$ is $a H(w)$. In the spirit of \cite{delmas2018}, we consider the measure 
$\mathcal{A}^\circ_{\mathbf{t},a}$ on $\T\times \R_+$ defined by: for nonnegative  measurable function $f $
defined on $\T\times \R_+$, 
\begin{equation}
\label{eq:def-A-intro}
\mathcal{A}^\circ_{\mathbf{t},a}(f) = \frac{a}{|\mathbf{t}|} \sum_{w \in
  \mathbf{t}^\circ}|\mathbf{t}_w| f\left(a\mathbf{t}_w, aH(w)\right). 
\end{equation}
In  \eqref{eq:def-A-intro}, instead  of  summing over  all the  internal
vertices ($w\in \mathbf{t}^\circ$) one could  also sum over all vertices
including the  leaves ($w\in \mathbf{t}$);  in this case the  measure is
denoted by  $\mathcal{A}_{\mathbf{t},a}$. The two measures  are close in
total                            variation                            as
$d_{\mathrm{TV}}(\mathcal{A}_{\mathbf{t},a},
\mathcal{A}_{\mathbf{t},a}^\circ)\leqslant  a$,
see \eqref{distance a  a}.  We mention  that the
measure   $\mathcal{A}_{\mathbf{t},a}$   was   already   considered   in
\cite{delmas2018} for functions $f$ depending only on the size. 
 
 For every finite rooted ordered tree $\mathbf{t}$ and $a>0$, we show
 (see Lemma \ref{distance between measures}) that the measures
 $\mathcal{A}_{\mathbf{t},a}^\circ$  and
 $\mathcal{A}_{\mathbf{t},a}$ can be approximated by $\Psi_{a\mathbf{t}}$.
  In Proposition \ref{length measure}, we give another expression for $\Psi_\tree$:
 \begin{equation}
\Psi_{\tree}(f) = \int_{\tree} \mu(\dd x)\int_{0}^{H(x)}f\left(\tree_{r\vir x},r\right)\dd r,
 \end{equation}
for every nonnegative measurable function $f$ defined on $\T \times \R_+ $. Here $\tree_{r\vir x}$ is the subtree of $\tree$ above level $r$ containing $x$.
This  latter expression of  $\Psi_\tree$ is used to prove it is
continuous as a function of $\tree$, see Proposition \ref{continuity of
  the measure}. 
\begin{thm}
\label{thm:Psi-cont}
The mapping $\tree \mapsto \Psi_\tree$, from $\T$ endowed with the
Gromov-Hausdorff-Prokhorov topology to $\M(\T \times \R_+)$,  the space
of nonnegative finite measures on $\T \times \R_+$,  endowed with the topology of weak convergence, is well defined and continuous.
\end{thm}
This allows to  derive a general invariance principle:  for any sequence
of  random  discrete  trees  $(\rddtree^n,  \,  n  \in  \N)$  such  that
$a_n  \rddtree^n$ converges  in distribution  to some  random real  tree
$\rdtree$    in    the   Gromov-Hausdorff-Prokhorov    topology    where
$(a_n, \, n\in \N)$ is a  sequence of positive numbers converging to $0$
and such  that $(a_n \ex{\H(\rddtree^n)},  \, n\in \N)$ is  bounded, one
has    the    convergence    in    distribution    of    the    measures
$\mathcal{A}_{\rddtree^n,a_n}^\circ$  and $\mathcal{A}_{\rddtree^n,a_n}$
to $\Psi_\rdtree$ (this is a  consequence of Lemma \ref{distance between
  measures} and  Theorem \ref{thm:Psi-cont}).  For example,  this applies
to Pólya trees, see Remark \ref{rem:polya}, which were shown to converge
to    the     Brownian    tree,    see     \cite{haas2012scaling}    and
\cite{panagiotou2018scaling}.   For BGW  trees,  we  have the  following
result which is a direct consequence of  the convergence on conditioned
BGW trees to stable Lévy tree, see \cite{duquesne2003limit}, and
Theorem \ref{thm:Psi-cont} and  Lemma \ref{distance between  measures}.  

\begin{corollary}
\label{cor:cv-intro} 
Let $\rddtree^n$ be a BGW($\xi$)  tree conditioned to have $n$ vertices,
with $\xi$  satisfying   \ref{xi1}  and
\ref{xi2},  and $(b_n,  \,  n \geqslant  1)$ be  defined  as in Theorem
\ref{thm:main}.  Let 
$\rdtree$   be   the  stable   Lévy   tree   with  branching   mechanism
$\psi(\lambda)  =  \kappa  \lambda^\gamma$. We have the following 
          convergence in distribution and of all positive moments
\[
\frac{b_n}{n^{2}} \sum_{w \in
  \rddtree^{n,\circ}}|\rddtree^n_{w}|f\left(\frac{b_n}{n}\rddtree^n_{w},\frac{b_n}{n}H(w)\right)\xrightarrow[n
\to \infty]{(d) + \mathrm{moments}} \Psi_{\rdtree}(f),
\]
where $f$ is a bounded continuous real-valued function defined on
$\T\times \R_+$. 
\end{corollary}

We improve this result by allowing the function $f$ to blow up as either
the  mass or  the  height goes  to zero  under  the stronger  assumption
\ref{xi3}: see Proposition \ref{general  functional}, and more precisely
Theorem  \ref{mass+height  convergence}  when  $f$ is  a  product  of  a
function of the  mass and a function  of the height, one of  them being a
power  function.  As   a  particular  case,  property   (i)  of  Theorem
\ref{thm:main} gives  a precise result when  $f$ is a power  function of
the  mass and  the height.  Related  to this  latter result,  we give  a
complete description of the finiteness of $\psij_{\rdtree}(f)$ for power
functions  $f$ where  $\rdtree$  is the  stable Lévy  tree  and we  also
compute its  first moment. We refer  to Corollaries \ref{cor:expectation
  functional  on  Levy  tree}  and \ref{moments  brownian  height},  and
Proposition   \ref{joint   phase   transition}  for   a   more   general
statement.  By   convention,  we  write   $\psij_\rdtree(g(x)h(u))$  for
$\psij_\rdtree(f)$  where  $f(x,u) =  g(x)h(u)$  and  we  see $g$  as  a
function of the mass and $h$ as a function of the height. In particular,
thanks to  \eqref{eq:def-Psi-intro}, we  have for $\alpha,  \beta\in \R$
that
$\psij_\rdtree(x^\alpha             u^\beta)=             \int_{\rdtree}
\m(\rdtree_y)^{\alpha'}\,     \H(\rdtree_y)^\beta\,     \ell(\dd     y)$
with $\alpha'=\alpha+1$.
\begin{proposition}
\label{prop:moment-intro}
	Let $\rdtree$ be the stable Lévy tree with branching mechanism
    $\psi(\lambda) = \kappa \lambda^\gamma$ and let $\alpha, \beta \in
    \real$. We have
\begin{align}\gamma \alpha + (\gamma-1)(\beta+1) > 0 
\quad \Longleftrightarrow \quad 
\psij_\rdtree(x^\alpha u^\beta) < \infty \ \text{a.s.} 
\quad \Longleftrightarrow \quad
\ex{\psij_\rdtree(x^\alpha u^\beta)} <\infty, \\
	\gamma \alpha + (\gamma-1)(\beta+1) \leqslant 0 
\quad \Longleftrightarrow \quad 
\psij_\rdtree(x^\alpha u^\beta) = \infty \ \text{a.s.} 
\quad \Longleftrightarrow \quad
\ex{\psij_\rdtree(x^\alpha u^\beta)} =\infty.
	\end{align}
	For every $\alpha,\beta \in \real$ such that $\gamma \alpha + (\gamma-1)(\beta+1) >0$, we have
	\begin{equation}
\label{eq:moment-psimh}
	\ex{\psij_\rdtree(x^\alpha u^\beta)} = \frac{1}{\kappa^{1/\gamma}|\Gammaeuler(-1/\gamma)|}\mathrm{B}\!\left(\alpha + (\beta+1)(1-1/\gamma),1-1/\gamma\right)\ex{\H(\rdtree)^\beta},
	\end{equation}
	where $\Gammaeuler$ is the gamma function and $\mathrm{B}$ is the
    beta function. Furthermore, we have $\ex{\psij_\rdtree(x^\alpha
      u^\beta)^p} < \infty$ for every $p\geq 1$ such that $p(\gamma \alpha +
    (\gamma-1)\beta) > 1-\gamma$. In the Brownian case ($\gamma=2)$, for
    every $\alpha, \beta\in \real$ such that $2\alpha + \beta +1>0$, we
    have	 
	\begin{equation}
\label{eq:moment-Psi-Br}
\ex{\psij_\rdtree(x^\alpha u^\beta)} = \frac{1}{\sqrt{\pi\kappa}}\left(\frac{\pi}{\kappa}\right)^{\beta/2}\xi(\beta)\mathrm{B}\!\left(\alpha + \frac{\beta+1}{2},\frac{1}{2}\right),
	\end{equation}
	where $\xi$ is the Riemann xi function defined by $\xi(s) = \frac{1}{2} s(s-1)\pi^{-s/2} \Gammaeuler(s/2)\zeta(s)$ for every $s \in \mathbb{C}$ and $\zeta$ is the Riemann zeta function.
\end{proposition}
Thanks to  Duquesne and Wang \cite{duquesne2017decomposition},
$\ex{\H(\rdtree)^\beta}$ is finite for all $\beta\in \R$, so that the
right hand side of \eqref{eq:moment-psimh} is finite.

We conclude the introduction by giving a formula for the  distribution of
$\rdtree_y$, the subtree above $y$,  when $y$ is chosen according to the
length measure $\ell(\dd y)$ on the
stable Lévy tree $\rdtree$, see  Proposition
\ref{prop:EYf}. This 
is a key result for  the proof of
Proposition \ref{prop:moment-intro} and it is also interesting by itself (it is
in particular related to the additive coalescent and the uniform pruning
on the skeleton of the Lévy 
tree, see Remark \ref{rem:voisin} in this direction). Let $\n$ denote the
excursion measure of height process $H$ which codes the (unnormalized) stable Lévy tree $\rdtree_H$. (Notice that
$\rdtree$ under $\P$ is distributed as $\rdtree_H$ conditionally on
$\{\m(\rdtree_H)=1\}$ under $\n$.) 
\begin{proposition}
   \label{prop:EYf-intro}
	Let $\rdtree$ be the stable Lévy tree with branching mechanism
    $\psi(\lambda) = \kappa \lambda^{\gamma}$ where $\kappa >0$ and
    $\gamma \in (1,2]$.  Let $f$ be a nonnegative measurable function
    defined on $\T$. We have:
\[
\E\left[
\int_{\rdtree} f(\rdtree_y) \, \ell(\dd y)
\right] =
 \n    \left[        (1-\m(\rdtree_H))^{-1/\gamma}    
 \ind_{\{\m(\rdtree_H)<1\}}\, f(\rdtree_H)\right].
\]
\end{proposition}

The paper is organized as follows. Section \ref{Sect def} establishes
notation and defines the main objects used in this paper (discrete trees
using Neveu's formalism, real trees, Gromov-Hausdorff-Prokhorov
topology). In Section \ref{Sect measure}, we give properties of the
measure $\Psi_\tree$ and prove its continuity with respect to
$\tree$. Section \ref{Sect BGW} introduces the setting of BGW trees and
stable Lévy trees and gives a first convergence result for continuous
functions. We gather some technical results in Section \ref{Sect
  technical}. Section \ref{Sect Levy measure} is devoted to the study of
functionals of the mass and height on the stable Lévy tree and Section
\ref{phase transition section} presents the general convergence result
for functions that may blow up and describes the phase change. Appendix
\ref{append} introduces a space of measures and studies random elements
thereof; its results are used in the proofs of Proposition \ref{general
  functional} and Theorem \ref{mass+height convergence}.

\section{Definitions and notations}\label{Sect def}
\subsection{Weak convergence in a Polish space}
Let $(S,\rho)$ be a Polish metric space. We denote by $\cb(S)$
(resp. $\cb_+(S)$, resp. $\cb_b(S)$) the set of measurable
functions defined on $S$ and taking values in $[-\infty , +\infty ]$ (resp. in $[0, +\infty
]$, resp. in $\R$ and bounded) and by $\cc(S)$ (resp. $\cc_+(S)$, resp. $\cc_b(S)$) the set of continuous real-valued
functions defined on $S$ (resp. nonnegative, resp.  bounded). For $f\in \cb(S)$, we set
$\norm{f}_{\infty} = \sup_{x \in S} |f(x)| $. 
For  $f\in \cc_b(S)$, we define its Lipschitz  and bounded  Lipschitz norm:
\begin{align*}
\norm{f}_{\mathrm{L}} = \sup_{x\neq y} \frac{|f(x) - f(y)|}{\rho(x,y)} 
\quad \text{and} \quad
\norm{f}_{\mathrm{BL}}= \norm{f}_{\infty} + \norm{f}_{\mathrm{L}}. 
\end{align*}
We denote by $\mathcal{M}(S)$ the set of nonnegative finite measures on
$S$. For every $\mu \in \mathcal{M}(S)$ and  $f\in \cb_+(S)$, we write $\mu(f) = \int f(x) \, \mu(\dd
x)$. The set $\mathcal{M}(S)$ is endowed with the topology of weak
convergence which can be metrized (see \cite[Section 8.3 and Theorem
8.3.2]{bogachev2007measure}) by the bounded Lipschitz distance (also
known as the Kantorovich-Rubinstein distance): if $\mu, \nu \in \M(S)$,
set 
\[
d_{\mathrm{BL}}(\mu, \nu) = \sup\left\{ |\mu(f) - \nu(f)|, \, f \in \cc_b(S) \text{ such that } \norm{f}_{\mathrm{BL}}
  \leqslant 1\right\}.
\]
Moreover, the space $(\M(S), d_{\mathrm{BL}})$ is Polish by
\cite[Theorem 8.9.4]{bogache
v2007measure}. We also recall the total
variation norm given by 
\[
d_{\mathrm{TV}}(\mu,\nu) =  \frac{1}{2}\sup\left\{ |\mu(f) - \nu(f)|, \, f \in
  \cb(S) \text{ such that } \norm{f}_{\infty }   \leqslant 1\right\}.
\]

\subsection{Discrete trees}
We recall Neveu's formalism for rooted ordered discrete trees. Let $\mathcal{U} = \cup_{n \geqslant 0} (\mathbb{N}^*)^n$ be the set of labels with the convention $(\mathbb{N}^*)^0 = \{\emptyset\}$. If $v  = (v^1,\ldots,v^n)\in \mathcal{U}$, we denote by $H(v) = n$. By convention, we set $H(\emptyset) = 0$. If $v = (v^1,\ldots,v^n), w = (w^1,\ldots,w^m) \in \mathcal{U}$, we write $vw = (v^1,\ldots,v^n,w^1,\ldots,w^m)$ for the concatenation of $v$ and $w$. In particular, $v \emptyset = \emptyset v = v$. We say that $v$ is an ancestor of $w$ and write $v \preccurlyeq w$ if there exists $u \in \mathcal{U}$ such that $w = vu$. If $v \preccurlyeq w$ and $v \neq w$ then we shall write $v \prec w$. The mapping $\mathrm{pr} \colon \mathcal{U} \setminus \{\emptyset\} \to \mathcal{U}$ is defined by $\mathrm{pr}(v^1,\ldots,v^n) = (v^1,\ldots,v^{n-1})$ ($\mathrm{pr}(v)$ is the parent of $v$). A finite rooted ordered tree $\mathbf{t}$ is a finite subset of $\mathcal{U}$ such that
\begin{enumerate}[label=(\roman*),leftmargin=*]
	\item $\emptyset \in \mathbf{t}$,
	\item $v \in \mathbf{t} \setminus \{\emptyset\} \Rightarrow \mathrm{pr}(v) \in \mathbf{t}$,
	\item for every $v \in \mathbf{t}$, there exists a finite integer $k_v(\mathbf{t}) \geqslant 0$ such that, for every $j \in \mathbb{N}^*$, $vj \in \mathbf{t}$ if and only if $1 \leqslant j \leqslant k_u(\mathbf{t})$.
\end{enumerate}
The number $k_v(\mathbf{t})$ is interpreted as the number of children of
the vertex $v$ in $\mathbf{t}$, $H(v)$ is its generation,
$\mathrm{pr}(v)$ is its parent and more generally, the vertices $v,
\mathrm{pr}(v), \mathrm{pr}^2(v), \ldots \mathrm{pr}^{H(v)}(v)=
\emptyset$ are its ancestors. The vertex $v$ is called a leaf
(resp. internal vertex) if $k_v(\mathbf{t}) = 0$ (resp. $k_v(\mathbf{t})
>0$). The vertex $\emptyset$ is called the root of $\mathbf{t}$. We
denote the set of leaves by $\operatorname{Lf}(\mathbf{t})$ and the set
of internal vertices by $\mathbf{t}^\circ$.  If $v \in \mathbf{t}$, we
define the subtree $\mathbf{t}_v$ of $\mathbf{t}$ above $v$ as 
\[
\mathbf{t}_v = \{w \in \mathcal{U} \colon \, vw \in \mathbf{t} \}.
\]
Moreover, for every $0 \leqslant k \leqslant H(v)$, we define the 
subtree $\mathbf{t}_{k\vir v}$ of $\mathbf{t}$ above level $k$ containing $v$ as
\[
\mathbf{t}_{k\vir v} = \mathbf{t}_{\mathrm{pr}^{H(v)-k}(v)}
\]
where $\mathrm{pr}^{H(v)-k}(v)$ is the unique ancestor of $v$ with
height $k$, with the convention that $\mathrm{pr}^{0}(v)=v$. We denote by $|\mathbf{t}| = \text{Card} (\mathbf{t})$ the number of vertices of $\mathbf{t}$ and by $\H(\mathbf{t}) = \sup_{v \in \mathbf{t}}H(v)$ the height of $\mathbf{t}$.


\subsection{Real trees}\label{real trees}
We recall the formalism of real trees, see \cite{evans2007probability}. A metric space $(\tree,d)$ is a real tree if the following two properties hold for every $x,y \in \tree$.
\begin{enumerate}[label=(\roman*),leftmargin=*]
	\item (Unique geodesics). There exists a unique isometric map $f_{x,y} \colon [0,d(x,y)] \to \tree$ such that $f_{x,y}(0) = x$ and $f_{x,y}(d(x,y)) = y$.
	\item (Loop-free). If $\phi$ is a continuous injective map from
      $[0,1]$ into $\tree$ such that $\phi(0) = x$ and $\phi(1) = y$,
      then we have $\phi([0,1]) = f_{x,y}\left([0,d(x,y)]\right)$.
\end{enumerate}
For a rooted real tree $(\tree,\root,d)$, that is a real tree with a
distinguished vertex $\root \in \tree$ called the root, we define the set of leaves by
\[
\operatorname{Lf}(\tree) = \left\lbrace x \in \tree \setminus
  \{\root\} \colon \, \tree \setminus \{x\} \text{ is connected}
\right\rbrace,
\]
with the convention that $\operatorname{Lf}(\tree) =\{\root\}$ if 
 $\tree=\{\root\}$. 
A weighted rooted real tree $(\tree, \root, d , \mu)$ is a rooted real tree
$(\tree,\root, d)$  equipped with a nonnegative finite measure $\mu$. 
In what follows, real trees will always be weighted and rooted and we
will simply call them  real trees.

Let us
consider a real tree $(\tree,\root,d,\mu)$.   The total
mass  of the  tree $\tree$  is 
defined   by    $\m(\tree)   =    \mu(\tree)$   and   its    height   by
$\H(\tree)   =   \sup_{x   \in    \tree}H(x)   \in   [0,\infty]$,   with
$H(x) =  d(\root, x)$  the height  of $x$. Note  that if  $(\tree,d)$ is
compact, then $\H(\tree) < \infty$.
The range of
the mapping $f_{x,y}$ described in (i) above is denoted by $\llbracket
x,y\rrbracket$ (this is the line segment between $x$ and $y$ in the
tree). We also write $\llbracket
x,y\llbracket=\llbracket
x,y\rrbracket\setminus \{y\}$.
In particular, $\llbracket\root,x\rrbracket$ is the path going
from the root to $x$ which we will interpret as the ancestral line of
vertex $x$. We define a partial order on the tree by setting $x
\preccurlyeq y$ ($x$ is an ancestor of $y$) if and only if $x \in
\llbracket\root,y\rrbracket$. If $x,y \in \tree$, there is a unique $z
\in \tree$ such that $\llbracket\root,x\rrbracket \cap \llbracket\root,
y\rrbracket = \llbracket\root,z \rrbracket$. We write $z = x \wedge y$
and call it the most recent common ancestor of $x$ and $y$. Let $x \in
\tree$ be a vertex.  Let
$r\in [0, H(x)]$. We denote by $x_r \in \tree$ be the unique ancestor of
$x$ with height $H(x_r) = r$. As in the discrete case, we also define the subtree $\tree_x$ of $\tree$ above $x$ as
\[
\tree_x = \left\{y \in \tree \colon \, x \preccurlyeq y\right\},
\]
and the subtree $\tree_{r\vir x}=\tree_{x_r}$ of $\tree$ above level $r$ containing $x$ as
\[
\tree_{r\vir x} = \left\{y \in \tree \colon \, H(x \wedge y) \geqslant
  r\right\}=T_{x_r}.
\]
Then  $\tree_x$  (resp. $\tree_{r\vir x}$)  can  be  naturally viewed  as  a
real tree, rooted at  $x$ (resp. at $x_r$)  and endowed
with       the       distance       $d$      and       the       measure
$\mu_{|\tree_x}  =   \mu(\cdot  \cap   \tree_x)$  (resp.    the  measure
$\mu_{|   \tree_{r\vir x}}$).   Note   that   $\tree_{0\vir x}   =  \tree$   and
$\tree_{H(x) \vir x}  = \tree_x$. 

\begin{remark}
   \label{rem:coding}
We recall the construction of a real tree from an excursion path, see
\emph{e.g.} \cite[Example 3.14]{evans2007probability} or \cite[Section
2.1]{duquesne2005probabilistic}. Let $e$ be a positive excursion path,
that is $e\in \cc_+(\R_+)$  such that $e(0) = 0$,
$e(s) >0$ for $0<s<\sigma$ and $e(s) = 0$ for $s\geqslant \sigma$ where
$\sigma \coloneqq \inf\{s>0\colon \, e(s) = 0\} \in (0,\infty)$ is the
duration of the excursion. Set $d_e(t,s) = e(t) + e(s) - 2\inf_{[t\wedge
  s, t\vee s]}e$ for every $t,s \in [0,\sigma]$ and define an
equivalence relation on $[0,\sigma]$ by letting $t\sim_e s$ if and only
if $d_e(t,s) = 0$.  The real tree $\tree_e$ coded by $e$
is defined as the quotient space $[0,\sigma]/\sim_e$ rooted at $p(0)$
where $p \colon [0,\sigma]  \to \tree_e$ is the quotient map and
equipped with the distance $d_e$ and the pushforward measure $\lambda
\circ p^{-1}$ where $\lambda$ is the Lebesgue measure on
$[0,\sigma]$. This defines a compact weighted rooted real tree. Notice
that the mass and height of $\tree_e$ are given by $\m(\tree_e)=\sigma$
and $\H(\tree_e)=\norm{e}_\infty $. 
\end{remark}

We will need to view discrete trees as real trees. Let $\mathbf{t}$ be a
finite rooted ordered tree and let $a > 0$. Suppose that $\mathbf{t}$ is
embedded into the plane such that the edges are straight lines with
length $a$ that only intersect at their incident vertices. Denote by
$\pi_{\mathbf{t},a} \colon \mathbf{t} \to \real^2$ the embedding and by
$a\mathbf{t} = \pi_{\mathbf{t},a}(\mathbf{t}) \subset \real^2$ the
embedded set. Moreover, for a vertex $v \in \mathbf{t}$, we denote by
$av = \pi_{\mathbf{t},a}(v)$ the corresponding vertex in
$a\mathbf{t}$. Then $a\mathbf{t}$ can be considered as a compact real
tree $(a\mathbf{t},d_{\mathbf{t}},\mu_{\mathbf{t}})$: the distance
$d_{\mathbf{t}}(x,y)$ between two points $x,y \in a\mathbf{t}$ is
defined as the shortest length of a curve that connects $x$ and $y$, and
the measure $\mu_{\mathbf{t}}$ is the pushforward of the uniform
probability measure on $\mathbf{t}$ by the embedding
$\pi_{\mathbf{t},a}$. In other words, $a\mathbf{t}$ is obtained from
$\mathbf{t}$ by connecting every vertex to its children in such a way
that all edges have length $a$ and is equipped with the measure
$\mu_{\mathbf{t}}$ supported on the set $\{av \colon \, v\in
\mathbf{t}\}$ and satisfying $\mu_{\mathbf{t}}(\{av\}) =
1/|\mathbf{t}|$ for every $v \in \mathbf{t}$. The tree $a\mathbf{t}$ is
naturally rooted at $a\emptyset$ (also denoted $\emptyset$). Notice that vertices of the form
$av$ with $v \in \mathbf{t}$ are precisely those vertices in
$a\mathbf{t}$ whose height is an integer multiple of $a$. Finally, to
simplify notation, for
every $v \in \mathbf{t}$, we will write $a\mathbf{t}_v$ instead of
$(a\mathbf{t})_{av}$ for the subtree of
$a\mathbf{t}$ above $av$. We stress that, unless $v=\emptyset$, the  measure
of the compact real tree
 $a\mathbf{t}_v$ has mass less than one, whereas the  measure
of the compact real tree
$a(\mathbf{t}_v)$ is by definition a probability measure. 

\subsection{Gromov-Hausdorff-Prokhorov topology}\label{GHP section}
Denote by $\T$ the set of measure-preserving and root-preserving
isometry classes of compact real trees. We will often
identify a class with  an element of this class. So we shall write that
$(\tree, \root, d, \mu) 
\in \T$ if  $(\tree,\root,d)$  is a rooted compact real tree and  $\mu$  is
a nonnegative finite measure on $ \tree$. 
 When there is no ambiguity,
we may write $\tree$ for $(\tree, \root, d, \mu) $. 

We start by giving the standard definition of the
Gromov-Hausdorff-Prokhorov distance. Let $(E,\delta)$ be a metric
space. Given a non-empty subset $A \subset E$ and $\epsilon >0$, the
$\epsilon$-neighborhood of $A$ is $A^\epsilon = \{x \in E\colon \,
d(x,A) < \epsilon\}$. The Hausdorff distance $\delta_{\mathrm{H}}$ between
two non-empty subsets $A,B \subset E$ is defined by
\[
\delta_{\mathrm{H}}(A,B) = \inf \{\epsilon >0 \colon A \subset B^\epsilon
\text{ and } B \subset A^\epsilon\}.
\]
Next, denoting by $\mathcal{B}(E)$ the Borel $\sigma$-field on $(E,\delta)$, the Lévy-Prokhorov distance between two finite nonnegative measures $\mu, \nu$ on $(E,\mathcal{B}(E))$ is
\[\delta_{\mathrm{P}}(\mu,\nu) = \inf\{\epsilon>0 \colon \, \mu(A)\leqslant \nu(A^\epsilon) + \epsilon \text{ and } \nu(A)\leqslant \mu(A^\epsilon) + \epsilon, \, \forall A \in \mathcal{B}(E)\}.\]
We can now give the standard distance used to define the Gromov-Hausdorff-Prokhorov topology. For two compact real trees $(\tree,\root,d,\mu),(\tree',\root',d',\mu') \in \T$, set
\begin{equation}
\ghp^\circ(\tree,\tree') = \inf \left\lbrace \delta(\phi(\root),\phi'(\root')) \vee \delta_{\mathrm{H}}(\phi(\tree), \phi'(\tree')) \vee \delta_{\mathrm{P}}(\mu\circ \phi^{-1}, {\mu'}\circ{\phi'}^{-1}) \right\rbrace,
\end{equation}
where the infimum is taken over all isometries $\phi \colon \tree \to E$ and $\phi'\colon \tree' \to E$ into a common metric space $(E,\delta)$. This defines a metric which induces the Gromov-Hausdorff-Prokhorov topology on $\T$.

It will be convenient for our purposes to define another metric which
induces the same topology on $\T$. Let
$(\tree,\root,d,\mu),(\tree',\root',d',\mu') \in \T$. 
Recall that a correspondence between $\tree$ and $\tree'$ is a subset $\mathcal{R} \subset \tree \times \tree'$ such that for every $x \in \tree$, there exists $x' \in \tree'$ such that $(x,x') \in \mathcal{R}$, and conversely, for every $x'\in \tree'$, there exists $x \in \tree$ such that $(x,x') \in \mathcal{R}$. In other words, if we denote by $p \colon \tree \times \tree'\to \tree$ (resp. $p'\colon \tree \times \tree' \to \tree'$) the canonical projection on $\tree$ (resp. on $\tree'$), a correspondence is a subset $\mathcal{R} \subset \tree \times \tree'$ such that $p(\mathcal{R}) = \tree$ and $p'(\mathcal{R}) = \tree'$. If $\mathcal{R}$ is a correspondence between $\tree$ and $\tree'$, its distortion is defined by
\[\operatorname{dis}(\mathcal{R}) = \sup \left\{\left|d(x,y) - d'(x',y')\right| \colon \, (x,x'), (y,y') \in \mathcal{R}\right\}.\]
Next, for any nonnegative finite measure $m$ on $\tree \times \tree'$, we define its discrepancy with respect to $\mu$ and $\mu'$ by
\[
\operatorname{D}(m; \mu, \mu') = d_{\mathrm{TV}}(m\circ p^{-1}
  ,\mu) + d_{\mathrm{TV}}(m \circ {p'}^{-1},\mu').
\]
Then the
Gromov-Hausdorff-Prokhorov distance between $\tree$ and $\tree'$ is
defined as 
\begin{equation}\ghp(\tree,\tree') = \inf\left\{\frac{1}{2}\operatorname{dis}(\mathcal{R})\vee \operatorname{D}(m; \mu,\mu') \vee m(\mathcal{R}^c)\right\}, \label{ghp}
\end{equation}
where the infimum is taken over all correspondences $\mathcal{R}$
between $\tree$ and $\tree'$ such that $(\root, \root') \in \mathcal{R}$
and all nonnegative finite measures $m$ on $\tree \times \tree'$. It can
be verified that $\ghp$ is indeed a distance on $\T$ which is equivalent
to $\ghp^\circ$ and that the space $(\T,\ghp)$ is a Polish metric space,
see \cite{addario2017scaling}.  

We gather some facts about the Gromov-Hausdorff-Prokhorov distance that will be useful later. We refer the reader to \cite{addario2017scaling} or \cite{ross2018scaling}. We have that
\begin{equation}\label{ghp inequality}
\frac{1}{2} \left|\H(\tree) - \H(\tree')\right| \vee \left|\m(\tree) -
  \m(\tree')\right| \leqslant \ghp(\tree, \tree') \leqslant \left(
  \H(\tree) +\H(\tree')\right) \vee \left(\m(\tree) +
  \m(\tree')\right). 
\end{equation}
When $\tree' = \{\emptyset\}$ is the trivial tree consisting only of the root with mass $0$, we have
\begin{equation}\label{ghp trivial tree}
\frac{1}{2}\H(\tree)\vee \m(\tree) \leqslant \ghp(\tree, \{\emptyset\}) \leqslant \H(\tree)\vee \m(\tree).
\end{equation}
We consider the subset of $\T$ of trees with either height or mass equal
to $0$:
\begin{equation}\TT = \left\{\tree \in \T\colon\, \m(\tree) = 0\text{ or }  \H(\tree) = 0 \right\}.\label{T_0}
\end{equation}
Note that $\TT\subset \T$ is a closed subset since the mappings $\m
\colon \T \to \real$ and $\H \colon \T \to \real$ are continuous with
respect to the Gromov-Hausdorff-Prokhorov topology, thanks to \eqref{ghp
  inequality}. We now give bounds 
for the distance of a tree $T$ to $\TT$ which are similar to \eqref{ghp
  trivial tree}. 
\begin{lemma}\label{distance to F}
	Let $\tree \in \T$. Then we have
	\begin{equation}\frac{1}{2}\H(\tree)\wedge \m(\tree)\leqslant \ghp(\tree,\TT)\leqslant \H(\tree)\wedge \m(\tree).
	\end{equation}
\end{lemma}
\begin{proof}
Let $(\tree, d, \root, \mu)\in \T$ and  $\delta>
\ghp(\tree,\TT)$. Then there exists $\tree' \in \TT$ such that
$\ghp(\tree,\tree') \leqslant \delta$. By \eqref{ghp inequality}, we get 
\[
\frac{1}{2} \left|\H(\tree) - \H(\tree')\right| \vee
    \left|\m(\tree) - \m(\tree')\right|\leqslant \delta.
\]
	But since $\tree' \in \TT$, either $\H(\tree') = 0$ or $\m(\tree') =
    0$. Therefore, either $\H(\tree)\leqslant 2\delta$ or $\m(\tree)
    \leqslant \delta$. Since $\delta >  \ghp(\tree,\TT)$ is arbitrary,
    this yields the lower bound. 
	
	To prove the upper bound, let $\tree' = \tree$ endowed with the zero measure $\mu'= 0$, and take $\mathcal{R} = \{(x,x) \colon \, x \in \tree\}$ and $m$ the zero measure on $\tree \times \tree'$. Then $\operatorname{dis}(\mathcal{R}) = 0$, $m(\mathcal{R}^c) = 0$ and $\operatorname{D}(m; \mu, \mu') = \mu(\tree) = \m(\tree)$. It follows that $\ghp(\tree, \tree') \leqslant \m(\tree)$. Note that $\tree' \in \TT$, therefore
	\[  \ghp(\tree,\TT)\leqslant \ghp(\tree,\tree')\leqslant \m(\tree).\]
	
	Next, let $\tree'' = \{\emptyset\}$ be the trivial tree consisting
    only of the root with mass $\m(\tree)$, \emph{i.e.} endowed with the
    measure $\mu'' = \m(\tree) \delta_{\emptyset}$. Take $\mathcal{R} =
    \tree \times \{\emptyset\}$ and $m(A \times B) =
    \mu(A)\delta_\emptyset(B)$. Then, we have  $\mathcal{R}^c = \emptyset$, so
    $m(\mathcal{R}^c) = 0$. Moreover, we have
\[
\operatorname{dis}(\mathcal{R}) = \sup\left\{|d(x,y)|\colon \, x,y \in
  \tree\right\} \leqslant 2\H(\tree).
\]
	 Since $m\circ p^{-1} = \mu$ and $m \circ {p''}^{-1} = \m(\tree)\delta_\emptyset = \mu''$, we get $\operatorname{D}(m,\mu,\mu'') = 0$. It follows that $\ghp(\tree,\tree'') \leqslant \H(\tree)$. Since $\tree'' \in \TT$, we deduce that
	\[ \ghp(\tree,\TT) \leqslant \ghp(\tree,\tree'')\leqslant \H(\tree).\]
	This finishes the proof of the upper bound.
\end{proof}
\section{A finite measure indexed by a tree}\label{Sect measure}
Let $(\tree,\emptyset,d, \mu)$ be a compact real tree. Let $x\in \tree$
and $r\in [0, H(x)]$, where $H(x)=d(\root, x)$. Recall that
$\tree_{r\vir x} = \{y \in \tree : \, H(x \wedge y) \geqslant r\}$ is the
subtree containing $x$ and starting at height $r$, endowed with the
distance $d$ and the measure $ \mu_{| \tree_{r\vir x}}$. It is
straightforward to check that $\tree_{r\vir x}$ is a compact real
tree and thus belongs to $\T$. Define a nonnegative measure $\Psi_\tree$
on $\T \times \R_+$ by, 
for every $f\in \cb_+(\T\times \R_+)$, 
\begin{equation}
\label{definition pi}
 \Psi_{\tree}(f) = \int_{\tree} \mu(\dd x)\int_{0}^{H(x)}f\left(\tree_{r\vir x},r\right)\dd r.
\end{equation}
As we will consider functions depending only on the mass and height of
the subtrees, we introduce the measure $\psij_\tree$ on $\R_+^2$ defined
by, for every  $f\in \cb_+(\R_+^2)$,
\begin{equation}
\label{psi joint}
\psij_\tree(f) = \int_\tree \mu(\dd x) \int_0^{H(x)} f\left(\m(\tree_{r\vir x}), \H(\tree_{r\vir x})\right) \, \dd r.
\end{equation}

\begin{lemma}\label{psi well defined}
	Let $\tree$ be a compact real tree. The mapping $(r,x) \mapsto
    \tree_{r\vir x}$ from $\{(r,x) \in \R_+\times \tree \colon\, r \leqslant
    H(x)\}$ to $\T$ is measurable with respect to the Borel
    $\sigma$-fields. Furthermore, the measure $\Psi_\tree$ is finite and
    does not depend on the choice of representative in the equivalence
    class in $\T$  of $\tree$.
\end{lemma}

\begin{proof}
  Let  $(\tree,\root,   d,\mu)$  be   a  compact   real  tree   and  set
  $A   \coloneqq  \{(r,x)   \in   \R_+\times  \tree\colon\,   r\leqslant
  H(x)\}$.
  For every $(r,x) \in A$, recall that $x_r \in \tree$ is the unique ancestor
  of $x$ with height $H(x_r) = r$.  We start by showing that the mapping
  $(r,x)  \mapsto   x_r$  is  continuous   from  $A$  to   $\tree$.  Let
  $(r,x), (s,y) \in A$. Without loss of generality, we can assume that
  $r\geq s$. If $H(x \wedge y) \geq s$, then we have $y_s \preccurlyeq
  x$ and thus $y_s \preccurlyeq
  x_r$. This implies that $d(x_r, y_s)= r-s$. If
  $H(x\wedge y)  < s$, then  we have  $x_r\in \llbracket x\wedge  y,x\rrbracket$
 and $y_s\in \llbracket x\wedge  y,y\rrbracket$. This implies that $x_r$
 and $y_s$ belong to  $\llbracket x, y\rrbracket$, and thus
 $d(x_r,y_s) \leqslant d(x,y)$. In 
  all cases, we have
\begin{equation*}
d(x_r,y_s) \leqslant d(x,y) + |r-s|.
\end{equation*}
This proves  that $(r,x) \mapsto x_r$ is continuous.

The mapping $y \mapsto \tree_y$ from $\tree$ to $\T$ is continuous from
below, in the sense that for $y\in \tree$
\begin{equation}\label{continuity from below}
\lim_{\substack{z\to y\\ z\preccurlyeq y}} 
\ghp(\tree_z,  \tree_y)=0.
\end{equation}
To see this, let $\delta >0$,  $y \in \tree$ and $(y_n, \, n \in
\mathbb{N})$ be a sequence in $\tree$ converging to $y$ such that $y_n
\preccurlyeq y$ for every $n \in \mathbb{N}$. Notice that since $\tree$
is compact, it holds that there is a finite number of subtrees with
height larger than $\delta$ attached to the branch $\llbracket
\root,y\rrbracket$. Thus, there are no subtrees with height larger than
$\delta$ attached to $\llbracket y_n, y\llbracket$ for $n$ larger than some $n_0$. Moreover, since $\tree_y = \bigcap_{n\in \mathbb{N}}
\tree_{y_n}$, we get that $\lim_{n\to \infty} \mu(\tree_{y_n}) =
\mu(\tree_y)$ implying that the mass of the subtrees attached to
$\llbracket y_n,y\llbracket$ goes to $0$ as $n$ goes to infinity. 

Define a correspondence between $\tree_{y_n}$ and $\tree_y$ by
\begin{equation*}
\mathcal{R} \coloneqq \left\{(z,z) \colon\, z \in \tree_y\right\} \bigcup \left\{(z,y) \colon \, z \in \tree_{y_n}\setminus \tree_y\right\}.
\end{equation*}
It is straightforward to check that $\operatorname{dis}(\mathcal{R})
\leqslant 2 (\delta + d(y_n,y))$ for $n\ge n_0$.  Consider the measure on
$\tree_{y_n}\times \tree_y$ defined by $m(\dd x , \, \dd z)= \mu_{| \tree_y} ( \dd z) \delta
_z (\dd x) =\mu_{| \tree_y} ( \dd x) \delta
_x (\dd z) $. 
 Then we have $\operatorname{D}(m;\mu_{|\tree_{y_n}}, \mu_{|\tree_y})
 \leqslant \mu(\tree_{y_n}) - \mu(\tree_y)$ and $m(\mathcal{R}^c) =
 0$. It follows from \eqref{ghp} that 
\begin{equation*}
\limsup_{n\to \infty} \ghp(\tree_{y_n},\tree_y) \leqslant \limsup_{n\to \infty} \left( \delta + d(y_n,y) + \mu(\tree_{y_n}) - \mu(\tree_y) \right)= \delta.
\end{equation*}
Since $\delta >0$ is arbitrary, \eqref{continuity from below} readily follows.

Now it is not difficult to see that the continuity from below \eqref{continuity from below} of the mapping $y \mapsto \tree_y$ implies its measurability. By composition, it follows that the mapping $(r,x) \mapsto \tree_{r\vir x}=T_{x_r}$ from $A$ to $\T$ is measurable.

Next, notice that $\Psi_{\tree}$ is finite since
\[ \Psi_{\tree}(1) = \int_{\tree} H(x)\mu(\dd x) \leqslant \H(\tree) \m(\tree) < \infty. \]

Finally, let $f\in \cb_+(\T \times \R_+ )$ and 
$(\tree,\root,d,\mu),(\tree',\root',d',\mu')$ be two compact real trees
such that there is a measure-preserving and root-preserving isometry
$\phi \colon \tree \to \tree'$. This means that $\phi$ is an isometry
satisfying $\mu' = \mu \circ \phi^{-1}$ and $\phi(\root) = \root'$. 
Moreover, for every $x, y \in \tree$, since $H(x\wedge y) = 2^{-1}
\left(d(\root, x) + d(\root, y) - d(x,y)\right)$, we deduce that
\[
H(x\wedge y) = H(\phi(x)\wedge \phi(y)).
\]
Using this and the definitions of $\tree_{r\vir x}$ and
$\tree'_{r\vir\phi(x)}$, it is easy to see that, for every $x \in \tree$
and  $r\in [0,  H(x)]$, $\phi$ induces a
measure-preserving and root-preserving isometry from $\tree_{r\vir x}$
to $\tree'_{r\vir \phi(x)}$ and therefore $f(\tree_{r\vir x},r) =
f(\tree'_{r \vir \phi(x)},r)$. Since $H(x) = H(\phi(x))$, it follows that 
\begin{align*}
\Psi_\tree(f) 
&= \int_\tree \mu(\dd x ) \int_0^{H(x)} f(\tree_{r\vir x},r)\, \dd r \\
&= \int_\tree \mu(\dd x ) \int_0^{H(\phi(x))} f(\tree'_{r\vir \phi(x)},r)\, \dd r \\
&= \int_{\tree'} \mu\circ \phi^{-1}(\dd y) \int_0^{H(y)} f(\tree'_{r\vir
  y},r)\, \dd r\\
&= \Psi_{\tree'}(f).
\end{align*}
This proves that $\Psi_\tree$ does not depend on the choice of
representative in the equivalence class of $\tree$ which completes the  proof.
\end{proof}

Recall that $\operatorname{Lf}(\tree) $ is the set of leaves of $\tree$. It
is well known that there exists a unique $\sigma$-finite measure $\ell$
on $(\tree, \mathcal{B}(\tree))$, called the length measure, such that
$\ell(\operatorname{Lf}(\tree)) = 0$ and $\ell(\llbracket x,y
\rrbracket) = d(x,y)$, see \emph{e.g.} \cite[Chapter 4,
§4.3.5]{evans2007probability}.  
The next result gives an alternative expression for $\Psi_\tree$ in terms of the length measure.
\begin{proposition}
Let $(\tree,\root,d,\mu) $ be a compact real tree. For every
$f\in \cb_+(\T \times \R_+)$, we have
	\begin{equation}\label{length measure}
	\Psi_\tree(f)  = \int_\tree \mu(\tree_y)f(\tree_y,H(y))  \, \ell(\dd y).
	\end{equation}
\end{proposition}
\begin{proof}Let $(\tree,\root,d,\mu)$ be a compact real tree and $f \in
  \cb_+ (\T \times \R_+ )$. Notice that $\{(x,y) \in \tree^2\colon\, y
  \preccurlyeq x\} = \{(x,y)\in \tree^2\colon\, d(\root, x)=d(\root, y)+d(x,y)\}$ is
  closed in $\tree^2$ and thus measurable. Moreover, the mapping $y
  \mapsto \tree_y$ is measurable from $\tree$ to $\T$ by the proof of
  Lemma \ref{psi well defined}. Thus the mapping $(x,y) \mapsto  \ind_{\{y
  \preccurlyeq x\}} f( \tree_y, H(y))$ is measurable. By Fubini's theorem, it follows that 
	\begin{align*}
	\int_\tree \mu(\tree_y) f\left(\tree_y, H(y)\right) \, \ell(\dd y) &= \int_\tree \mu(\dd x) \int_\tree \ind_{\{y \preccurlyeq x\}} f\left(\tree_y, H(y)\right)\, \ell(\dd y) \\
	&= \int_\tree \mu(\dd x) \int_{\llbracket \root,x\rrbracket} f\left(\tree_y, H(y)\right) \, \ell(\dd y).
	\end{align*}
	Let $x \in \tree$ and let $f_{\root,x} \colon [0,H(x)] \to \llbracket \root,x\rrbracket$ be the unique isometry such that $f_{\root,x}(0) = \root$ and $f_{\root,x}(H(x)) = x$. Using that $\ell_{|\llbracket \root,x\rrbracket} = \lambda \circ f_{\root,x}^{-1}$ where $\lambda$ is the Lebesgue measure on $[0,H(x)]$, we get that
	\begin{equation*}
	\int_{\llbracket \root,x\rrbracket} f\left(\tree_y, H(y)\right) \, \ell(\dd y) = \int_0^{H(x)} f\left(\tree_{f_{\root,x}(r)}, H(f_{\root,x}(r))\right) \, \dd r.
	\end{equation*}
	Since $f_{\root,x}$ is an isometry, for every $r \in [0,H(x)]$,
    $f_{\root,x}(r)$ is the unique ancestor of $x$ at height $r$, that
    is $x_r$,  and  $H(f_{\root, x}(r))=r$. As
    $\tree_{f_{\root,x}(r)}=\tree_{x_r}=\tree_{r\vir x}$  for every $r
    \in [0,H(x)]$, it follows that 
	\begin{equation*}
	\int_\tree \mu(\tree_y) f\left(\tree_y, H(y)\right) \, \ell(\dd y) =
    \int_\tree \mu(\dd x) \int_0^{H(x)} f\left(\tree_{r\vir x}, r \right) \,
    \dd r. 
	\end{equation*}
This  concludes the proof.
	\end{proof}

The main result of this section concerns the continuity of the mapping $\Psi \colon \tree \mapsto \Psi_\tree$. 
\begin{proposition}
\label{continuity of the measure}
	The mapping $\Psi \colon \tree \mapsto \Psi_\tree$, from $\T$
    endowed with the Gromov-Hausdorff-Prokhorov topology to
    $\mathcal{M}(\T\times \R_+)$ endowed with the topology of weak
    convergence, is well defined and continuous. 
\end{proposition}

The end of this section is devoted to the proof of Proposition
\ref{continuity of the measure}. For $\tree$ a compact real tree, $x\in \tree$, $s\in
[0, +\infty ]$, $r\in 
[0, s \wedge H(x)]$, we define the
following set of elements of 
$T$ such that their common ancestor with $x$ has height in $[r, s]$
\[
\tree_{[r,s]\vir x}=
\{y \in \tree\colon \, H(y \wedge x) \in [r,s]\}.
\]
Recall that $x_r$ is  the ancestor of $x$ at height  $r$ in $\tree$, and
is  also seen  as  the root  of  the tree  $\tree_{r\vir x}$.  We shall  see
$\tree_{[r,s]\vir x}$  as   a compact  real   tree  rooted   at  $x_r$   with  measure
$\mu_{|\tree_{[r,s]\vir x}}   =   \mu(\cdot   \cap   \tree_{[r,s]\vir x})$
and thus $\tree_{[r,s]\vir x}\in \T$.    Recall that
$\m(\tree_{[r,s]\vir x})=\mu(\tree_{[r,s]\vir x})$    denotes     its    mass    and
$\H(\tree_{[r,s]\vir x})=\sup\{          H(y)           \colon\,          y\in
\tree_{[r,s]\vir x}\subset{\tree}\}-r                                        $
its        height.        Notice       in        particular       that
$\tree_{[r, +\infty] \vir x}=\tree_{r\vir x}$ for $r\in [0, H(x)]$.

We first establish an estimate for the Gromov-Hausdorff-Prokhorov distance between subtrees of two real trees in terms of the distance between the trees themselves.

\begin{lemma}\label{continuity lemma}
	Let $\tree, \tree'$ be compact real trees and let $\delta > \ghp(\tree,\tree
    ')$. Let $\mathcal{R}$ be a correspondence between $\tree$ and
    $\tree'$ such that $(\root,\root')\in \mathcal{R}$ and let $m$ be a
    measure on $\tree \times \tree'$ such that 
\[ 
\frac 1 2 \operatorname{dis} (\mathcal{R})\vee \operatorname{D}(m;
\mu,\mu') \vee m(\mathcal{R}^c)  \leqslant \delta. 
\]
	Then for every $(x,x')$ in $\mathcal{R}$ and every $r \geqslant 0$
    such that $6\delta \leqslant r \leqslant H(x) \wedge H(x')$, we have 
\begin{equation}
 \label{continuity lemma t}
	\ghp(\tree_{r\vir x},\tree'_{r\vir x'} )\leqslant 8 \delta +
    2\m\left(\tree_{[r-6\delta, r+6\delta]\vir x}  \right)
+ 2\H(\tree_{[r-3\delta, r+6\delta]\vir x}).
\end{equation}
\end{lemma}
\begin{proof}
  Similarly to $x_r$, we denote  by $x'_r$ the ancestor of $x'$  at
  height $r$ in $\tree'$, which is also seen as the root of $\tree'_{r, x'}$.
	 We shall bound $\ghp(\tree_{r\vir x}, \tree'_{r\vir x'})$ from above by
\[
\frac{1}{2}\operatorname{dis}(\widetilde{\mathcal{R}}) \vee
    \operatorname{D}(\widetilde{m}; \widetilde{\mu}, \widetilde{\mu}')
    \vee \widetilde{m}(\widetilde{\mathcal{R}}^c)
\]
	where $\widetilde{\mathcal{R}}$ is a well chosen correspondence
    between $\tree_{r\vir x}$ and $\tree'_{r \vir x'}$ and
    $\widetilde{m}$ (resp. $\widetilde{\mu}$, $\widetilde{\mu}'$) is the
    restriction of the measure $m$ (resp. $\mu$, $\mu'$) to
    $\tree_{r\vir x}\times \tree'_{r\vir x'}$ (resp. $\tree_{r\vir x}$,
    $\tree'_{r\vir x'}$). We begin by noticing that, for every
    $(t,t'),(s,s') \in \mathcal{R}$, we have 
	\begin{equation}\left| d(t,s) - d'(t',s') \right| \leqslant \operatorname{dis}(\mathcal{R}) \leqslant 2 \delta. \label{distances}
	\end{equation}
	In particular, taking $(s,s') = (\root,\root')\in \mathcal{R}$ yields
	\begin{equation}
\label{distances 2}
\left|H(t) - H(t') \right| \leqslant 2 \delta.
	\end{equation}
	Using this, we get that for $(t,t') \in \mathcal{R}$
	\begin{align}H(t'\wedge x') &= \frac 1 2 \left( H(t') + H(x') - d'(t',x')\right) \notag\\ 
	&\geqslant \frac 1 2 \left(H(t) - 2 \delta + H(x) - 2 \delta - d(t,x) - 2  \delta \right) \notag \\ 
	&= H(t \wedge x) - 3 \delta.  \label{heights}
	\end{align}
	

\noindent
\textbf{Step 1:} we construct a correspondence  between $\tree_{r\vir x}$ and
$\tree'_{r\vir x'}$ and give an upper bound of its distortion. 
    Let        $(t,t')\in \mathcal{R}$. Assume  that $H(t \wedge x)
    \geqslant r +  3 \delta$. Then, we get that $t        \in
    \tree_{r\vir x}$ and  that 
    $H(t'\wedge x') \geqslant r$ by  \eqref{heights}, that is $t' \in
    \tree'_{r\vir x'}$. This gives that  $(t,t')\in  \tree_{r\vir x}\times
    \tree'_{r\vir x'}$. Similarly, if $H(t' \wedge x')
    \geqslant r +  3 \delta$, we get $(t,t')\in  \tree_{r\vir x}\times
    \tree'_{r\vir x'}$. Therefore, the following set
\[
\widetilde{\mathcal{R}}=
\{(t,t')\in \mathcal{R} \colon \max (H(t \wedge x), H(t' \wedge x'))\geq
r+ 3\delta\}
\bigcup  \left(\tree_{[r, r+3\delta]\vir x}\times \{x'_r\}\right)
\bigcup  \left(\{x_r\} \times \tree'_{[r, r+3\delta]\vir x'}\right)
\]
 is a correspondence
    between $\tree_{r\vir x}$ and $\tree'_{r\vir x'}$.
 We give a bound of its  
    distortion. 
Let $(t,t'), (s,s') \in
    \widetilde{\mathcal{R}}$.  \\
	\textbf{Case 1:} Assume that $(t,t') \in \mathcal{R}$ and $(s,s')
    \in \mathcal{R}$, then by \eqref{distances} we have 
	\[ \left|d(t,s) - d'(t',s')\right| \leqslant 2 \delta.\]
	\textbf{Case 2:} Assume that $(t,t') \in \mathcal{R}$ and $(s,s')
    \notin \mathcal{R}$. Without loss of generality, we may assume that
    $s=x_r$ and thus $H(s'\wedge x') \in [r, r+3\delta]$. Let $y' \in \tree'$ such that $(x_r,y') \in \mathcal{R}$, then using \eqref{distances} and the triangle inequality, we get
	\begin{align*}
	\left|d(t,s) - d'(t',s')\right| &\leqslant \left| d(t,x_r) - d'(t',y')\right| + \left|d'(t',y') - d'(t',s') \right| \\
	&\leqslant 2 \delta +  d'(y',s') \\
	&\leqslant 2 \delta + d'(y',x'_r) + d'(x'_r,s').
	\end{align*}
	Notice that by \eqref{heights}, we have $H(y'\wedge x') \geqslant
    H(x_r \wedge x) - 3 \delta = r-3\delta$, so either $H(y' \wedge x')
    \geqslant r$ or $H(y'\wedge x') \in [r-3\delta,r)$. In the first
    case, $x'_r$ is necessarily an ancestor of $y'$ and we have
    $H(y'\wedge x'_r) = r$. In the second case, we have $y'\wedge x' =
    y'\wedge x'_r$ and $H(y'\wedge x'_r) \geqslant r-3\delta$. Thus, in
    all cases we have $H(y'\wedge x'_r) \geqslant r-3\delta$ and then 
\[
d'(y',x'_r) = H(y') + H(x'_r) - 2H(y'\wedge x'_r) \leqslant H(x_r) + 2
\delta + r - 2(r-3\delta) = 8 \delta.
\]
	On the other hand, since we assumed that $H(s'\wedge x') \in [r, r+3\delta]$, we get  that $x'_r$ is an ancestor of $s'$ and   $s'\in
    \tree'_{[r, r+3\delta]\vir x'}$. We deduce that 
	\begin{equation}\label{ancestor}
	d'(x'_r,s') = H(s') - H(x'_r) = H(s') - r
\leq  \H(\tree'_{[r,
  r+3\delta]\vir x'}).
	\end{equation}
	It follows that
\[	\left|d(t,s) - d(t',s')\right| \leqslant 10 \delta +\H(\tree'_{[r,
  r+3\delta]\vir x'}).
\]
	\textbf{Case 3:} Assume that $(t,t'),(s,s') \notin \mathcal{R}$. \\
	\textbf{Case 3a.} If $t=s=x_r$, then necessarily $H(t'\wedge x'),
    H(s'\wedge x') \in [r,r+3\delta)$. Arguing as in \eqref{ancestor}, we have 
\[
\left| d(t,s)-d'(t',s')\right| = d'(t',s') \leqslant d'(t',x'_r)+d'(x'_r,s')
\leqslant 2 \H(\tree'_{[r,
  r+3\delta]\vir x'}).
\]
	\textbf{Case 3b.} If $s=x_r$ and $t' = x'_r$, then by the same
    argument we used to
    get \eqref{ancestor}, we have
\[
\left| d(t,s)-d(t',s')\right| \leqslant d(t,x_r) + d(x'_r,s') 
\leq    \H(\tree_{[r,  r+3\delta]\vir x})+ \H(\tree'_{[r,  r+3\delta]\vir x'}).
\]

	It follows that
	\begin{equation}\label{dis}
	\operatorname{dis} ( \widetilde{\mathcal{R}}) \leqslant 10 \delta +
    2 \H(\tree_{[r,  r+3\delta]\vir x})+ 2\H(\tree'_{[r,  r+3\delta]\vir x'}). 
	\end{equation}
	
\noindent
\textbf{Step 2:} we define a measure on  $\tree_{r\vir x}\times 
\tree'_{r\vir x'}$ and give an upper bound of its discrepancy. 	Denote by
$\widetilde{m}$ the restriction of the measure $m$ to $\tree_{r\vir x}\times
\tree'_{r\vir x'}$. Let  $A \subset \tree_{r\vir x}$ be a Borel set. We have $\widetilde{m} \circ \widetilde{p}^{-1}(A) = \widetilde{m}(A \times \tree'_{r\vir x'}) = m(A \times\tree'_{r\vir x'})$ where $\widetilde{p} \colon \tree_{r\vir x} \times \tree'_{r\vir x'} \to \tree_{r\vir x}$ is the canonical projection. Notice that
	\begin{align*}
	m\left(A \times \tree'\right) - m\left(A \times \tree'_{r\vir x'}\right) &= m\left(A \times(\tree'\setminus\tree'_{r\vir x'})\right) \\
	&=m\left(A \times(\tree'\setminus\tree'_{r\vir x'}) \cap \mathcal{R}\right) + m\left(A \times(\tree'\setminus\tree'_{r\vir x'})\cap \mathcal{R}^c\right) \\
	&\leqslant m\left(A \times(\tree'\setminus\tree'_{r\vir x'}) \cap \mathcal{R}\right) + \delta.
	\end{align*}
	For  $(t,t') \in \big( A \times(\tree'\setminus\tree'_{r\vir x'})\big) \cap \mathcal{R}$, using \eqref{heights} and the fact that $A \subset \tree_{r\vir x}$, we get
	\[H(t'\wedge x')\geqslant H(t\wedge x) - 3 \delta \geqslant r- 3\delta.\]
	Moreover, we have $H(t'\wedge x') < r <r +3 \delta$ since $t' \notin
    \tree'_{r\vir x'}$. This gives the inclusion 
	$\big( A \times(\tree'\setminus\tree'_{r\vir x'})\big) \cap \mathcal{R} \subset
    \tree\times \tree'_{[r-3\delta, r+3\delta]\vir x'}$. As $d_{\mathrm{TV}}(m\circ
      {p'}^{-1} ,\mu') \leq \operatorname{D}(m; \mu,
    \mu')  \leq  \delta$, we deduce that
	\begin{align*}
	m\left(A \times \tree'\right) - m\left(A \times \tree'_{r\vir x'}\right) 
&\leqslant m\left( \tree\times \tree'_{[r-3\delta, r+3\delta]\vir x'}\right) + \delta \\
	&\leqslant \mu'\left(\tree'_{[r-3\delta, r+3\delta]\vir x'}\right) +
     d_{\mathrm{TV}}(m\circ {p'}^{-1} ,\mu') + \delta\\ 
	&\leqslant \mu'\left(\tree'_{[r-3\delta, r+3\delta]\vir x'}\right) + 2
      \delta. 
	\end{align*}
Recall  that $\widetilde{\mu}$ is the restriction of the measure $\mu$ to
    $\tree_{r, x}$. 
It follows that
	\begin{align*}\left|\widetilde{m} \circ \widetilde{p}^{-1}(A) -
   \widetilde{\mu}(A)\right| 
&= \left|m\left(A \times \tree'_{r\vir x'}\right)  - \mu(A)\right| \\
	&\leqslant \left|m\left(A \times \tree'_{r\vir x'}\right)  - m(A \times
      \tree')\right| + \left|m\left(A \times \tree'\right)  -
      \mu(A)\right| \\ 
	&\leqslant \left|m\left(A \times \tree'_{r\vir x'}\right)  - m(A \times
      \tree')\right| + \operatorname{D}(m;\mu, \mu') \\ 
	&\leqslant \mu'\left(\tree'_{[r-3\delta, r+3\delta]\vir x'}\right)  + 3 \delta.
	\end{align*}
	By symmetry, we deduce that 
	\begin{equation}\label{mass}
	\operatorname{D}(\widetilde{m}; \widetilde{\mu}, \widetilde{\mu}') 
\leqslant \m\big(\tree_{[r-3\delta, r+3\delta]\vir x}\big) +
\m\big(\tree'_{[r-3\delta, r+3\delta]\vir x'}\big) + 6 \delta. 
	\end{equation}
	
\noindent
\textbf{Step 3}: we give an upper  bound of
$\widetilde{m}(\widetilde{\mathcal{R}}^c)$. Let $(t,t') \in \tree_{r\vir x}
\times \tree_{r\vir x'}\setminus \widetilde{\mathcal{R}}$. If $H(t\wedge x)
> r + 3\delta$ then necessarily $(t,t') \notin \mathcal{R}$ by
our construction of $\widetilde{\mathcal{R}}$. Therefore, we have 
\begin{align}
\label{discrepancy}
	m(\widetilde{\mathcal{R}}^c)= m(\tree_{r\vir x}\times \tree'_{r\vir x'}
      \setminus \widetilde{\mathcal{R}}) 
&= m\left((t,t') \in \tree_{r\vir x} \times \tree'_{r\vir x'}\setminus
  \widetilde{\mathcal{R}} \colon \, H(t\wedge x) >r + 3
  \delta\right) \\
 & \quad\quad\quad + m\left((t,t') \in \tree_{r\vir x} \times \tree'_{r\vir x'} \colon \, H(t\wedge x) \in
   [r,r+3\delta]\right)\notag \\ 
	&\leqslant m(\mathcal{R}^c) + \mu\left(\tree_{[r, r+3\delta]\vir  x}\right)
      + d_{\mathrm{TV}}(m\circ {p}^{-1} ,\mu)
\notag\\ 
	&\leqslant \m\big(\tree_{[r, r+3\delta]\vir x}\big) + 2\delta. \notag
\end{align}

\noindent
\textbf{Step 4}: we can now conclude. Combining \eqref{dis}, \eqref{mass} and
\eqref{discrepancy} and using the definition of the
Gromov-Hausdorff-Prokhorov distance, we get 
\begin{equation}
\label{continuity lemma t t'}
\ghp(\tree_{r\vir x},\tree'_{r\vir x'} )\leqslant 6 \delta + 
\m\big(\tree_{[r-3\delta, r+3\delta]\vir x}\big)  + 
\m\big(\tree'_{[r-3\delta, r+3\delta]\vir x'}\big)  + 
\H\big(\tree_{[r, r+3\delta]\vir x}\big) +
\H\big(\tree'_{[r, r+3\delta]\vir x'}\big). 
\end{equation}

First, notice that
	\begin{align*}
\m\big(\tree'_{[r-3\delta, r+3\delta]\vir x'}\big) 
& =	\mu'\left(t'\colon \, H(t'\wedge x') \in
  [r-3\delta,r+3\delta]\right) \\
&\leqslant m\left((t,t') \colon \,  H(t'\wedge x') \in [r-3\delta,r+3\delta]\right) +d_{\mathrm{TV}} (m\circ {p'}^{-1} ,\mu')\\
	&\leqslant  m\left((t,t') \in \mathcal{R}\colon \,  H(t'\wedge x')
      \in [r-3\delta,r+3\delta]\right) +
      m\left(\mathcal{R}^c\right)+\delta \\ 
	&\leqslant  m\left((t,t') \in \mathcal{R}\colon \,  H(t'\wedge x')
      \in [r-3\delta,r+3\delta]\right) + 2 \delta. 
	\end{align*}
Using \eqref{heights}, we get by symmetry that, for 	 $(t,t') \in
\mathcal{R}$, 
\begin{equation}
   \label{eq:Ht-Ht'-comparaison}
H(t'\wedge x') - 3 \delta \leqslant H(t\wedge x) \leqslant H(t'\wedge x') + 3 \delta.
 \end{equation}	
	We deduce that
	\begin{align}
\m\big(\tree'_{[r-3\delta, r+3\delta]\vir x'}\big) 
&\leqslant m\left((t,t') \in \mathcal{R}\colon \,  H(t\wedge x) \in
  [r-6\delta,r+6\delta]\right) + 2 \delta \notag\\ 
	&\leqslant m\left((t,t') \colon \,  H(t\wedge x) \in [r-6\delta,r+6\delta]\right) + 2 \delta \notag\\
	&\leqslant \mu\left(t \colon \,  H(t\wedge x) \in [r-6\delta,r+6\delta]\right) + d_{\mathrm{TV}}(m\circ p^{-1} ,\mu)+ 2\delta \notag \\
	&\leqslant 	\m\big(\tree_{[r-6\delta, r+6\delta]\vir x}\big)  + 3 \delta.\label{mu mu'}
	\end{align}
	
Secondly, let $t'\in \tree'_{[r, r+3\delta]\vir  x'}$. We have $H(t' \wedge x')\in [r,
r+3\delta]$. Let $t\in T$ such that  $(t,t') \in \mathcal{R}$. 
Thanks to
\eqref{eq:Ht-Ht'-comparaison}, we get  $H(t\wedge x) \in [r-3\delta,
    r+6\delta]$. Since  $(t,t') \in \mathcal{R}$, we also have $|H(t')
    -H(t)|\leq  2\delta$ by\eqref{distances 2}. We deduce that
\begin{align}
\notag
 \H\big(\tree'_{[r, r+3\delta]\vir x'}\big)
&=   \sup \left\{H(t') \colon \, t'\in \tree', H(t'\wedge x') \in [r,
  r+3\delta]\right\} - r\\
\notag
&\leq   \sup\left\{H(t)\colon \, t \in \tree, H(t\wedge x) \in
  [r-3\delta,r+6\delta] \right\} - r + 2\delta\\
&= \H\big(\tree_{[r-3\delta , r+6\delta]\vir x}\big) -\delta. 
\label{H H'}
\end{align}
	Using \eqref{mu mu'} and \eqref{H H'} in conjunction with
    \eqref{continuity lemma t t'} yields the result. 
\end{proof}

\begin{proof}[Proof of Proposition \ref{continuity of the measure}]
	Fix a compact real tree $\tree=(\tree, d, \root, \mu)$. We will show that
    $\Psi_{\tree'}\to\Psi_{\tree}$ weakly as $\tree'  \to \tree$ for
    $\ghp$. Let $\epsilon >0$ and let $\tree'=(\tree', d', \root',
    \mu')$ be a compact real tree  such that
    $\ghp(\tree,\tree') \leqslant \epsilon$. Then there exist a
    correspondence $\mathcal{R}$ between $\tree$ and $\tree'$ and a
    measure $m$ on $\tree \times \tree '$ such that $(\root,\root')\in
    \mathcal{R}$, $m(\mathcal{R}^c) \leqslant \epsilon$,
    $\operatorname{dis} (\mathcal{R})  \leqslant 2\epsilon$ and
    $\operatorname{D}(m; \mu, \mu ') \leqslant \epsilon$. In particular,
    we will make constant use of the inequalities $|m(\tree \times \tree
    ') - \m(\tree)| \leqslant \epsilon$ and $|H(x) - H(x')| \leqslant 2
    \epsilon$ for $(x,x') \in \mathcal{R}$. Let $f\in \cc_b(\T \times \R_+
    )$ be  Lipschitz. Write 
	\[\Psi_\tree (f) - \Psi_{\tree'}(f) = A_1 + A_2 + A_3 + A_4, \]
	where
	\begin{align*} 
A_1 
&= \int_\tree \mu(\dd x) \int_0^{H(x)} f(\tree_{r\vir x},r)\dd r -
      \int_{\tree} m\circ p^{-1}(\dd x) \int_0^{H(x)} f(\tree_{r\vir x},r)
      \dd r \\ 
A_2
&= \int_{\mathcal{R}} m(\dd x, \dd x') \left( \int_0^{H(x)}
  f(\tree_{r\vir x},r) \dd r -   \int_0^{H(x')} f(\tree'_{r\vir x'},r) \dd
  r\right) \\ 
A_3
&= \int_{\mathcal{R}^c} m(\dd x, \dd x') \left( \int_0^{H(x)}
  f(\tree_{r\vir x},r) \dd r -   \int_0^{H(x')} f(\tree'_{r\vir x'},r) \dd
  r\right) \\ 
A_4 
&= \int_{ \tree '} m\circ {p'}^{-1}( \dd x') \int_0^{H(x')}
  f(\tree'_{r\vir x'},r) \dd r - \int_{\tree'} \mu(\dd x') \int_0^{H(x')}
  f(\tree'_{r\vir x'},r)\dd r. 
	\end{align*}
	
	Notice that
	\begin{equation}|A_1| \leqslant 2d_{\mathrm{TV}}(m \circ p^{-1} ,\mu) \sup_{x \in \tree} \int_0^{H(x)} f(\tree_{r\vir x},r) \dd r \leqslant 2\H(\tree) \norm{f}_\infty \epsilon. \label{A1}
	\end{equation}
	Similarly, we have
	\begin{equation}|A_4| \leqslant 2\H(\tree ') \norm{f}_\infty \epsilon \leqslant 2(\H(\tree) + 2\epsilon) \norm{f}_\infty \epsilon, \label{A4}
	\end{equation}
	where in the second inequality we used that $\H(\tree') \leqslant \H(\tree) + 2\ghp(\tree,\tree') \leqslant \H(\tree) + 2\epsilon$ by \eqref{ghp inequality}. Next, we have
	\begin{equation} |A_3| \leqslant m(\mathcal{R}^c)(\H(\tree) + \H(\tree '))\norm{f}_\infty \leqslant 2(\H(\tree) + \epsilon)\norm{f}_\infty \epsilon.\label{A3}
	\end{equation}
	
	We now provide a bound for $A_2$. We have
\begin{multline}
A_2 
= \int_{\mathcal{R}} \ind_{\lbrace H(x)\geqslant H(x')\rbrace}m(\dd x,
\dd x') \left( \int_0^{H(x)} f(\tree_{r\vir x},r) \dd r -
  \int_0^{H(x')} f(\tree'_{r\vir x'},r) \dd r\right) 
\\ + \int_{\mathcal{R}} \ind_{\lbrace H(x)<H(x')\rbrace}m(\dd x, \dd x')
\left( \int_0^{H(x)} f(\tree_{r\vir x},r) \dd r -   \int_0^{H(x')}
  f(\tree'_{r\vir x'},r) \dd r\right). 
	\end{multline}
	We only treat the first term, the second one being similar. We have
	\begin{multline*}\int_{\mathcal{R}} \ind_{\lbrace H(x)\geqslant
        H(x')\rbrace}m(\dd x, \dd x') \left( \int_0^{H(x)}
        f(\tree_{r\vir x},r) \dd r -   \int_0^{H(x')} f(\tree'_{r\vir x'},r) \dd
        r\right) 
\\= \int_{\mathcal{R}} \ind_{\lbrace H(x)\geqslant H(x')\rbrace}m(\dd x,
\dd x') \left( \int_0^{H(x')} \left(f(\tree_{r\vir x},r) - f(\tree
    '_{r\vir x'},r) \right) \dd r +   \int_{H(x')}^{H(x)} f(\tree_{r\vir x},r)
  \dd r\right). 
	\end{multline*}
On the one hand, we get
	\begin{align}
	\left|\int_{\mathcal{R}} \ind_{\lbrace H(x)\geqslant H(x')\rbrace} m(\dd x, \dd x')\int_{H(x')}^{H(x)} f(\tree_{r\vir x},r) \dd r\right|  &\leqslant \int_{\mathcal{R}} \norm{f}_\infty|H(x) - H(x')| m(\dd x, \dd x') \notag\\
	& \leqslant \norm{f}_\infty m (\tree \times \tree ')\operatorname{dis} (\mathcal{R} ) \notag\\
	& \leqslant 2\norm{f}_\infty (\m(\tree) + \epsilon)\epsilon.\label{A2'}
	\end{align}
	On the other hand, we have
	\begin{align}
	&\left|\int_{\mathcal{R}} \ind_{\lbrace H(x)\geqslant
      H(x')\rbrace}m(\dd x, \dd x') \int_0^{H(x')}
      \left(f(\tree_{r\vir x},r) - f(\tree '_{r\vir x},r) \right) \dd r \right|
      \notag\\ 
	& \leqslant  \L{f} \int_{\mathcal{R}} \ind_{\lbrace H(x)\geqslant
      H(x')\rbrace}m(\dd x, \dd x') \int_{0}^{H(x')}\!\!
      \ghp\left(\tree_{r\vir x},\tree '_{r\vir x}\right) \ind_{\lbrace r
      \geqslant 6 \epsilon\rbrace} \dd r + \int_{\mathcal{R}} m(\dd x,
      \dd x') \int_{0}^{6\epsilon} 2 \norm{f}_\infty \dd r \notag\\ 
      &\leqslant 2\L{f} \int \!m(\dd x, \dd x') \int_{0}^{H(x)}
      \left( \m( \tree_{[r-3\varepsilon, r+ 6\varepsilon]\vir x})+\H(
      \tree_{[r-6\varepsilon, r+ 6\varepsilon]\vir x}) \right)\, \ind_{\lbrace r
      	\geqslant 6 \epsilon\rbrace} \, \dd r \notag \\
	&\hspace{4cm}  + 8\L{f} \H(\tree)(\m(\tree) + \epsilon) \epsilon +
      12\norm{f}_\infty(\m(\tree) + \epsilon)\epsilon.\label{A2} 
	\end{align}
	where we used \eqref{continuity lemma t} for the last
    inequality. Using Fubini's theorem, we get 
	\begin{multline}\label{A21}
\int \!m(\dd x, \dd x') \int_{0}^{H(x)} \!\m(
      \tree_{[r-6\varepsilon, r+ 6\varepsilon]\vir x}) 
      \ind_{\lbrace r \geqslant 6\epsilon\rbrace}\, \dd r \\
\begin{aligned}[b]
&= 	\int \!m(\dd x, \dd x') \int_{0}^{H(x)} \!\mu(t \colon \, H(t\wedge
x) \in [r-6\epsilon,r+6\epsilon]) \ind_{\lbrace r \geqslant 6
  \epsilon\rbrace}\, \dd r \\ 
&= \int \!m(\dd x, \dd x') \int_\tree \mu(\dd t) \int_{0}^{H(x)}
\!\ind_{\lbrace H(t\wedge x) \in
  [r-6\epsilon,r+6\epsilon]\rbrace}\ind_{\lbrace r \geqslant 6
  \epsilon\rbrace} \,\dd r \\ 
&\leqslant 12\,  \m(\tree)(\m(\tree) + \epsilon) \epsilon. \end{aligned} 
\end{multline}
	Moreover, since $\tree$ is compact, it holds that for every $x \in
    \tree$ and every $\delta >0$, there is a finite number of subtrees
    with height larger than $\delta$ attached to the branch $\llbracket
    \root , x \rrbracket$. Let  $r \in (0,H(x))$. Recall  that $x_r$ is 
    the unique ancestor of $x$ with height $H(x_r) = r$. Assume that
    $x_r$ is not a branching point. Then, for every $\delta >0$ and for
    $\epsilon >0$ small enough (depending on $\delta$), there are no
    subtrees with height larger than $\delta$ attached to $\llbracket
    x_{r-3\epsilon}, x_{r+6\epsilon}\rrbracket$. (To be precise, if $y \in \llbracket
    x_{r-3\epsilon}, x_{r+6\epsilon}\rrbracket$ is a branching point, the tree attached at $y$ is $\tree_{[s,s]\vir x}$ with $s = H(y)$). Therefore, we have $
\H(T_{[r-3\varepsilon, r+6\varepsilon]\vir x}) \leqslant \delta+9\epsilon$.
	This proves that,	for every $r \in (0,H(x))$ such that $x_r$ is
    not a branching point, 
	\begin{equation}
\label{branching points}
\lim_{\epsilon \to 0}\H(T_{[r-3\varepsilon, r+6\varepsilon]\vir x}) =0 . 
	\end{equation}
But since $\tree$ is compact, there are (at most) countably many  $r \in
(0,H(x))$ such that $x_r$ is a branching point. It follows that
\eqref{branching points} holds for every $x \in \tree$ and $\dd r$-a.e. $r \in [0,H(x)]$.
Notice that  	$\H(T_{r-3\varepsilon, r+6\varepsilon\vir x})\leq
\H(\tree)$ and the measure $\ind_{\lbrace 0 \leqslant r \leqslant
  H(x)\rbrace}\,\mu(\dd x) \dd r$ is finite as its total mass is less
than $\H(\tree) \m(\tree) $ which is finite. We get by the dominated convergence theorem that
\[
\lim_{\epsilon \to 0}\int_{\tree} \mu(\dd x) \int_{0}^{H(x)}
\!\H(T_{[r-3\varepsilon, r+6\varepsilon]\vir x}) \ind_{\lbrace r \geqslant 6\epsilon\rbrace} \,\dd r = 0.
\]
	Since
\[\left|\int_{\tree} \left(m \circ p^{-1}(\dd x)-\mu(\dd x)\right)
  \int_{0}^{H(x)} \!\H(T_{[r-3\varepsilon, r+6\varepsilon]\vir x})
  \ind_{\lbrace r \geqslant 6\epsilon\rbrace} \,\dd r \right| 
 \leqslant 2\H(\tree)^2 d_{\mathrm{TV}}(m \circ p^{-1},
  \mu) \leqslant 2\H(\tree)^2 \epsilon, 
\]
	it follows that
	\begin{equation}\lim_{\epsilon \to 0}\int m(\dd x, \dd x')
      \int_{0}^{H(x)} \!\H(T_{[r-3\varepsilon, r+6\varepsilon]\vir x})
      \ind_{\lbrace r \geqslant 6 \epsilon\rbrace}\,\dd r  =
      0.\label{A22} 
	\end{equation}
	Thus, by
	\crefrangemultiformat{equation}{(#3#1#4)--(#5#2#6)}{ and~(#3#1#4)--
      (#5#2#6)}{, ~(#3#1#4)--(#5#2#6)}{
      and~(#3#1#4)--(#5#2#6)}\cref{A1,A4,A3,A2',A2,A21,A22}, we deduce
    that 
	\[\lim_{\epsilon \to 0}\sup_{\substack{\ghp(\tree,\tree')< \epsilon}}\Psi_{\tree'}(f) = \Psi_{\tree}(f)\]
	for every  Lipschitz function $f\in \cc_b(\T \times \R_+)$. This proves that $\Psi \colon \T \to \M(\T \times \R_+)$ is continuous which concludes the proof.
\end{proof}

\section{Bienaymé-Galton-Watson trees and stable Lévy trees}\label{Sect BGW}
Throughout this work, we fix a random variable $\xi$ whose distribution is critical and belongs to the domain of attraction of a stable distribution with index $\gamma \in (1,2]$. More precisely, we assume that $\xi$ takes values in $\mathbb{N} = \{0,1,2,\ldots\}$ and that it satisfies the following conditions:
\begin{enumerate}[label=($\xi$\arabic*),leftmargin=*]
	\item $\xi$ is critical, \emph{i.e.} $\ex{\xi} = 1$, and nondegenerate, \emph{i.e.} $\pr{\xi = 0} >0$,\label{xi1}
	\item $\xi$ belongs to the domain of attraction of a stable distribution with index $\gamma \in (1,2]$, \emph{i.e.} $\ex{\xi^2 \ind_{\{\xi \leqslant n\}}} = n^{2-\gamma}L(n)$, where $L \colon \R_+ \to \R_+$ is a slowly varying function. \label{xi2}
\end{enumerate}
By \cite[Theorem XVII.5.2]{feller1971introduction} or \cite[Theorem
5.2]{janson2011stable}, assumption \ref{xi2} is equivalent to the
existence of a positive sequence $(b_n, \, n \geqslant 1)$ such that, if
$(\xi_n,\, n \geqslant 1)$ is a sequence of independent random variables with the same distribution as $\xi$, then
\begin{equation} 
\label{stable CLT}
\frac{1}{b_n}\left( \sum_{k=1}^n \xi_k - n \right) \cvlawd X_1,
\end{equation}
where $(X_t, \, t \geqslant 0)$ is a strictly stable spectrally positive
Lévy process with Laplace transform $\ex{\exp(-\lambda X_t)} = \exp(t\kappa \lambda^\gamma)$ where $\gamma \in (1,2]$ and $\kappa >0$. Note that we have automatically $b_n/n \to 0$ as $n \to \infty$. In most of our results, we make the following stronger assumption on $\xi$:
\begin{enumerate}[start=2, label = ($\xi$\arabic*)$'$,leftmargin=*]
	\item $\ex{\xi^2 \ind_{\{\xi \leqslant n\}}} = n^{2-\gamma}L(n) $
      where $L \colon \R_+ \to \R_+$ is a slowly varying function which
      is  bounded away
      from zero and infinity. \label{xi3} 
\end{enumerate}
Assumption \ref{xi3} is equivalent to  the normalizing
sequence $(b_n, \, n \geqslant 1)$ which appears in \eqref{stable CLT}
satisfying 
\begin{equation}\label{b_n bounded}
\underline{b}\:\! n^{1/\gamma} \leqslant  b_n \leqslant \overline{b}\:\!  n^{1/\gamma}, \quad \forall n \geqslant 1,
\end{equation}
for some constants $0 <\underline{b} < \overline{b}< \infty$. Indeed, if
$\gamma = 2$, we have the convergence of $nb_n^{-2} L(b_n)$ to some
positive constant by \cite[Theorem 5.2 and
Eq. (5.44)]{janson2011stable}. Similarly, if $\gamma \in (1,2)$,  using
\cite[Theorem 5.3 and Eq. (5.7)]{janson2011stable}, we have as $n \to
\infty$ that 
\[
n \pr{\xi > b_n}\sim \frac{2-\gamma}{\gamma}n b_n^{-\gamma}L(b_n).
\] 
On the other hand, \cite[Eq. (5.10)]{janson2011stable} entails the
convergence of $n \pr{\xi > b_n}$ to some positive constant. Therefore,
for $\gamma \in (1,2]$, the sequence $n^{1/\gamma} b_n^{-1}
L(b_n)^{1/\gamma}$ converges to some positive constant. Thus, if  $L$ is
bounded away from 0 and infinity, then \eqref{b_n bounded} follows. The proof of the
converse (which we shall not use) is left for the reader. 

\subsection{Results on conditioned Bienaymé-Galton-Watson trees}
Recall that the span of the integer-valued random variable $\xi$ is the largest integer $\spn$ such that a.s. $\xi \in a + \spn \mathbb{Z}$ for some $a \in \mathbb{Z}$. As we only consider $\xi$ with $\pr{\xi = 0} >0$, the span is the largest integer $\spn$ such that a.s. $\xi \in \spn \mathbb{Z}$, \emph{i.e.} the greatest common divisor of $\{k \geqslant 1\colon \, \pr{\xi = k} >0\}$.

Assume that $\xi$ satisfies \ref{xi1} and \ref{xi2} and denote by $\mathfrak{g}$ the density of the random variable $X_1$ appearing in \eqref{stable CLT}. Then the function $\mathfrak{g}$ is continuous on $\real$ (in fact infinitely differentiable) and satisfies
\begin{equation}
\label{g(0)} 
\mathfrak{g}(0) = \frac{1}{\kappa^{1/\gamma}
  \left|\Gammaeuler(-1/\gamma)\right|},
\end{equation}
where $\Gamma$ is Euler's gamma function, see \cite[Lemma XVII.6.1]{feller1971introduction} or \cite[Example 3.15 and Eq. (4.6)]{janson2011stable}. In particular, when $\gamma = 2$, $\mathfrak{g}$ is the density of a centered Gaussian distribution with variance $2 \kappa$ and we have
\begin{equation}
\label{g(0),2} 
\mathfrak{g}(0) = \frac{1}{2\sqrt{\kappa \pi}}\cdot
\end{equation}
Recall that $(\xi_n, \,  n \geqslant 1)$ is a sequence  of independent random
variables   with   the   same   distribution   as   $\xi$   and   define
$S_n = \sum_{k=1}^n \xi_k$. The following result is a direct consequence
of the local  limit theorem, see \emph{e.g.}  \cite[Chapter 4, Theorem
4.2.1]{ibragimov}.

\begin{lemma}[Local limit theorem]
\label{local limit}
Assume that $\xi$ satisfies \ref{xi1} and \ref{xi2} and denote its span by $\spn$. We have
\[ 
\lim_{n\to \infty}\sup_{k \geqslant 0} \left| \frac{b_n}{\spn} \pr{S_n = \spn k} -
  \mathfrak{g}\left( \frac{\spn k-n}{b_n}\right) \right| = 0, 
\]
	where $\mathfrak{g}$ is the density of the random variable $X_1$
    defined in \eqref{stable CLT}. In particular, for any fixed $k
    \geqslant 0$, we have as $n \to \infty$ with $n \equiv k\, 
    (\mathrm{mod} \ \spn)$,
\begin{equation}
   \label{eq:Sn--sim}
 \pr{S_n = n-k} \sim \frac{\spn \mathfrak{g}(0) }{b_n}\cdot 
\end{equation}
\end{lemma}

Let $\rddtree$  be a BGW($\xi$) tree,  see \emph{e.g.}  Athreya and  Ney \cite{Athreya}.  By the
well-known Otter-Dwass formula, we have, for every $n \geqslant 1$,
\begin{equation}\label{otter dwass}
\pr{|\rddtree| = n} = \frac{1}{n}\pr{S_n = n-1}.
\end{equation}
In particular, we get $\pr{|\rddtree| = n} = 0$ if $n \centernot\equiv
1\ (\mathrm{mod} \ \spn)$ while $\pr{|\rddtree| = n} >0$ for all large
$n$ with $n \equiv 1 \ (\mathrm{mod} \ \spn)$ by Lemma \ref{local
  limit}. We denote by $\supp$ the support of the random variable
$|\rddtree|$ when $\rddtree$ is not reduced to the root, that is
\begin{equation}\label{support}
\supp = \left\lbrace n \geqslant 2 \colon \, \pr{|\rddtree| = n} >0 \right\rbrace.
\end{equation}
In particular, the previous discussion implies that $\supp \subset 1 +\spn \mathbb{N}$ and conversely, $1+\spn n \in \supp$ for all large $n$. In what follows, we only consider $n \in \supp$ and convergences should be understood along the set $\supp$.

We will also need the following sub-exponential tail bounds for the height of conditioned BGW trees, see \cite[Theorem 2]{kortchemski2017sub} and the discussion thereafter. For every $n \in \supp$, $\rddtree^n$ will denote a BGW($\xi$) tree conditioned to have $n$ vertices, that is $\rddtree^n$ is distributed as $\rddtree$ conditionally on $\{|\rddtree| = n\}$.
\begin{lemma}\label{subgaussian}
	Assume that $\xi$ satisfies \ref{xi1} and \ref{xi2}. For every
    $\alpha \in (0,\gamma/(\gamma - 1))$ and every $\beta \in
    (0,\gamma)$, there exist two finite constants $C_0, c_0 >0$ such that for every $y \geqslant 0$ and  $n \in \supp$, we have
	\begin{align}
	\pr{\frac{b_n}{n}\H(\rddtree^n) \leqslant y } &\leqslant C_0  \exp\left(-c_0y^{-\alpha}\right),\label{subgaussian 0}\\ 
	\pr{\frac{b_n}{n}\H(\rddtree^n) \geqslant y} &\leqslant C_0 \exp\left(-c_0y^\beta\right). \label{subgaussian infinity}
	\end{align}
\end{lemma}
\begin{remark}\label{remark subgaussian}
	\begin{enumerate}[label=(\roman*),leftmargin=*]
		\item[]
		\item If moreover $\xi$ satisfies \ref{xi3}, then we can take $\alpha = \gamma/(\gamma-1)$ in \eqref{subgaussian 0}, see Appendix \ref{Kortchemski}.
		\item If $\xi$ has finite variance $\sigmaxi^2 \in (0,\infty)$ (in
          which case \ref{xi3} is satisfied), we have $\gamma = 2$ and
          we can take $b_n =\sigmaxi\sqrt{n}$ in \eqref{stable CLT} with
          $\kappa=1/2$ (this is just the central limit theorem). Then both \eqref{subgaussian 0} and \eqref{subgaussian infinity} hold with $\alpha = \beta = 2$, see \cite[Theorem 1.1 and Theorem 1.2]{addario2013sub}.
	\end{enumerate}
\end{remark}
An immediate consequence of Lemma \ref{subgaussian} is the following estimate for the moments of $\H(\rddtree^n)$ which extends \cite[Corollary 1.3]{addario2013sub}.
\begin{lemma}
\label{height finite moments}
	Assume that $\xi$ satisfies \ref{xi1} and \ref{xi2}. For every $p \in \real$, we have
	\[\sup_{n \in \supp} \ex{\left(\frac{b_n}{n}\H(\rddtree^n)\right)^p\,} < \infty.\]
\end{lemma}
\begin{proof}
	Let $p >0$. Fix $\beta \in (0,\gamma)$. By Lemma \ref{subgaussian}, we have for every $n \in \supp$
	\[\ex{ \left(\frac{b_n}{n}\H(\rddtree^n)\right)^p \,} = p\int_0^\infty  y^{p-1} \pr{ \frac{b_n}{n}\H(\rddtree^n) > y} \, \dd y \leqslant C_0 p\int_0^\infty  y^{p-1} \e^{-c_0 y^\beta} \, \dd y < \infty.\]
	Similarly, fix $\alpha \in (0,\gamma/(\gamma-1))$ and apply Lemma \ref{subgaussian} to get
\[
\ex{\left(\frac{b_n}{n}\H(\rddtree^n)\right)^{-p}\,} =
    p\int_0^\infty y^{p-1} \pr{\frac{b_n}{n}\H(\rddtree^n)< \frac{1}{y}}
    \, \dd y \leqslant C_0 p\int_0^\infty y^{p-1} e^{-c_0 y^\alpha} \,
    \dd y < \infty.
\]
This  proves the result.
\end{proof}

We end this section with the following lemma used in the proof of Remark
\ref{rem:main}-\ref{local regime remark}. 
\begin{lemma}\label{local regime lemma}
	Assume that $\xi$ has finite variance $\sigmaxi^2 \in (0,\infty)$. Let
    $\alpha', \beta \in \real$ such that $2\alpha' + \beta < 0$ and set
    $f_{\alpha',\beta}(\mathbf{t})= |\mathbf{t}|^{\alpha'}\,
    \H(\mathbf{t})^\beta\ind_{\{|\mathbf{t}|>1\}}$.  Then we have 
\[
\ex{f_{\alpha',\beta}(\rddtree)} < \infty, 
\quad
	\lim_{n\to \infty} \ex{f_{\alpha',\beta}(\rddtree^n)^2} = 0 
\quad    \text{and} \quad 
\sum_{n \in      \supp}\frac{\sqrt{\ex{f_{\alpha',\beta}(\rddtree^n)^2}}}{n} <
    \infty. 
\]
\end{lemma}
\begin{proof}
	We have
	\begin{equation*}
	\ex{f_{\alpha',\beta}(\rddtree)} = \sum_{n \in \supp} n^{\alpha'} \ex{\H(\rddtree^n)^\beta} \pr{|\rddtree| = n} .
	\end{equation*}
Using \eqref{otter dwass} and \eqref{eq:Sn--sim}, 
\eqref{g(0),2} with $b_n = \sigmaxi \sqrt{n}$, we have as
$n\to \infty$ that
	\begin{equation*}
	\pr{|\rddtree|= n}  \sim \frac{\spn}{\sqrt{2\pi\sigmaxi^2}} n^{-3/2}.
	\end{equation*}
	Since $\ex{\H(\rddtree^n)^\beta} = O(n^{\beta/2})$ as $n \to \infty$
    by Lemma \ref{height finite moments}, we get that 
\[
	\ex{f_{\alpha',\beta}(\rddtree)} \leqslant C \sum_{n \in \supp}
    n^{-3/2+\alpha' + \beta/2} < \infty. 
\]
	 Applying Lemma \ref{height finite moments} again gives 
	$\ex{f_{\alpha',\beta}(\rddtree^n)^2} = n^{2\alpha'}
    \ex{\H(\rddtree^n)^{2\beta}}\ind_{\{n>1\}}\leqslant M
    n^{2\alpha'+\beta}$ for some finite constant $M>0$, and the last
    term converges to $0$ as $n \to \infty$. Finally, we have 
\[
	\sum_{n\in\supp}
    \frac{\sqrt{\ex{f_{\alpha',\beta}(\rddtree^n)^2}}}{n}
 \leqslant \sqrt{M} \sum_{n\in\supp} n^{-1+\alpha' + \beta/2} < \infty.
\]
\end{proof}

\subsection{Stable Lévy trees}\label{levy tree}
Let us briefly recall the definition of the height process and the
associated Lévy tree, see \emph{e.g.}
\cite{le1998branching, duquesne2002random, duquesne2003limit, kortchemski2013simple}. Recall that $(X_t, \, t\geqslant 0 )$ is a strictly stable Lévy process with Laplace exponent $\psi(\lambda) = \kappa \lambda^\gamma$ where $\gamma \in (1,2]$ and $\kappa >0$. For $\gamma \in (1,2)$, denote by $\pi$ the associated Lévy measure
\begin{equation}
\label{eq:def-pi}
\pi(\dd x) = \frac{\kappa \gamma (\gamma -1)}{\Gammaeuler(2-\gamma)}
\frac{\dd x}{x^{ 1+\gamma}}\cdot
\end{equation}
Le Gall and Le Jan \cite{le1998branching} proved that there exists a continuous process $(H(t), \, t \geqslant 0)$ called the $\psi$-height process such that for every $t \geqslant 0$, we have the following convergence in probability
\begin{equation*}
H(t) = \lim_{\epsilon \to 0} \frac{1}{\epsilon} \int_0^t \ind_{\{X_s < I_t^s + \epsilon\}}\, \dd s,
\end{equation*}
where $I_t^s = \inf_{[s,t]} X$. In the Brownian case, $H$ is a (scaled) reflected Brownian motion. Let $\n$ be the excursion measure of $H$ above $0$ and set
\begin{equation}\sigma = \inf \left\{s > 0 \colon \, H(s) = 0\right\} \quad \text{and} \quad \Hexc = \sup_{s \geqslant 0} H(s)
\end{equation} 
for the duration of the excursion and its maximum. We choose to normalize the excursion measure $\n$ such that the distribution of $\sigma$ under $\n$ is $\pi_*$ given by
\begin{equation}
\label{density duration excursion}
\pi_*(\dd x) =\n\left[\sigma \in \dd x\right]= \mathfrak{g}(0) \frac{\dd
  x}{x^{1+1/\gamma}},  
\end{equation}
with  $\mathfrak{g}(0)$  given in \eqref{g(0)}. Furthermore, by
\cite[Eq. (14)]{duquesne2005probabilistic}, the distribution of $\H$
under $\n$ is given by  
\begin{equation}\label{density height excursion}
 \n\left[\H >x\right] = \left(\kappa (\gamma-1)x\right)^{-1/(\gamma -1)}.
\end{equation}

We have 
 the following equality in ``distribution''   for the height process,
 see \emph{e.g.} \cite[Eq. (40)]{duquesne2017decomposition}, 
\begin{equation*}
\left(H(xt), \, t \geqslant 0\right) \quad \text{under} \ x^{1/\gamma}\n \quad \lawd  \quad x^{1-1/\gamma} H \quad \text{under}\ \n.
\end{equation*}
Using this, one can make sense of the conditional probability measure $\excm{x}[\bullet] = \n[\bullet | \sigma = x]$ such that $\excm{x}$-a.s., $\sigma = x$ and 
\[\n[\bullet] = \int_0^\infty \excm{x}[\bullet]\, \pi_*(\dd x) .\]
Informally, $\excm{x}$ can be seen as the distribution of the excursion of $H$ with duration $x$. Moreover, the height process $H$ has the following scaling property
\begin{equation}
\left(H(s), \, s \in [0,x]\right) \ \text{under} \ \excm{x} \quad \lawd \quad \left(x^{1-1/\gamma} H(s/x),\, s\in [0,1]\right) \ \text{under} \ \excm{1}. \label{scaling H}
\end{equation}
See also Lemma \ref{lemma scaling} for the scaling property of $H$ and
related processes. 

We call the stable Lévy tree with branching mechanism $\psi(\lambda) =
\kappa \lambda^\gamma$, the compact real tree $\rdtree$ coded by the
$\psi$-height process $H$ under $\excm{1}$. See Remark \ref{rem:coding}
 for the coding of real trees by excursion paths.
 
\begin{remark}Notice that $\sigma=\m(\rdtree_H)$ and
  $\Hexc=\H(\rdtree_H)$ are the mass and the height of the tree
  $\rdtree_H$ coded by the height process $H$ under $\n$. Furthermore,
  for $s \in [0,\sigma]$, the notation $H(s)$ is consistent with the one
   introduced in Section \ref{real trees} since $H(s)$ is the height of
   $s$ in the tree coded by $H$ under $\n$. 
\end{remark}

\subsection{Convergence of continuous functionals}
For every $n \in \supp$, we let $\rddtree^n$ be a BGW($\xi$) tree conditioned to have $n$ vertices, and let $\rdtree^n = (b_n/n)\rddtree^n$ be the associated real tree rescaled so that all edges have length $b_n/n$. Duquesne \cite{duquesne2003limit} (see also \cite{kortchemski2013simple}) showed that the convergence in distribution 
\begin{equation}\label{BGW trees convergence}
\rdtree^n \cvlawd \rdtree
\end{equation}
 holds in the space $\T$ where $\rdtree$ is the stable Lévy tree with
 branching mechanism $\psi(\lambda) = \kappa \lambda^\gamma$. 

The following result is an immediate consequence of Proposition
\ref{continuity of the measure}. Recall from \eqref{definition pi} and
\eqref{psi joint} the definitions of the measures $\Psi_{\tree}$ and
$\psij_\tree$. 
\begin{corollary}\label{convergence for continuous functions}
	 Assume that $\xi$ satisfies \ref{xi1} and \ref{xi2}. Let
     $\rddtree^n$ be a BGW($\xi$) tree conditioned to have $n$ vertices
     and let $\rdtree^n = (b_n/n)\rddtree^n$ be the associated real tree
     rescaled so that all edges have length $b_n/n$ (where $b_n$ is the
     normalizing sequence in \eqref{stable CLT}). Then we have the
     convergence in distribution $\Psi_{\rdtree^n} \cvlaw
     \Psi_{\rdtree}$ in $\mathcal{M}(\T \times \R_+)$, where $\rdtree$
     is the stable Lévy tree with branching mechanism $\psi(\lambda) =
     \kappa \lambda^\gamma$. In particular, we have $\psij_{\rdtree^n}
     \cvlaw \psij_{\rdtree^n}$ in $\M(\R_+^2)$. 
\end{corollary}

The convergence  in distribution obtained in  Corollary \ref{convergence
  for continuous  functions} is unsatisfactory  to study the asymptotics
of additive functionals of large BGW trees as it involves  the real
tree $\rdtree^n$  instead of  the (discrete)  BGW tree  $\rddtree^n$. To
remedy  this, we  shall  introduce  a discrete  version  of the  measure
$\Psi_\tree$  when $\tree$  is  associated with  a  discrete tree.   Let
$\mathbf{t}$  be a  discrete tree  and $a  >0$.  Recall that $a
\mathbf{t}$ denotes the real tree associated to $\mathbf{t}$ where the branches
have length $a$, and that for $v\in \mathbf{t}$, $av$ denotes the
corresponding vertex in $a\mathbf{t}$, see Section \ref{real
  trees} for the definitions.   We define  two
nonnegative     measures      $\mathcal{A}^\circ_{\mathbf{t},a}$     and
$\mathcal{A}_{\mathbf{t},a}$     on      $\T\times \R_+$     by,      for     every
$f\in \cb_+(\T\times \R_+)$,
\begin{equation}
\label{A general}
\boxed{\mathcal{A}^\circ_{\mathbf{t},a}(f) 
= \frac{a}{|\mathbf{t}|} \sum_{w \in \mathbf{t}^\circ}|\mathbf{t}_w|
  f\left(a\mathbf{t}_{w}, aH(w)\right)}
\quad\text{and}\quad 
\boxed{\mathcal{A}_{\mathbf{t},a}(f) 
= \frac{a}{|\mathbf{t}|} \sum_{w \in \mathbf{t}}|\mathbf{t}_w| f\left(
  a\mathbf{t}_{w},aH(w)\right)}, 
\end{equation}
 where $a\mathbf{t}_w$
is the subtree of $a\mathbf{t}$ above $aw$. Note that the sum is
over all internal vertices of $\mathbf{t}$ for
$\mathcal{A}^\circ_{\mathbf{t},a}$, while for 
$\mathcal{A}_{\mathbf{t},a}$
the sum  extends over all vertices including the
leaves. In other words, the measure $\mathcal{A}_{\mathbf{t},a}^\circ$
ignores the subtrees rooted at a leaf of $\mathbf{t}$ (which are trivial
trees consisting only of a root equipped with a scaled Dirac
measure). Let us take a moment to explain why we introduce the measure
$\mathcal{A}_{\mathbf{t},a}^\circ$. While $\mathcal{A}_{\mathbf{t},a}$
seems more natural, the measure $\mathcal{A}_{\mathbf{t},a}^\circ$ has
the advantage of putting no mass on  the set 
\[
\TT\times\R_+ = \left\{\tree \in \T\colon \, \m(\tree)
  = 0 \text{ or } \H(\tree) = 0\right\} \times \R_+.
\]  
This will be  useful as we are  interested in sums of  the form \eqref{A
  general}  where   the  function  $f$   may  blow  up   on  $\TT\times
\R_+$.
 We now give estimates for the distances between
the three measures
$\mathcal{A}_{\mathbf{t},a}^\circ$, $\mathcal{A}_{\mathbf{t},a}$ and
$\Psi_{a\mathbf{t}}$,  on $\T\times \R_+$, which are associated with the  discrete tree $\mathbf{t}$ and  $a >0$.

\begin{lemma}\label{distance between measures}
	Let $\mathbf{t}$ be a discrete tree and let $a > 0$.  We have
	\begin{align}
	d_{\mathrm{BL}}\left(\Psi_{a\mathbf{t}},\mathcal{A}_{\mathbf{t},a}\right) 
&\leqslant a\left( \frac{3}{4}\mathcal{A}_{\mathbf{t},a}(1) + 1
  \right),
\label{distance psi a}\\ 
d_{\mathrm{TV}}(\mathcal{A}_{\mathbf{t},a},
      \mathcal{A}^\circ_{\mathbf{t},a}) 
&\leqslant \frac{1}{2}a. 
\label{distance a a}
	\end{align}
\end{lemma}
\begin{proof}
	Let $f \in \cc_b( \T\times \R_+ )$ be Lipschitz.  Recall that
    $\tree = a\mathbf{t}$ is the real tree associated
    with $\mathbf{t}$, rescaled so that all edges have length $a$ and
    equipped with the uniform probability measure on the set of vertices
    whose height is an integer multiple of $a$. Recall also that for  $v
    \in\mathbf{t}$, $av$ denotes the corresponding vertex in $\tree =
    a\mathbf{t}$. In particular, $H(av) = aH(v)$, where $H(av)$ is the
    height of $av$ in the real tree $ a\mathbf{t}$ and $H(v)$ is the height of
    $v$ in the discrete tree $\mathbf{t}$. Thus, we have
	\begin{align*}
\Psi_{\tree}(f)
	= \frac{1}{|\mathbf{t}|} \sum_{v \in \mathbf{t}}
      \int_0^{H(av)}f(\tree_{r\vir av},r) \, \dd r  
&= \frac{1}{|\mathbf{t}|} \sum_{v \in \mathbf{t}}
  \int_0^{aH(v)}f(\tree_{r\vir av},r) \, \dd r \\ 
&= \frac{a}{|\mathbf{t}|} \sum_{v \in \mathbf{t}} \sum_{k=1}^{H(v)}
  \int_{k-1}^k f\left( \tree_{ar\vir av},ar\right) \, \dd r. 
	\end{align*}
	On the other hand, note that for every $1 \leqslant k \leqslant
    H(v)$, we have $\tree_{ak\vir av} = \tree_{aw}$ where $w \in \mathbf{t}$
    is the unique ancestor of $v$ with height $k$. Thus, we have
\[
\sum_{v \in \mathbf{t}} \sum_{k=1}^{H(v)} f\left( \tree_{ak\vir av},ak
	\right) 
=\sum_{v \in \mathbf{t}} 
\sum_{{w \preccurlyeq
  v}\atop{w\ne\emptyset}} f\left( \tree_{aw},aH(w)\right)  
=\sum_{w \neq \emptyset}|\mathbf{t}_w| f\left( \tree_{aw}, aH(w)\right) 
= \frac{|\mathbf{t}|}{a}\mathcal{A}_{\mathbf{t},a}(f) - |\mathbf{t}|f\left( \tree,0\right).
\]
	Therefore, we deduce that
\begin{align}\label{int dist}
	\left|\Psi_{\tree}(f) - \mathcal{A}_{\mathbf{t},a}(f) \right| 
&\leqslant \frac{a}{|\mathbf{t}|} \sum_{v \in \mathbf{t}}
  \sum_{k=1}^{H(v)} \int_{k-1}^{k} \left| f \left( \tree_{ar\vir av},ar
  \right) - f\left( \tree_{ak\vir av},ak\right) \right|\, \dd r +
  a\norm{f}_\infty\notag\\ 
	&\leqslant\frac{a}{|\mathbf{t}|} \sum_{v \in \mathbf{t}}
      \sum_{k=1}^{H(v)} \int_{k-1}^{k}
      \norm{f}_{\mathrm{L}}\big(\ghp\left( \tree_{ar\vir av} ,
      \tree_{ak\vir av}\right)+a(k-r)\big)\, \dd r + a\norm{f}_\infty. 
\end{align}
	Since for $k-1<r \leqslant k$, the tree $\tree_{ar\vir av}$ is
    obtained by grafting $\tree_{ak\vir av}$ on top of a branch of height
    $a(k-r)$ and no mass, it is straightforward to check that
    $\ghp\left( \tree_{ar\vir av} , \tree_{ak\vir av}\right) \leqslant
    a(k-r)/2$. It follows that 
\[
	\left|\Psi_{\tree}(f) - \mathcal{A}_{\mathbf{t},a}(f) \right| 
\leqslant \frac{a}{|\mathbf{t}|} \sum_{v \in \mathbf{t}}
\sum_{k=1}^{H(v)} \frac{3a}{4} \norm{f}_{\mathrm{L}}+ a\norm{f}_\infty 
\leqslant\frac{3a}{4}\norm{f}_{\mathrm{L}} \mathcal{A}_{\mathbf{t},a}(1)+ a \norm{f}_\infty.
\]
By definition of the distance $d_{\mathrm{BL}}$, we deduce that
	\[d_{\mathrm{BL}}\left(\Psi_{\tree},\mathcal{A}_{\mathbf{t},a}\right) \leqslant a\left( \frac{3}{4}\mathcal{A}_{\mathbf{t},a}(1) + 1 \right).\]

	Next, let $f \in \cb_b(\T\times\R_+)$.  We have
\begin{align*}
	\left| \mathcal{A}_{\mathbf{t},a}(f) -
      \mathcal{A}_{\mathbf{t},a}^\circ(f) \right|
&= \frac{a}{|\mathbf{t}|}\left|\sum_{w \in
  \operatorname{Lf}(\mathbf{t})} |\mathbf{t}_{w}|
  f\left(\tree_{a{w}},aH(w)\right)\right| \leqslant
  \frac{a}{|\mathbf{t}|} \left| \operatorname{Lf}(\mathbf{t})
  \right|\norm{f}_\infty \leqslant  a \norm{f}_\infty. 
\end{align*}
Taking the supremum over all   $f\in \cb_b( \T\times\R_+)$ such
    that $\norm{f}_\infty \leqslant 1$ yields $d_{\mathrm{TV}}\left( \mathcal{A}_{\mathbf{t},a}, \mathcal{A}_{\mathbf{t},a}^\circ \right) \leqslant \frac{1}{2}a$.
\end{proof}

We now restate the convergence of Corollary \ref{convergence for
  continuous functions} in terms of the discrete trees $\rddtree^n$. To
avoid cumbersome notations, we write
\[
\boxed{\mathcal{A}_n^\circ = \mathcal{A}_{\rddtree^n,b_n/n}^\circ}
\quad\text{and}\quad
\boxed{\mathcal{A}_n = \mathcal{A}_{\rddtree^n,b_n/n}}.
\]
Recall that for a discrete tree $\mathbf{t}$,  $w\in \mathbf{t}$ and 
$a>0$, we have  that  $\H(a \mathbf{t}_w)=a \H(\mathbf{t}_w)$ and
$\m(a\mathbf{t}_w)=|\mathbf{t}_w|/|\mathbf{t}|$. We shall also consider the following variant of the measure
$\mathcal{A}_n^\circ$ for functions depending only on the mass and
height: for every measurable function $f$ belonging to $ \cb_+( [0,1]\times
\R_+)$,
\begin{equation}
\label{mass+height measure}
\boxed{\costj_n(f) = \frac{b_n}{n^2} \sum_{w \in
    \rddtree^{n,\circ}}|\rddtree_w^n| f\left(\frac{|\rddtree_w^n|}{n},
    \frac{b_n}{n}\H(\rddtree_w^n)\right)}.
\end{equation}

We have the following upper bound of their total mass. 
\begin{lemma}
   \label{lem:bound-An(1)}
We have:
	\begin{equation}
\label{eq:mass-A_n}
\mathcal{A}_n^\circ(1)\leq  \frac{b_n}{n}
     \H(\rddtree^n)
\quad\text{and}\quad
\mathcal{A}_n(1)  \leqslant \frac{b_n}{n}
      \left(\H(\rddtree^n)+1\right).
	\end{equation}
\end{lemma}
\begin{proof}
   The proof is elementary as 
\begin{align*}
\mathcal{A}_n^\circ(1)
&= \frac{b_n}{n^2}\sum_{w \in
        \rddtree^{n,\circ}} |\rddtree^n_{w}| = \frac{b_n}{n^2}\sum_{w \in
        \rddtree^{n,\circ}} \sum_{w \preccurlyeq v} 1 
\leq \frac{b_n}{n^2}\sum_{v \in
        \rddtree^{n}} \H(\rddtree^n)\leq   \frac{b_n}{n}
     \H(\rddtree^n),\\
\mathcal{A}_n(1) 
& =\frac{b_n}{n^2}\sum_{w \in
        \rddtree^n} |\rddtree^n_{w}| = \mathcal{A}_n^\circ(1) +
      \frac{b_n}{n^2}  \left| \operatorname{Lf}(\mathbf{t}) \right|
\leq   \frac{b_n}{n}
      \left(\H(\rddtree^n)+1\right). 
\end{align*}
\end{proof}
We have the following convergence of  $\mathcal{A}_n^\circ$ as $n$ goes
to infinity. 
\begin{corollary}\label{convergence for continuous functions discrete}
  Assume  that   $\xi$  satisfies   \ref{xi1}  and  \ref{xi2}   and  let
  $\rddtree^n$   be  a   BGW($\xi$)   tree  conditioned   to  have   $n$
  vertices. Then for every $f\in \cc_b( \T \times \R_+)$, we have
  the convergence in distribution and of all positive moments
\begin{equation}
\label{convergence for continuous functions discrete, general}
\mathcal{A}_n^\circ(f)=\frac{b_n}{n^{2}} \sum_{w \in
  \rddtree^{n,\circ}}|\rddtree^n_{w}|f\left(\frac{b_n}{n}
  \rddtree^n_{w},\frac{b_n}{n}H(w)\right)\xrightarrow[n  
  \to \infty]{(d) + \mathrm{moments}} \Psi_{\rdtree}(f), 
\end{equation}
	where $\rdtree$ is the stable Lévy tree with branching mechanism $\psi(\lambda) = \kappa \lambda^\gamma$.
	In particular, for every   $f\in \cc_b( [0,1] \times \R_+)$, we have
	\begin{equation}
\label{convergence for continuous functions discrete, mass+height}
\costj_n(f) =\frac{b_n}{n^{2}} \sum_{w \in
  \rddtree^{n,\circ}}|\rddtree^n_{w}|f\left(\frac{|\rddtree^n_{w}|}{n},\frac{b_n}{n}\H(\rddtree^n_{w})\right)\xrightarrow[n
\to \infty]{(d) + \mathrm{moments}} \psij_{\rdtree}(f). 
	\end{equation}
\end{corollary}

\begin{remark}
	By \eqref{distance a a}, we have that a.s. and in $L^1$
	\[d_{\mathrm{TV}}\left(\mathcal{A}_n,\mathcal{A}_n^\circ \right) \xrightarrow[n\to \infty]{} 0.\]
	In particular, the convergences of Corollary \ref{convergence for continuous functions discrete} still hold if we sum over $\rddtree^n$ instead of $\rddtree^{n,\circ}$.
\end{remark}
\begin{remark}
\label{rem:polya}
Another model  of random  trees is  the class of  Pólya trees  which are
random  uniform   unordered  trees.    In  \cite{panagiotou2018scaling},
Panagiotou and Stufler show that the scaling limit of Pólya trees is the
Brownian  tree  and  that  the  sub-exponential  tail  bounds  of  Lemma
\ref{subgaussian}  hold in  this case  with $\alpha  = \beta  = 2$.  Let
$\Omega\subset \N$ be such that $\Omega\cap \{0, 1\} \neq \Omega$ and
let $\mathsf{T}^n$  denote the  uniform random  unordered tree  with $n$
vertices and vertex outdegree in $\Omega$.  Then there exists a finite constant
$c_\Omega  >0$  such  that  $(c_\Omega/\sqrt{n})\mathsf{T}^n\label{key}$
converges in distribution to the  Brownian tree $\rdtree$ with branching
mechanism $\psi(\lambda)  = 2\lambda^2$.  Thus, the  result of Corollary
\ref{convergence   for   continuous   functions  discrete}   holds   for
$\mathsf{T}^n$ and the proof is exactly the same as in the BGW case: for
every $f\in \cc_b(\T \times \R_+)$,
\[
\frac{c_\Omega}{n^{3/2}} \sum_{w \in
  \mathsf{T}^{n,\circ}}|\mathsf{T}^n_{w}|f\left(\frac{c_\Omega}{\sqrt{n}}\mathsf{T}^n_{w},\frac{c_\Omega}{\sqrt{n}}H(w)\right)\xrightarrow[n
\to \infty]{(d) + \mathrm{moments}} 
\Psi_{\rdtree}(f).
\]
\end{remark}
\begin{proof}[Proof of Corollary \ref{convergence for continuous functions discrete}]
	Denote by $\rdtree^n =(b_n/n)\rddtree^n$ the real tree associated
    with $\rddtree^n$ rescaled so that all edges have length $b_n/n$ and
    equipped with the uniform probability measure on the set of vertices
    whose height is an integer multiple of $b_n/n$. By Lemma
    \ref{distance between measures}, we have 
\[
d_{\mathrm{BL}}\left(\Psi_{\rdtree^n},\mathcal{A}_n^\circ\right)
\leqslant
d_{\mathrm{BL}}\left(\Psi_{\rdtree^n},\mathcal{A}_n\right)
+ 2d_{\mathrm{TV}}(\mathcal{A}_n, \mathcal{A}_n^\circ) 
\leqslant
\frac{b_n}{n}\left(\frac{3}{4} \mathcal{A}_n(1) + 2\right).
\]
Thanks to \eqref{eq:mass-A_n} and Lemma \ref{height finite moments}, we
have that $M =\sup_{n \in \supp} \ex{\mathcal{A}_n(1)} $ is finite. 
	It follows that
\[
\limsup_{n \to
  \infty}\ex{d_{\mathrm{BL}}\left(\Psi_{\rdtree^n},\mathcal{A}^\circ_n\right)}
\leqslant \lim_{n\to \infty}\frac{b_n}{n}\left( \frac{3M}{4} + 2\right)
= 0.
\]
	Thus, using that $\Psi_{\rdtree^n} \cvlaw \Psi_\rdtree$ in
    $\mathcal{M}(\T\times \R_+)$ by Corollary \ref{convergence for
      continuous functions}, Slutsky's lemma yields the convergence in
    distribution $\mathcal{A}^\circ_n \cvlaw \Psi_\rdtree$ in
    $\M(\T\times \R_+)$ which proves \eqref{convergence for continuous
      functions discrete, general}. 
	
	Let $f\in \cc_b( \T\times\R_+)$. Using Skorokhod's representation
    theorem, we may assume that the convergence \eqref{convergence for
      continuous functions discrete, general} holds almost surely. To
    prove the convergence of positive moments, it suffices to show that
    the family $(\mathcal{A}_n^\circ(f), \,{n \in \supp})$ is bounded in
    $L^p$ for every $p \in [1,\infty)$. This is the case as  by \eqref{eq:mass-A_n},
    we have $\mathcal{A}_n^\circ(f)\leqslant \norm{f}_\infty
\mathcal{A}_n^\circ(1)\leqslant \norm{f}_\infty \frac{b_n}{n}
\H(\rddtree^n)$,
and  the family $(\frac{b_n}{n} \H(\rddtree^n) , \,{n \in \supp})$ is bounded in $L^p$ for every $p \in [1,\infty)$ by Lemma \ref{height finite moments}. This completes the proof.
\end{proof}

	The Gromov-Hausdorff-Prokhorov convergence \eqref{BGW trees
      convergence} allowed us to derive an invariance principle
    \eqref{convergence for continuous functions discrete, general} for a
    certain class of additive functionals on BGW trees, namely those
    associated with real-valued continuous bounded functions $f$ defined on $\T
    \times\R_+$. In the sequel, we will be looking at a 
    similar invariance principle when $f$ blows up on $\TT\times\R_+$. It
    is not surprising  that the Gromov-Hausdorff-Prokhorov convergence alone
    does not allow us to say anything about the convergence of
    $\Psi_{\rdtree^n}(f)$ in this case as the next remark illustrates.

\begin{remark}\label{non generic phase transition}
  Let $\rddtree^{\nn}$ be a Catalan tree with $n$ vertices, where
  $n\in \supp = 2\N+1$. In other
  words,  $\rddtree^{\nn}$ is  uniformly distributed  among the  set of
  full binary ordered trees with $\nn$ vertices, which corresponds to a
  BGW($\xi$)          tree         with     
  $\pr{\xi  =  0}  = \pr{\xi  =  2}  =  1/2$  conditioned to  have  size
  $\nn$. Notice  that $\xi$  has finite variance  $\sigmaxi^2 =  1$. Take
  $b_\nn =   \sqrt{\nn} /2 $ so  that by  \eqref{BGW trees
    convergence},    $\rdtree^{\nn}=(1/2\sqrt{\nn})    \rddtree^{\nn}$
  converges in  distribution in  $\T$ to  the Brownian  continuum random
  tree   $\rdtree$  with   branching   mechanism   $\psi(\lambda)  =   2
  \lambda^2$.
 In    fact,   it    is    well   known,    see   \emph{e.g.}    \cite[Theorem
 7.9]{pitman2006combinatorial},  that  there  is  a  representation  of
 $\rdtree^{\nn}$  such that  the  almost sure  convergence holds.
  Denote  by  $\rdtree^{\nn}_{\varepsilon}$  the   real  tree  obtained  from
  $\rdtree^{\nn}$   by  stretching   the   leaves  by   a  distance   of
  $\varepsilon\geq 0$  and equip it  with the uniform probability  measure on
  the set of branching points and leaves. Fix $0 < \alpha < 1/2$ and set
  $\varepsilon_\nn  = \nn^{-\alpha}$.  It is  clear from  this construction  that
  $\rdtree^{\nn}_{\varepsilon_\nn}$ is a $\T$-valued  random variable and that
  a.s.
\[ 
\ghp\left(\rdtree^{\nn}_{\varepsilon_\nn}, \rdtree^{\nn}\right)
    \leqslant \varepsilon_\nn.
\]
	So it follows that $\rdtree^{\nn}_{\varepsilon_\nn}$ converges to $\rdtree$
    a.s. in the sense of the Gromov-Hausdorff-Prokhorov distance. 
We consider $f(\tree, r)=\m(\tree)^{-\alpha}$ and if $\nu\in \cm(\T\times \R_+)$ we write
$\nu(x^{-\alpha})$ for $\nu(f)$. 
According to \cite[Theorem 3.1]{delmas2018},  we have the following
a.s. convergence 
$\mathcal{A}_{\nn}(x^{-\alpha}) \xrightarrow[n\to\infty]{}
\Psi_{\rdtree}(x^{-\alpha})$. 
	In conjunction with the identity $\Psi_{\rdtree^\nn}(x^{-\alpha}) =
    \mathcal{A}_{\nn}(x^{-\alpha} ) - 1/(2\sqrt{\nn})$ this proves the a.s. convergence
\[
\Psi_{\rdtree^{\nn}}(x^{-\alpha}) \xrightarrow[n\to
\infty]{}\Psi_\rdtree(x^{-\alpha}).
\]
	On the other hand, we have
\[
\Psi_{\rdtree^{\nn}_{\varepsilon_\nn}}(x^{-\alpha}) -
\Psi_{\rdtree^{\nn}}(x^{-\alpha})  
= \frac{1}{|\rddtree^{\nn}|}
\sum_{w\in \operatorname{Lf}(\rddtree^{\nn})}
\int_{(2\sqrt{\nn})^{-1} H({w})}^{(2\sqrt{\nn})^{-1}  H({w}) + \varepsilon_\nn}
\left(\frac{\left|\rddtree^{\nn}_w\right|}{\left|\rddtree^{\nn}\right|}\right)^{-\alpha}\,
\dd r = \frac{n+1}{2} \nn^{\alpha - 1} 
\varepsilon_\nn 
\]
	since $|\rddtree^{\nn}| = \nn$ and
    $|\operatorname{Lf}(\rddtree^{\nn})| = (n+1)/2$. Thus, we
    get 
\[
\Psi_{\rdtree^{\nn}_{\varepsilon_\nn}}(x^{-\alpha}) -
    \Psi_{\rdtree^{\nn}}(x^{-\alpha}) \xrightarrow[n\to\infty]{}
    \frac{1}{2}\cdot
\]  

	In conclusion, even though we have the a.s. convergence
    $\rdtree^{\nn}_{\varepsilon_\nn}$ towards $ \rdtree$ in $\T$,
    $\Psi_{\rdtree^{\nn}_{\varepsilon_\nn}}(x^{-\alpha})$ does not converge to $
    \Psi_{\rdtree}(x^{-\alpha})$ for $\alpha\in (0, 1/2)$.
	This proves that the continuity of $\Psi_{\tree}(f)$ in $\tree$ when
    $f$ blows up on $\TT$, which has been  observed in
    \cite{delmas2018}, is indeed specific to BGW trees. 
\end{remark}

\section{Technical lemmas}\label{Sect technical}
In this section, we gather some technical results that will be used later.
The next lemma,  which gives sufficient conditions for boundedness in
$L^1$ of functionals of the mass and height on BGW trees, will be a key
ingredient in proving our convergence results. Recall that $\rddtree$ is
a BGW($\xi$) tree and $\rddtree^n$ is a BGW($\xi$) conditioned to have
$n$ vertices. Recall from \eqref{mass+height measure} the definition of the measure $\costj_n$ and notice that $ \costj_n( [0, 1]\times
\R_+ \setminus (0, 1]\times \R_+^*)=0$. For this reason, we also see $\costj_n$ as a measure on $(0,1]\times \R_+^*$.
By convention, we write
$\costj_n(g(x) h(u))$ for $\costj_n(f)$ where $f(x,u) = g(x)h(u)$, and
we see $g$ as a function of the mass and $h$ as a function of the
height. 

\begin{lemma}\label{mass+height bounded in L1}
	Assume that $\xi$ satisfies \ref{xi1} and \ref{xi3}. Suppose that
    $f\in \cb_+( (0,1] \times \R_+^*)$ satisfies one of the
    following assumptions:  
	\begin{enumerate}[label = (\roman*),leftmargin=*]
		\item $f$ is of the form $f(x,u) = g(x) u^\beta$ or $f(x,u) =
          x^\alpha h(u)$ where $\alpha, \beta \in \real$ and $g,h$ are
          nonincreasing and 
\begin{equation}
\label{int test bounded}
\int_0 f(x^{\gamma/(\gamma-1)},x) \, \dd x < \infty. 
\end{equation}
		\item $f(x,u) = g(x)\e^{u^\eta}\ind_{ [1,\infty)}(u)$ where
          $\eta \in (0,\gamma)$ and $g \in \cb_+( (0,1] )$ is
          nonincreasing and satisfies $\int_0 g(x) \e^{-x^{-r_0}} \,
          \dd x < \infty$ for some $r_0 \in (0,\gamma
          -1)$. 
	\end{enumerate}
	Then, we have
	\begin{equation*}\sup_{n \in \supp} \ex{\costj_n(f)} < \infty.
	\end{equation*}
\end{lemma}

\begin{proof}[Proof of Lemma \ref{mass+height bounded in L1}]
  Here $c$, $C$  and $M$ denote positive finite constants  that may vary
  from  expression  to  expression  (but are  independent  of  $n$ and $x$). Let
  $n \in \supp$ so  that
  $\pr{S_n  =  n-1}  >0$.  Observe  that
  $w \in \rddtree^{n,\circ}$  if and only if $|\rddtree^{n}_w|  > 1$ and
  that the root $\emptyset$ is the only vertex in $\rddtree^n$ such that
  $|\rddtree^{n}_w|   =  n$.   Thus,  for   every  $f\in \cb_+(
  [0,1]\times\R_+)$, we have the decomposition 
\begin{align*}
\ex{\costj_n(f)}
&=	\frac{b_n}{n^2}\ex{\sum_{w \in \rddtree^{n,\circ}} |\rddtree^n_{w}|
      f\left( \frac{|\rddtree^n_{w}|}{n}, \frac{b_n}{n}
        \H(\rddtree^n_{w})\right)} \\
&= \frac{b_n}{n^2}\ex{\sum_{w \in \rddtree^n}\ind_{\{1<|\rddtree^n_{w}|
    <n\}} |\rddtree^n_{w}| f\left( \frac{|\rddtree^n_{w}|}{n},
    \frac{b_n}{n} \H(\rddtree^n_{w})\right)} +
\frac{b_n}{n}\ex{f\left(1,\frac{b_n}{n}\H(\rddtree^n)\right)}. 
\end{align*}
	By \cite[Lemma 5.1]{janson2016asymptotic}, we have
	\begin{multline} \frac{b_n}{n^2}\ex{\sum_{w \in \rddtree^n}\ind_{\{1<|\rddtree^n_{w}| <n\}} |\rddtree^n_{w}| f\left( \frac{|\rddtree^n_{w}|}{n}, \frac{b_n}{n} \H(\rddtree^n_{w})\right)} 
	\\= \frac{b_n}{n} \sum_{k=1}^n \frac{\pr{S_k = k-1} \pr{S_{n-k} = n-k}}{\pr{S_n = n-1}} \ex{ f\left(\frac{k}{n}, \frac{b_n}{n} \H(\rddtree^k)\right)}\ind_{\{1<k <n\}}, \label{decomp 2}
	\end{multline}
	where by convention the summand is zero for $k \notin \supp$. Using
    Lemma \ref{local limit} and \eqref{b_n bounded}, we get for every $n
    \in \supp$ and every $1 < k < n$
	\[b_n\frac{\pr{S_k = k-1} \pr{S_{n-k} = n-k}}{\pr{S_n = n-1}} \leqslant C \frac{b_n^2}{b_k b_{n-k}}\leqslant C \left(\frac{n^2}{k(n-k)}\right)^{1/\gamma}.\]
	We deduce that
	\begin{align}
 \label{decomposition} 
	\frac{b_n}{n^2}\ex{\sum_{w \in \rddtree^{n,\circ}} |\rddtree^n_{w}|
      f\left( \frac{|\rddtree^n_{w}|}{n}, \frac{b_n}{n}
      \H(\rddtree^n_{w})\right)} 
&\leqslant \frac{C}{n}\sum_{k=1}^n g_n(k)+
  \frac{b_n}{n}\ex{f\left(1,\frac{b_n}{n}\H(\rddtree^n)\right)} \notag
      \\ 
	&= C\int_0^1 g_n(\lceil nx\rceil)\, \dd x +
      \frac{b_n}{n}\ex{f\left(1,\frac{b_n}{n}\H(\rddtree^n)\right)}, 
	\end{align}
	where we set
	\begin{equation}\label{g_n}g_n(k) =
      \left(\frac{n^2}{k(n-k)}\right)^{1/\gamma}
      \ex{f\left(\frac{k}{n}, \frac{b_n}{n}
          \H(\rddtree^k)\right)}\ind_{\{1<k <n\}} \quad  \text{for all }
      k \in \supp, 
	\end{equation}
	and $g_n(k) = 0$ for $k \notin \supp$. We will constantly make use of the following inequality
	\begin{equation}\label{b_n b_k}
	c \left(\frac{k}{n}\right)^{1-1/\gamma} \leqslant \frac{b_n}{n}
    \frac{k}{b_k} \leqslant C \left(\frac{k}{n}\right)^{1-1/\gamma}
    \quad \text{for all } 1\leqslant k \leqslant n,
	\end{equation}
	which follows easily from \eqref{b_n bounded}.
	
\noindent
	\textbf{First  case.}	Assume (i). First, we consider the case
    $f(x,u) = g(x)u^\beta$.  
	Since $b_n/n \to 0$, we deduce from Lemma \ref{height finite moments} that
\begin{equation}\label{integral test, part 3}
\lim_{n\to \infty}\frac{b_n}{n}\ex{f\left(1,\frac{b_n}{n}\H(\rddtree^n)\right)}
=	g(1)\lim_{n\to \infty}\frac{b_n}{n}\ex{\left(\frac{b_n}{n}
        \H(\rddtree^n) \right)^\beta \, } =0. 
\end{equation}
For every $1/n < x \leqslant (n-1)/n$, it holds that $x \leqslant \lceil
nx \rceil /n \leqslant 2x$ and $n - \lceil nx \rceil \geqslant
n(1-x)/2$. Thus, for every $x \in (0,1)$, using Lemma \ref{height
      finite moments} for the last inequality, we have
	\begin{align*}
	g_n(\lceil nx \rceil) 
&\leqslant M x^{-1/\gamma} (1-x)^{-1/\gamma} g\left(\frac{\lceil nx
  \rceil}{n}\right) \ex{\left(\frac{b_n}{n} \H(\rddtree^{\lceil nx
  \rceil})\right)^\beta\, }\ind_{\{1 < nx \leqslant n-1\}} \\ 
&\leqslant M x^{-1/\gamma}(1-x)^{-1/\gamma} g(x) \left(\frac{b_n}{n}
      \frac{\lceil nx \rceil}{b_{\lceil nx \rceil}}\right)^\beta \sup_{k
      \in \supp}
      \ex{\left(\frac{b_k}{k}\H(\rddtree^k)\right)^\beta\,}\ind_{\{1 <
      nx \leqslant n-1\}}  \\ 
&\leqslant M x^{(\beta+1)(1-1/\gamma)-1}(1-x)^{-1/\gamma} g(x). 
	\end{align*}
	It follows that
	\begin{equation}
	\label{integral test, part 4} \int_0^1 g_n(\lceil nx \rceil)\, \dd x
    \leqslant M \, \int_0^1 g(x) x^{(\beta+1)(1-1/\gamma) - 1} (1-x)^{-1/\gamma} \, \dd x,
	\end{equation}
	where the right-hand side is finite by   \eqref{int test bounded} as $\gamma
    >1$. Combining \eqref{integral test, part 3} and \eqref{integral
      test, part 4}, it follows from \eqref{decomposition} that 
\[
\sup_{n \in \supp} \ex{\costj_n(f)}= \sup_{n \in \supp} \frac{b_n}{n^2}\ex{\sum_{w \in \rddtree^{n,\circ}}
  |\rddtree^n_{w}| f\left( \frac{|\rddtree^n_{w}|}{n}, \frac{b_n}{n}
    \H(\rddtree^n_{w})\right)} < \infty.
\]
	
	Next, we consider the case $f(x,u) = x^\alpha h(u)$. By Lemma
    \ref{subgaussian} and (i) from Remark \ref{remark subgaussian}, we
    have, for every $k \in \supp$, 
	\begin{equation}
\label{stochastic order}
\pr{\frac{b_k}{k}\H(\rddtree^k) \leqslant y} \leqslant 1 \wedge \left( C_0
      \exp\left(-c_0 y^{-\gamma/(\gamma-1)}\right)\right).
	\end{equation}
	Denoting by $Y$ a random variable whose cdf is given by the right-hand side and using \eqref{b_n b_k}, we get, for every $2 \leqslant k \leqslant n$,
	\begin{equation}\frac{b_n}{n}\H(\rddtree^k) \geqslant_{\mathrm{st}} \frac{b_n}{n}\frac{k}{b_k} Y\geqslant c\left(\frac k n \right)^{1-1/\gamma} Y,\label{ordre sto} 
	\end{equation}
	where $\geqslant_{\mathrm{st}}$ denotes the usual stochastic order. In particular, since $Y$ has density \[y \mapsto  C y^{-(2\gamma-1)/(\gamma-1)} \exp\left(-c_0 y^{-\gamma/(\gamma-1)}\right)\ind_{[0,a]}(y)\]
	for some $a >0$, the first inequality in \eqref{ordre sto} applied with $k = n$ gives, for every $n \in \supp$,
	\begin{equation}
\label{integral test, part 1}
\ex{h\left(\frac{b_n}{n}\H(\rddtree^n)\right)} \leqslant \ex{h(Y)}
\leqslant C \int_0^\infty h(y) \e^{-c_0 y^{-\gamma/(\gamma-1)}}\,
\frac{\dd y}{y^{(2\gamma-1)/(\gamma-1)}}\cdot 
	\end{equation}
	Note that the last integral is finite: indeed, since $h$ is nonincreasing, we have
	\begin{equation*}\int_1^\infty h(y) \e^{-c_0 y^{-\gamma/(\gamma-1)}} \frac{\dd y}{y^{(2\gamma-1)/(\gamma-1)}} \leqslant h(1) \int_1^\infty \frac{\dd y}{y^{(2\gamma-1)/(\gamma-1)}} < \infty, \label{finite integral 1}
	\end{equation*}
	and by \eqref{int test bounded}
	\begin{equation}\int_0^1 h(y) \e^{-c_0 y^{-\gamma/(\gamma-1)}} \frac{\dd y}{y^{(2\gamma-1)/(\gamma-1)}} \leqslant \sup_{0 <y \leqslant 1}\frac{\e^{-c_0 y^{-\gamma/(\gamma-1)}}}{y^{1+(\alpha +1)\gamma/(\gamma-1)}} \int_0^1 h(y) y^{\alpha \gamma/(\gamma-1)} \, \dd y < \infty.\label{finite integral 2}
	\end{equation}
Then, applying \eqref{ordre sto} with $k = \lceil nx \rceil$ and using
the fact that $h$ is nonincreasing, we get for every $x \in (0,1)$ 
	\begin{align*}g_n(\lceil nx \rceil)& \leqslant M x^{-1/\gamma}(1-x)^{-1/\gamma}\left(\frac{\lceil nx \rceil }{n} \right)^\alpha\ex{ h\left(\frac{b_n}{n}\H(\rddtree^{\lceil nx \rceil})\right)}\ind_{\{1 < nx \leqslant n-1\}}\\
	& \leqslant M x^{\alpha-1/\gamma}(1-x)^{-1/\gamma}\ex{h\left(c x^{1-1/\gamma}Y\right)}\\
	&\leqslant Mx^{\alpha-1/\gamma}(1-x)^{-1/\gamma} \int_0^a h\left(c x^{1-1/\gamma} y \right)\e^{-c_0y^{-\gamma/(\gamma-1)}} \, \frac{\dd y}{y^{(2\gamma-1)/(\gamma-1) }} \\
	&\leqslant Mx^{1+\alpha-1/\gamma}(1-x)^{-1/\gamma} \int_0^{acx^{1-1/\gamma}} h(u) \e^{-r xu^{-\gamma/(\gamma-1)}} \, \frac{\dd u}{u^{(2\gamma-1)/(\gamma-1)}},
	\end{align*}
	for some positive constant $r >0$, where in the last inequality
    we made the change of variable $u = c x^{1-1/\gamma} y$. Therefore
    we have 
	\begin{equation}\int_0^1 g_n\left(\lceil nx \rceil \right)\, \dd x \leqslant M \int_0^1 x^{1+\alpha-1/\gamma}(1-x)^{-1/\gamma} \, \dd x \int_0^{acx^{1-1/\gamma}} h(u)\e^{-rxu^{-\gamma/(\gamma-1)}} \,\frac{\dd u}{u^{(2\gamma-1)/(\gamma-1)}}\cdot \label{integral test, part 2}
	\end{equation}
	It remains to check that the last integral is finite. But, arguing
    as in     \eqref{finite integral 2} with $r$ instead of $c_0$, we have 
	\begin{multline*}\int_{1/2}^1 x^{1+\alpha-1/\gamma}(1-x)^{-1/\gamma} \, \dd x \int_0^{acx^{1-1/\gamma}} h(u)\e^{-rxu^{-\gamma/(\gamma-1)}} \,\frac{\dd u}{u^{(2\gamma-1)/(\gamma-1)}}\\ \leqslant M\int_{1/2}^1 (1-x)^{-1/\gamma} \, \dd x \int_0^{ac} h(u)\e^{-ru^{-\gamma/(\gamma-1)}/2} \,\frac{\dd u}{u^{(2\gamma-1)/(\gamma-1)}} < \infty.
	\end{multline*}
Let $\delta=\gamma/(\gamma-1)$. Making the change of variable $y =
xu^{-\delta}$ with $u$ fixed, we have, thanks to \eqref{int test bounded},
	\begin{multline*}
	\int_0^{1/2} x^{1+\alpha-1/\gamma}(1-x)^{-1/\gamma} \, \dd x
    \int_0^{acx^{1-1/\gamma}} h(u)\e^{-rxu^{-\delta}} \,\frac{\dd
      u}{u^{1+\delta}} \\ 
\begin{aligned}
&\leq  \int_{(ac)^{-\delta}}^\infty y^{1+\alpha - 1/\gamma} \e^{-ry} \, \dd
y \int_0^\infty h(u) u^{\alpha \delta} \ind_{\{yu^\delta \leqslant
  1/2\}} \, \dd u \\ 
&\leqslant \int_{(ac)^{-\delta}}^\infty y^{1+\alpha - 1/\gamma}
    \e^{-ry} \, \dd y \int_0^{ac} h(u) u^{\alpha \delta} \, \dd u <
    \infty. 
	\end{aligned}
	\end{multline*}
	The right-hand side of \eqref{integral test, part 1} and
    \eqref{integral test, part 2} being finite and $(b_n/n, \, n\geq 1)$
    being bounded, we deduce from \eqref{decomposition} that
\[
\sup_{n \in \supp} \ex{\costj_n(f)}=\sup_{n \in \supp} \frac{b_n}{n^2}\ex{\sum_{w \in \rddtree^{n,\circ}}
  |\rddtree^n_{w}| f\left( \frac{|\rddtree^n_{w}|}{n}, \frac{b_n}{n}
    \H(\rddtree^n_{w})\right)} < \infty. 
\]

\noindent
	\textbf{Second case.}  Assume (ii). Fix $\eta\in (0,\gamma)$ and
    set $h(u) = \e^{u^{\eta}}\ind_{\{u \geqslant 1\}}$. Choose $\beta
    \in (\eta, \gamma)$ such that $\beta(1-1/\gamma) >r_0$. By
    \eqref{subgaussian infinity} and \eqref{b_n b_k}, we have, for every
    $k \in \supp$ such that $2 \leqslant k \leqslant n$, 
	\begin{equation}\label{ordre sto 2}
	\frac{b_n}{n} \H(\rddtree^k) \leqslant_{\mathrm{st}}\frac{b_n}{n}\frac{k}{b_k}Z\leqslant C \left(\frac k n\right)^{1-1/\gamma}Z,
	\end{equation}
	where $Z$ has density $z \mapsto  M z^{\beta -1} \e^{-c_0 z^\beta}\ind_{[a,\infty)}(z)$ for some $a >0$.
	So,  we get for $x\in (0, 1)$
	\begin{align*}
g_n(\lceil n x \rceil)  
&\leqslant M x^{-1/\gamma} (1-x)^{-1/\gamma} g\left(\frac{\lceil nx
  \rceil}{n}\right) \ex{h\left(\frac{b_n}{n}\H(\rddtree^{\lceil n
  x\rceil})\right)} \\ 
	&\leqslant M x^{-1/\gamma} (1-x)^{-1/\gamma} g\left(x\right)
      \ex{h\left(Cx^{1-1/\gamma} Z\right)}\\ 
	&\leqslant  M x^{-1/\gamma} (1-x)^{-1/\gamma} g\left(x\right)
      \int_a^\infty h\left(Cx^{1-1/\gamma}z\right) z^{\beta- 1} \e^{-c_0
      z^\beta} \, \dd z\\ 
	&\leqslant  M x^{-1/\gamma} (1-x)^{-1/\gamma} g\left(x\right)
      \int_a^\infty  z^{\beta- 1} \e^{c_1z^\eta-c_0 z^\beta}
      \ind_{\lbrace Cx^{1-1/\gamma} z \geqslant 1\rbrace}\, \dd z, 
	\end{align*}
 where we used  \eqref{b_n b_k} for
    the first and second inequalities,  the monotonicity of $g$ and $h$
    for the second and the fact that $\left(C x^{1-1/\gamma} z\right)^\eta\leq  c_1
    z^\eta $	for  some finite constant $c_1 >0$ for the last.  Notice that if $r < c_0$, then the function $z \mapsto \e^{c_1 z^\eta - (c_0-r) z^\beta}$ is bounded on $\R_+$ as $\beta > \eta$. It follows that
	\begin{align}
	\int_0^1 g_n(\lceil nx \rceil) \, \dd x 
&\leqslant M \int_0^1 x^{-1/\gamma} (1-x)^{-1/\gamma} g(x)\, \dd x
  \int_0^\infty z^{\beta- 1} \e^{-r z^\beta} \ind_{\lbrace
  Cx^{1-1/\gamma} z \geqslant 1\rbrace}\, \dd z \notag\\ 
	&\leqslant M \int_0^1 x^{-1/\gamma}(1-x)^{-1/\gamma} \e^{-r
      C^{-\beta}x^{-\beta(1-1/\gamma)}}g(x) \, \dd x \notag\\ 
	&\leqslant M \int_0^1 (1-x)^{-1/\gamma} \e^{- x^{-r_0}}g(x) \, \dd
      x< \infty,
\label{exponential growth, part 1}
	\end{align}
	where in the last inequality we used that the function $x \mapsto
    x^{-1/\gamma} \e^{x^{-r_0}-rC^{-\beta}x^{-\beta(1-1/\gamma)}}$ is
    bounded on $(0,1]$ as $\beta(1-1/\gamma) > r_0$. On the other
    hand,  we have 
	\begin{align}\label{exponential growth, part 2}
	 \frac{b_n}{n} \ex{f\left(1,\frac{b_n}{n} \H(\rddtree^n)\right)} \leqslant \frac{b_n}{n} g(1)\ex{h(Z)}\leqslant M\frac{b_n}{n}\int_1^\infty z^{\beta -1}\e^{cz^\eta - c_0 z^\beta}\, \dd z \leqslant M,
	 \end{align}
	where we used the first inequality from \eqref{ordre sto 2} with $k=n$  and
    the fact that $h$ in nondecreasing for the first inequality and that
    $b_n/n$ converges to $0$ as $n \to \infty$ for the last.  Combining \eqref{exponential growth, part 1} and \eqref{exponential growth, part 2}, we deduce from \eqref{decomposition} that
\[
\sup_{n \in \supp} \ex{\costj_n(f)}= \sup_{n \in \supp} \frac{b_n}{n^2}\ex{\sum_{w \in \rddtree^{n,\circ}}
  |\rddtree^n_{w}| f\left( \frac{|\rddtree^n_{w}|}{n}, \frac{b_n}{n}
    \H(\rddtree^n_{w})\right)} < \infty.
\]
\end{proof}

As a consequence of the following lemma, we get that $(\costj_n(x^\alpha u^\beta),
\,\allowbreak n \in \supp)$ is bounded in $L^p$ for some $p >1$.

\begin{lemma}
\label{lem:Amh^p}
	Let  $\alpha, \beta \in \real$ such that $\gamma \alpha + (\gamma -
    1)(\beta+1) >0$. For  every  $p \geq 1$ such that $p(\gamma
    \alpha + (\gamma-1)\beta)>1-\gamma$ and $\delta \in \real$, we have:
\begin{equation}
\label{1}
	\sup_{n\in \supp}\ex{\left(\frac{b_n}{n}\H(\rddtree^n)\right)^\delta
      \costj_n(x^\alpha u^\beta)^p} < \infty.  
\end{equation}
\end{lemma}
\begin{proof}
	Set $M_n =\frac{b_n}{n}\H(\rddtree^n) $ for $n\in \supp$. 
Let $p_0, q_0 \in (1,\infty)$ such that $1/p_0 + 1/q_0 = 1$. By Hölder's
inequality and thanks to
    \eqref{eq:mass-A_n}, we have 
	\begin{equation}\label{3}
	\costj_n(x^\alpha u^\beta)^{p_0} 
\leqslant    M_n^{p_0/q_0}\costj_n(x^{p_0\alpha}u^{p_0\beta}). 
	\end{equation}
Assume that $p_0>p$ satisfies $p_0(\gamma
    \alpha + (\gamma-1)\beta)>1-\gamma$. Set $r=p_0/p$ and $s$  such
    that $1/r+1/s=1$. We deduce that
	\begin{align*}
	\ex{M_n^\delta \costj_n(x^\alpha u^\beta)^p} 
&= \ex{M_n^{\delta+p/q_0}M_n^{-p/q_0} \costj_n(x^\alpha u^\beta)^p} \\
	&\leqslant \ex{M_n^{s(\delta + p/q_0)}}^{1/s} \ex{M_n^{-p_0/q_0}
      \costj_n(x^\alpha u^\beta)^{p_0}}^{1/r} \\ 
	&\leqslant \ex{M_n^{s(\delta + p/q_0)}}^{1/s}
      \ex{\costj_n(x^{p_0\alpha}u^{p_0\beta})}^{1/r}, 
	\end{align*}
where we used Hölder's inequality for the first inequality and
\eqref{3} for the second. Since $p_0(\gamma
    \alpha + (\gamma-1)\beta)>1-\gamma$, the function $f(x,u) = x^{p_0\alpha}
    u ^{p_0\beta}$ satisfies assumption (i) of Lemma \ref{mass+height bounded
      in L1}. We deduce that  $\sup_{n\in\supp}
    \ex{\costj_n(x^{p_0\alpha}u^{p_0\beta})}  <  \infty$. 
Then use Lemma \ref{height finite moments} to get \eqref{1}.
\end{proof}

\section{Functionals of the mass and height on the stable Lévy
  tree}\label{Sect Levy measure} 
In this  section, our goal  is to study  the finiteness and  compute the
first moment  of the random variable  $\psij_\rdtree(f)$ where $\rdtree$
is the  stable Lévy tree and  $f$ is a measurable  function. Recall from
Section \ref{levy tree} that $H$ denotes the $\psi$-height process under
its excursion measure $\n$, $\sigma$ is the duration of an excursion and
$\H$ is its height.  Notice that $\sigma$  and $\H$ are the mass and the
height of  the tree  $\rdtree_H$ coded by  $H$. Furthermore,  the stable
Lévy tree $\rdtree$  (under $\P$) is the real tree  $\rdtree_H$ coded by
$H$,         see          Remark         \ref{rem:coding},         under
$\n^{(1)}[\bullet]= \n[\bullet \, |\, \sigma=1]$.

\subsection{On the fragmentation (on the skeleton)
 of  Lévy trees}
\label{sec:frag}
In this section only we consider a general  continuous height process $H$
under its excursion measure $\n$
 associated with a branching mechanism
$\psi(\lambda)= a\lambda + \beta (\lambda^2/2)+ \int
\pi(\dd r) (\e^{-\lambda r} -1 + \lambda r)$ with $a, \beta \geq 0$,
$\pi$ a $\sigma$-finite measure on $(0, \infty )$ such that $\int
\pi(\dd r) \, (r \wedge r^2)<\infty$ and such that
$\int^\infty  \dd \lambda/ \psi(\lambda)<\infty $. We refer to  \cite[Section~1]{duquesne2002random} for a complete presentation of the subject.

We will present a decomposition of a general Lévy tree
using Bismut's decomposition. 
Define the length and height of the excursion of $H$
    above level $r$ that straddles $s$ 
\begin{equation}
	\sigma_{r,s} = \int_0^\sigma \ind_{\lbrace m(s,t)\geqslant
      r\rbrace}\, \dd t 
=\ttt_{r,s}^+-\ttt_{r,s}^-
\quad \text{and}\quad 
\H_{r,s} = \sup_{t \in [\ttt_{r,s}^{\scriptscriptstyle
    -},\ttt_{r,s}^{\scriptscriptstyle +}]}H(t)-r, 
\end{equation}
	where $m(s,t) = \inf_{[s\wedge t , s \vee t]}H$ is the minimum of
    $H$ between times $s,t$ and $\ttt_{r,s}^{-} = \sup\lbrace t < s\colon
    \, H(t)=r\rbrace$ and $\ttt_{r,s}^{+} = \inf\lbrace t>s \colon \,
    H(t)=r\rbrace$ are the beginning and the end of the excursion of $H$
    above level $r$ that straddles time $s$, see Figure
    \ref{excursion}. Then, we consider $H^+_{r,s}=(H^+_{r,s}(t), t\geq
    0)$ the excursion of $H$ 
    above level $r$ that straddles $s$ defined as:
\[
H^+_{r,s}(t)=H\left((t+ \ttt_{r,s}^{-} ) \wedge \ttt_{r,s}^{+} \right) -r,
\]
and  $H^-_{r,s}=(H^-_{r,s}(t), t\geq
    0)$ the excursion of $H$ below defined as $H^-_{r,s}(t)=H(t)$ for
    $t\in [0, \ttt_{r,s}^{-} ]$ and 
$H^-_{r,s}(t+ \sigma_{r,s})$ for $t>\ttt_{r,s}^{-}$. Notice that the
duration and height of the excursion $H^+_{r,s}$ are given by
$\sigma^+_{r, s}=\sigma_{r, s}$ and $H_{r,s}$; that the duration of the excursion
$H^-_{r,s}$ is given by $\sigma^-_{r, s}=\sigma- \sigma_{r,s}$; and 
that 
\begin{equation}
   \label{eq:decomp-sigma}
\sigma= \sigma^+_{r, s}+\sigma^-_{r, s}. 
\end{equation}

	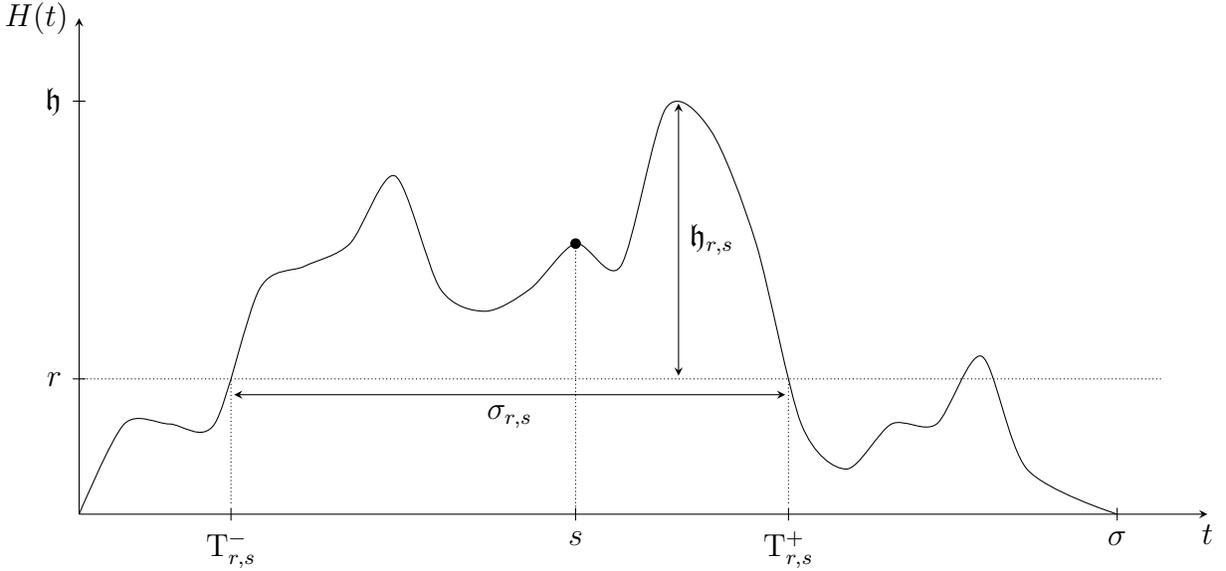
\begin{figure}[ht]
		\centering
		\begin{tikzpicture}[scale=0.6,>=stealth]
		\draw [->] (0,0) -- (25,0) node[below] {$t$};
		\draw [->] (0,0) -- (0,11) node[left] {$H(t)$};
		\draw [name path=graph] plot [smooth] coordinates {(0,0)  (1,2) (2,2) (3,2) (4,5) (5,5.5) (6,6) (7,7.5) (8,5) (9,4.5) (10,5) (11,6) (12,5.5) (13,9) (14,8.5) (15,6) (16,2) (17,1) (18,2) (19,2) (20,3.5) (21,1) (23,0)};
		\draw [thin] (11,-4pt) node[below] {$s$} -- (11,4pt);
		\draw [thin] (23,-4pt) node[below] {$\sigma$} -- (23,4pt);
		\draw ($(-4pt,9)+(0,4.5pt)$) node[left] {$\H$} -- ($(4pt,9)+(0,4.5pt)$);
		\draw (-4pt,3) node[left] {$r$} -- (4pt,3);
		\draw [densely dotted] (11,0) -- (11,6);
		\draw [densely dotted] (0,3) -- (24,3);
		\filldraw[fill=black] (11,6) circle [radius=3pt];
		\path [name path=hor] (0,3)--(25,3);
		\draw [name intersections={of=graph and hor, by={x,y}},<->] ($(x) + (2pt,-10pt)$) --($(y) + (-2pt,-10pt)$) node[midway, below] {$\sigma_{r,s}$};
		\draw [<->] ($(13,9)+(8pt,3pt)$) --  ($(13,3)+(8pt,2pt)$) node[midway,right] {$\H_{r,s}$};
		\node (minus) at (x|- 0,0){};
		\draw ($(minus)+(0,-4pt)$) node[below]{$\ttt_{r,s}^{-}$} --
        ($(minus)+(0,4pt)$); 
		\draw[densely dotted] (minus) -- (x);
		\node (plus) at (y|- 0,0){};
		\draw ($(plus)+(0,-4pt)$) node[below]{$\ttt_{r,s}^{+}$} --
        ($(plus)+(0,4pt)$); 
		\draw[densely dotted] (plus) -- (y);
		\end{tikzpicture}
		\caption{The duration $\sigma_{r,s}$ and the height $\H_{r,s}$ of the excursion of $H$ above level $r$ that straddles time $s$.}
		\label{excursion}
	\end{figure}

    Recall notations from Remark \ref{rem:coding}.    For
    $s\in [0,  \sigma]$ and $r\in  [0, H(s)]$, the  function $H^+_{r,s}$
    codes     the   subtree   $\rdtree_{r\vir s}:=(\rdtree_H)_{r\vir p(s)}$   and
    $H^-_{r,s}$          codes                  the          subtree
    $\rdtree_{r\vir s}^-:=(\rdtree_H\setminus             \rdtree_{r\vir
      s})\cup 
    \{x_{r,s}\}$,
    where  $x_{r,s}$ is  the ancestor  of $p(s)$,  the image  of $s$  on
    $\rdtree_H$, at  distance $r$  from the root  of $\rdtree_H$.   The next
    lemma  says that  when $s$  and $r$  are chosen  ``uniformly'' under
    $\n$, then  the random  trees $\rdtree_{r\vir s}$  and
    $\rdtree_{r\vir s}^-$
    are    independent    and    distributed    as    $\rdtree_H$    under
    $\n[\sigma\bullet]$. This  result is  a consequence of Bismut's
    decomposition of the excursion of the height process.
\begin{lemma}
   \label{lem:H+,H-}
	Let $H$ be a continuous height process associated with a general  branching mechanism
      under
its excursion measure $\n$. Then for every nonnegative measurable functions
$f_+$ and $f_-$ defined on $\cc_+(\R_+)$, we have:
\[
\n\left[ \int_0^\sigma \dd s \int_0^{H(s)} f_+(H^+_{r,s}) \, f_- (H^-_{r,s})
  \, \dd r\right]
= \n\left[ \sigma f_+(H)\right]\, \n\left[ \sigma f_-(H)\right].
\]
\end{lemma}



\begin{remark}
   \label{rem:voisin}
   Lemma  \ref{lem:H+,H-} allows  to recover directly 
   the
   distribution of the size of the two fragments given by the 
   fragmentation   measure  $q^{ske}(\dd s,\! \dd r)=2 \beta
   \sigma_{r,s}^{-1} \ind_{[0,
     H(s)]}(r) \, \dd s \, \dd r $  on   the  skeleton    in   \cite[Lemma
   5.1]{voisin}. The Brownian case ($\pi=0$ and $\beta>0$) appears
   already in \cite{ap:sac} and then in \cite{as:psf}. 
\end{remark}

\begin{proof}
We follow the proof of \cite[Lemma 3.4]{duquesne2005probabilistic} and use notations from \cite{duquesne2002random} on the càd-làg Markov process
process $\left(\rho_s, \eta_s; \, s\in [0, \sigma]\right)$ under
$\n$, which is an $\M(\R_+)^2$-valued process.  The  process
$(\rho, \eta)$ is a Markov process which allows to recover the (\emph{a priori}
non-Markovian) height process as a.s.  $[0,
H(t)]=\Supp(\rho_t)=\Supp(\eta_t)$. (The process $\rho$ is called the
exploration process associated with $H$ and is strong Markov.) 
Thanks to \cite[Proposition~3.1.3]{duquesne2002random}, we have that:
\begin{equation}
   \label{eq:rho-eta=M}
\n\left[\int_0 ^\sigma \dd s \, F( \rho_s, \eta_s)\right]
=\int \MM( \dd \mu, \! \dd \nu) \, F(\mu, \nu),
\end{equation}
where
$\MM=\int_0^\infty  \dd t \, \e^{-at} \, \MM_{[0, t]}$ and, for any
interval $I$,  $\MM_I$ is the law on $\M(\R_+)^2$ of the pair $(\mu_I,
\nu_I)$ defined by:
\begin{align*}
   \mu_I(f)
&= \int \cn (\dd r, \! \dd \ell, \! \dd x)\, \ind_I(r)\, x f(r) 
+ \beta \int_I \dd r \, f(r),\\
   \nu_I(f)
&= \int \cn (\dd r, \! \dd \ell, \! \dd x) \, \ind_I(r) (\ell- x) f(r) 
+ \beta \int_I \dd r \, f(r),
\end{align*}
with $\cn  (\dd r, \! \dd  \ell, \! \dd  x)$ a Poisson point  measure on
$(\R_+)^3$                         with                        intensity
$ \dd  r \,  \pi(\dd \ell)\,  \ind_{[0, \ell]} (x)\,  \dd x$. 
We write $\tilde \rho=(\rho, \eta)$ and $\tilde \eta=(\eta, \rho)$.
We recall
that the process $\left( \rho_s;    \,    s\in    [0,    \sigma]\right)$
is strong Markov under $\n$, see
\cite[Proposition~1.2.3]{duquesne2002random}, and the time reversal
property of $(\rho, \eta)$, see \cite[Corollary~3.1.6]{duquesne2002random},  that is 
$\left(\tilde \rho_{s};     \,    s\in     [0,
  \sigma]\right)$  and $\left(\tilde \eta_{(\sigma-s)-};     \,    s\in     [0,
  \sigma]\right)$ have the same distribution under $\n$.

For a measure $\mu$ on $\R_+$ and $u>0$ we define the measure
$\mu^{[u]}$, the measure $\mu$ erased up to level $u$ and shifted
by $u$, by $\mu^{[u]} (f)=\int f(r-u)\ind_{\{r>u\}} \,
\mu(\dd r)$ for $f\in \cb_+(\R_+)$. We write
 $\tilde \rho^{[u]}=(\rho^{[u]}, \eta^{[u]})$
and  similarly for $\tilde \eta$. 
 Let $F^\varepsilon_i$, for
$\varepsilon\in\{+, -\}$ and $i\in \{{\rm g}, {\rm d}\}$, be measurable nonnegative
functionals defined on the set of càd-làg $\M(\R_+)^2$-valued functions. We
shall compute:
\begin{multline*}
A= \n\Big[ \int_0^\sigma \dd s \int_0^{H(s)}\dd r\, 
F_{\rm d}^+ \left( \tilde \rho_{s+t}^{[r]}; t\in [0, \ttt^+_{r,s} -s])\right)
F_{\rm g}^+ \left( \tilde \eta_{(s-t)-}^{[r]}; t\in [0, \ttt^-_{r,s} -s]\right)\\
F_{\rm d}^- \left( \tilde \rho_{\ttt^+_{r,s}+t}; t\in
  [0,\sigma-  \ttt^+_{r,s}]\right) 
F_{\rm g}^- \left(\tilde \eta_{(\ttt^-_{r,s}-t)-}; t\in [0,
  \ttt^-_{r,s}]\right) 
\Big].
  \end{multline*}
We write
$\ind_{[0, r]} \tilde \rho=(\ind_{[0, r]} \rho, \ind_{[0, r]} \eta)$. Using the  Markov property of $\tilde \rho$ at time $s$, the time
reversal property,  again the  Markov property of $\tilde \rho$ at
time $s$, \eqref{eq:rho-eta=M}
and 
the transition kernel of $\tilde \rho$ given in
\cite[Proposition~3.1.2]{duquesne2002random}, we get that:
\[
A= \n\left[ \int_0^\sigma \dd s \int_0^{H(s)}\dd r\, 
G^+ \left( \tilde \rho_{s}^{[r]}\right)
G^- \left(\ind_{[0, r]}\tilde \rho_{s}\right)\right],
\]
for some measurable nonnegative functions $G^-$ and $G^+$ such  that
 for $\varepsilon\in \{+, -\}$
\begin{equation}
   \label{eq:MMG}
\MM[G^\varepsilon]=\n\left[\int_0^\sigma \dd s \, 
F_\text{d}^\varepsilon( \tilde \rho_{s+t}, t\in [0, \sigma-s])
F_\text{g}^\varepsilon( \tilde \rho_{(s-t)-}, t\in [0, s])
\right].
\end{equation}
 Then using
\eqref{eq:rho-eta=M} and the definition of $\MM$, we get, with $\tilde \mu=(\mu, \nu)$:
\begin{align*}
A
&= \int_0^\infty  \dd t \, \e^{-at} \int_0^t \dd r\,\, \MM_{[0, t]}(\dd
  \tilde \mu) \,  
G^+ \left( \tilde \mu^{[r]}\right)
G^- \left(\ind_{[0, r]}\tilde \mu\right)   \\
&= \int_0^\infty  \dd t \, \e^{-at} \int_0^t \dd r\,\, \MM_{[0, t-r]}[G^+]
  \, \MM_{[0, r]}[G^-]  \\
&= \left(\int_0^\infty  \dd r \, \e^{-ar} \, \MM_{[0, r]}[G^+]\right)
\left(\int_0^\infty  \dd r \, \e^{-ar} \, \MM_{[0, r]}[G^-]\right)\\
&= \MM[G^+]\, \MM[G^-],
\end{align*}
where we used the independence  property, that is $\MM_I*\MM_J=\MM_{I\cup
  J}$ when $I$ and $J$ are disjoint, for the second equality. 
We deduce from \eqref{eq:MMG} and the monotone class
theorem that for any measurable nonnegative
functionals $F^+$ and $F^-$ defined on the set of càd-làg
$\M(\R_+)^2$-valued functions, we have:
\begin{multline*}
   \n\left[\int_0^\sigma \dd s \int_0^{H(s)}\dd r\, 
F^+(\tilde \rho_{t+ T^-_{r, s}}; t\in [0, \sigma_{r,s}])
F^-(\tilde \rho_{t+ 
  \sigma_{r,s}\ind_{\{t>T^-_{r,s}\}}}; t\in [0, \sigma-\sigma_{r,s}]) 
\right]\\
\begin{aligned}
&= \n\left[\int_0^\sigma \dd s\, F^+(\tilde \rho_t; t\in [0, \sigma])\right]
 \n\left[\int_0^\sigma \dd s\, F^-(\tilde \rho_t; t\in [0,
  \sigma])\right].\\
&= \n\left[\sigma F^+(\tilde \rho_t; t\in [0, \sigma])\right]
 \n\left[\sigma  F^-(\tilde \rho_t; t\in [0,
  \sigma])\right].   
\end{aligned}
\end{multline*}
Then use that $H$  is a measurable functional of the
exploration process $\tilde \rho$ to conclude. 
\end{proof}

\subsection{First moment of $\Psi_\rdtree$}
We start with the main result of this section which gives the first
moment of functionals of the stable Lévy tree. 
Recall that $\rdtree_H$ is the real tree coded by $H$, see Remark
\ref{rem:coding}.

\begin{proposition}
   \label{prop:EYf}
	Let $\rdtree$ be the stable Lévy tree with branching mechanism
    $\psi(\lambda) = \kappa \lambda^{\gamma}$ where $\kappa >0$ and
    $\gamma \in (1,2]$. Let $f\in \cb_+(\T)$, and set $\tilde f(T,r)=f(T)$ for $T\in \T$ and
$r\in \R_+$. We have:
\begin{equation}
   \label{eq:EYf}
\E\left[\Psi_{\rdtree}(\tilde f)\right] =
 \n    \left[    \sigma     (1-\sigma)^{-1/\gamma}    f
  (\rdtree_H)\ind_{\{\sigma<1\}}\right].
\end{equation}
\end{proposition}

\begin{proof}
Let $f\in \cb_+(\T)$ and set $\tilde f(T,r)=f(T)$ for $T\in \T$ and
$r\in \R_+$. Using notations from Section \ref{sec:frag}, we have 
$\Psi_{\rdtree_H}(\tilde f)=\int_0^\sigma\dd s \int_0^{H(s)}
f(\rdtree_{H^+_{r,s}})\, \dd r$. Thus, on the one hand, we get 
for $\lambda>0$
\begin{align}
   \n\left[\expp{-\lambda \sigma} \Psi_{\rdtree_H}(\tilde f)\right]
&= \n\left[
\int_0^\sigma\dd s \int_0^{H(s)}  \expp{-\lambda \sigma^+_{r,s}}
  f(\rdtree_{H^+_{r,s}})\,  \expp{-\lambda 
  \sigma^-_{r,s}}\, \dd r\right]\notag\\
&= \n\left[\sigma \expp{-\lambda \sigma}\right]
 \n\left[\sigma \expp{-\lambda \sigma} f(\rdtree_H)\right]\notag\\
&= \mathfrak{g}(0)^2 \int_0^\infty \expp{-\lambda u} \excm{u}
  \left[f(\rdtree_H)\right]\frac{\dd   u}{u^{1/\gamma}} \,
\int_0^\infty \expp{-\lambda y} \, \frac{\dd
  y}{y^{1/\gamma}} \notag\\ 
&= \mathfrak{g}(0)^2 \int_0^\infty \expp{-\lambda r} \dd r\, 
\int_0^r  \excm{u}
  \left[f(\rdtree_H)\right]\frac{\dd   u}{(u(r-u))^{1/\gamma}} ,\label{laplace1}
\end{align}
where we used  \eqref{eq:decomp-sigma} for the first  equality, Lemma
\ref{lem:H+,H-} for  the second, \eqref{density duration  excursion} for
the third  and the change of variable  $r=u+y$ for the last.   On the other
hand, we  consider the random variable   $H^r=(r^{1-1/\gamma} H(s/r),
s\in [0, r])$ for $r>0$. According to
\eqref{scaling H}, $H^r$ under $\excm{1}$ is distributed as $H$ under $\excm{r}$. Then, we have for $\lambda>0$
\begin{equation}\label{laplace2}
   \n\left[\expp{-\lambda \sigma} \Psi_{\rdtree_H}(\tilde f)\right]
= \mathfrak{g}(0) \int_0^\infty \expp{-\lambda r}
  \E\left[\Psi_{\rdtree_{H^r}}(\tilde f)\right] \, \frac{\dd   r}{r^{1+1/\gamma}} \cdot
\end{equation}
Comparing \eqref{laplace1} and \eqref{laplace2}, we deduce that $\dd r$-a.e., for $r>0$
\begin{equation}
   \label{eq:Er=Nr}
\E\left[\Psi_{\rdtree_{H^r}}(\tilde f)\right] 
= r^{1+ 1/\gamma} \mathfrak{g}(0) \int_0^r  
\frac{\excm{u}
  \left[f(\rdtree_H)\right]}{(r-u)^{1/\gamma}}\, \frac{\dd   u}{u^{1/\gamma}} 
= r^{1+ 1/\gamma} \n \left[ \sigma (r-\sigma)^{-1/\gamma} f
  (\rdtree_H)\ind_{\{\sigma<r\}}\right] .
\end{equation}

From now  on, we  assume that  $f\in \cc_+(\T)$ is bounded and that there exists $\epsilon >0$ such that $f(\tree)=0$ if
$\m(\tree)>1-\varepsilon$.  As
$\m(\rdtree_H)=\sigma$,  the  map
$r    \mapsto     \n    \left[    \sigma     (r-\sigma)^{-1/\gamma}    f
  (\rdtree_H)\ind_{\{\sigma<r\}}\right]$
is continuous at $r=1$ by dominated convergence.  By definition of $H^r$
and  the  continuity  of  the  height function,  we  get  that a.s. $\lim_{r\rightarrow 1} \norm{H^r -H^1}_\infty =0$. 
Following \cite[Proposition 2.10]{abraham2013note}, we get that the
$\T$-valued function 
$r \mapsto \rdtree_{H^r}$ is then a.s. continuous at $r=1$. We deduce from 
Proposition \ref{continuity of the measure} that $r\mapsto
\Psi_{\rdtree_{H^r}}(\tilde f)$ is continuous at $r=1$. We also have
\[
\Psi_{\rdtree_{H^r}}(\tilde f)\leq  \m (\rdtree_{H^r})  \H
(\rdtree_{H^r}) \norm{f}_\infty 
\leq r ^{2- 1/\gamma} \H(H^1) \norm{f}_\infty .
\]
Since $\H(H^1)$ is integrable, we deduce by dominated
convergence that 
the map $r \mapsto \E\left[\Psi_{\rdtree_{H^r}}(\tilde f)\right]  $ is
continuous at $r=1$. We deduce from \eqref{eq:Er=Nr} that for all
$f\in \cc_+(\T)$ 
bounded and such that  there exists $\epsilon >0$ for which $f(\tree)=0$ if
$\m(\tree)>1-\varepsilon$, we have:
\[
\E\left[\Psi_{\rdtree_{H^1}}(\tilde f)\right] =
 \n    \left[    \sigma     (1-\sigma)^{-1/\gamma}    f
  (\rdtree_H)\ind_{\{\sigma<1\}}\right].
\]
By monotone convergence, this equality holds if  $f\in \cc_+(\T)$ is
bounded. Then use that $\rdtree_{H^1}$ is distributed as $\rdtree$ to
get \eqref{eq:EYf}. 
\end{proof}

The next result is a direct consequence of Proposition \ref{prop:EYf},
using that $\pi_*$, defined in \eqref{density duration excursion}, 
is the distribution of $\sigma$ under $\n$. 
Recall the notation $\psij_\rdtree(g(x)h(u))$ which means that $g$ is a function of the mass and $h$ a function of the height.
\begin{corollary}
\label{cor:expectation functional on Levy tree}
	Let $\rdtree$ be the stable Lévy tree with branching mechanism
    $\psi(\lambda) = \kappa \lambda^{\gamma}$ where $\kappa >0$ and
    $\gamma \in (1,2]$. Then we have for every  $f\in \cb_+([0,1]\times\R_+)$ 
\begin{equation}
   \label{expectation functional mass+height}
\ex{\psij_\rdtree(f)} 
= \mathfrak{g}(0) \int_0^1 x^{-1/\gamma}(1-x)^{-1/\gamma}
  \ex{f\left(x,x^{1-1/\gamma}\H(\rdtree)\right)}\, \dd x,
\end{equation}
where $\mathfrak{g}(0)$ is given in \eqref{g(0)}. In particular, we have
for every $g\in \cb_+([0,1])$
\[
\ex{\psij_\rdtree(g(x))} 
=\mathfrak{g}(0)\int_0^1 x^{-1/\gamma}(1-x)^{-1/\gamma}g(x)\, \dd x.
\]
\end{corollary}

\begin{remark}
	An equivalent way to state \eqref{expectation functional mass+height} is the following equality of measures
	\[\ex{\psij_\rdtree(f)} = C(\gamma,\kappa)\ex{f\left(V,V^{1-1/\gamma}\H(\rdtree)\right)} \quad \text{with} \quad C(\gamma,\kappa) = \mathrm{B}(1-1/\gamma,1-1/\gamma)\mathfrak{g}(0), \]
	where $V$ is a random variable with distribution
    Beta$(1-1/\gamma,1-1/\gamma)$, independent of $\H(\rdtree)$ and $\mathrm{B}$ is the beta function. Using
    \eqref{length measure}, this can be interpreted in the following way
    where we recall that $\ell$ denotes the length measure on a real
    tree: taking a stable Lévy tree $\rdtree$ under $\mathbb{P}$ and
    simultaneously choosing a vertex $y \in \rdtree$ uniformly according
    to the measure $C(\gamma,\kappa)^{-1}\mu(\rdtree_y) \ell (\dd y)$,
    then the mass and height of the subtree $\rdtree_y$ are  jointly
    distributed as $V$ and  $V^{1-1/\gamma} \H(\rdtree)$.  
\end{remark}

While the measure $\ex{\psij_\rdtree(\bullet)}$ is not known explicitly, its moments can be expressed in terms of the moments of $\H(\rdtree)$.
\begin{corollary}\label{moments height}
  Let  $\rdtree$  be  the  stable Lévy  tree  with  branching  mechanism
  $\psi(\lambda)     =     \kappa    \lambda^\gamma$.      For     every
  $\alpha,       \beta       \in       \mathbb{C}$       such       that
  $\Re(\gamma\alpha + (\gamma-1)(\beta+1)) >0$, we have
	\begin{equation}\ex{\psij_\rdtree(x^\alpha u^\beta)} = \mathfrak{g}(0) \mathrm{B}\!\left(\alpha + (\beta+1)(1-1/\gamma),1-1/\gamma\right)\ex{\H(\rdtree)^\beta},
	\end{equation}
	where $\mathrm{B}$ is the beta function.
\end{corollary}
Observe that $\H(\rdtree)$ has finite moments of all order. This can be seen as a consequence of the convergence in distribution
$\frac{b_n}{n} \H(\rddtree^n)\cvlaw \H(\rdtree)$
together with the fact that $\left(\frac{b_n}{n}\H(\rddtree^n), \, n \in \mathbb{N}\right)$ is bounded in $L^p$ for every $p \in \real$ by Lemma \ref{height finite moments}. The first moment of $\H(\rdtree)$ is given in \cite[Proposition 3.4]{duquesne2017decomposition}. We shall discuss the other moments in  a future work.

Note that taking $\beta = 0$, we recover \cite[Lemma 4.6]{delmas2018}.
Heuristically, the condition $\Re(\gamma\alpha + (\gamma-1)(\beta+1)
)>0$ is due to the fact that under the excursion measure $\n$, the
height $\Hexc$ scales as $ \sigma^{1-1/\gamma}$ (see also Lemma \ref{lemma scaling}
below), implying that for $\alpha, \beta\in \R$
\[
\ex{\int_\rdtree \mu(\dd x)
    \int_0^{H(x)}\!\! \m(\rdtree_{r\vir x})^\alpha\, \H(\rdtree_{r\vir x})^\beta\,
    \dd r}<\infty  \Longleftrightarrow  \ex{\int_\rdtree
    \mu(\dd x)
    \int_0^{H(x)}\!\!\m(\rdtree_{r\vir x})^{\alpha+\beta(1-1/\gamma)}\, \dd
    r}<\infty. 
\]
Thus, the condition on $\alpha, \beta$ corresponds to the phase
transition observed in \cite[Lemma 4.6 and Remark 4.8]{delmas2018} for 
 functionals depending only on the mass (that is $\beta=0$). 

In the Brownian case, $\H(\rdtree)$ is the maximum of the (scaled) Brownian excursion whose moments are known explicitly. Therefore we get an explicit formula for the moments of the measure $\ex{\psij_\rdtree(\bullet)}$.
\begin{corollary}\label{moments brownian height}
	Let $\rdtree$ be the Brownian tree with branching mechanism $\psi(\lambda) = \kappa \lambda^2$. 
For every $\alpha, \beta \in \mathbb{C}$ such that $\Re(2\alpha + \beta+1) >0$, we have
	\begin{equation}\ex{\psij_\rdtree(x^\alpha u^\beta)} = \frac{1}{\sqrt{\pi\kappa}}\left(\frac{\pi}{\kappa}\right)^{\beta/2}\xi(\beta)\mathrm{B}\!\left(\alpha + \frac{\beta+1}{2},\frac{1}{2}\right),
	\end{equation}
	where $\xi$ is the Riemann xi function defined by $\xi(s) = \frac{1}{2} s(s-1)\pi^{-s/2} \Gammaeuler(s/2)\zeta(s)$ for every $s \in \mathbb{C}$ and $\zeta$ is the Riemann zeta function.
\end{corollary}
\begin{proof}
	The normalized excursion of the height process $H$ is distributed as $\sqrt{2/\kappa} \,\Br$ where $\Br$ is the normalized Brownian excursion, see \emph{e.g.} \cite{duquesne2002random}. Therefore we get the identity $\H(\rdtree)\law\sqrt{2/\kappa}\, \max \Br$. By \cite[Proposition 2.1 and Eq. (4.10)]{biane2001probability}, we have
	\[\ex{\left(\max {\Br}\right) ^\beta}= 2 \left(\frac{\pi}{2}\right)^{\beta/2} \xi(\beta), \quad\forall \beta \in \mathbb{C}.\]
	The result follows then from Corollary \ref{moments height} and the
    value of $\mathfrak{g}(0)$ 
given in \eqref{g(0),2} .
\end{proof}

\subsection{Finiteness of $\psij_{\rdtree}(f)$}
This section is devoted to the study of the finiteness of functionals of
the mass and height on the stable Lévy tree. Arguing as in the proof of Lemma \ref{lem:Amh^p} and using Corollary \ref{moments height} and the fact that $\H(\rdtree)$ has finite moments of all orders, we get the following result.
\begin{lemma}\label{lem:finite moment delta}
	Let $\rdtree$ be the stable Lévy tree with branching mechanism
    $\psi(\lambda) = \kappa \lambda^\gamma$ where $\kappa>0$ and $\gamma \in (1,2]$.
    Let  $\alpha, \beta \in \real$ such that $\gamma \alpha + (\gamma - 
    1)(\beta+1) >0$. For  and every  $p \geq 1$ such that $p(\gamma
    \alpha + (\gamma-1)\beta)>1-\gamma$ and $\delta \in \real$, we have:
\begin{equation}
\label{finite moment delta}
	\ex{\H(\rdtree)^\delta \psij_{\rdtree}(x^\alpha u^\beta)^p}< \infty.
\end{equation}
\end{lemma}

We now state the main result of this section which gives an integral
test for the finiteness of functionals of the mass and height on the
stable Lévy tree. 

\begin{proposition}
\label{joint phase transition}
 Let $\rdtree$ be the stable Lévy tree with branching mechanism
 $\psi(\lambda) = \kappa \lambda^\gamma$  where $\kappa>0$ and $\gamma \in (1,2]$. Let $f
 \in \cb_+( [0,1]\times \R_+)$ be of the form $f(x,u) = g(x)u^\beta$ or
 $f(x,u) = x^\alpha h(u)$ where $\alpha, \beta \in \real$, and $g,h$
 nonincreasing. Then we have 
\begin{equation}
\psij_{\rdtree}(f) \begin{cases}
< \infty & \text{a.s.,} \\
= \infty & \text{a.s.,}
\end{cases}
\end{equation}
according as
\begin{equation}\int_0 f(x^{\gamma/(\gamma-1)},x)\, \dd x \begin{cases}<\infty, \\
= \infty. \end{cases} \label{integral test}
\end{equation} 
Furthermore, if $\psij_\rdtree(f)$ is a.s. finite then we have
$\ex{\psij_\rdtree(f)} < \infty$. 
\end{proposition}
\begin{proof}
	We first prove that if $\int_0 f(x^{\gamma/(\gamma-1)},x)\, \dd x$
    is finite then $\ex{\psij_\rdtree(f)}$ is finite and thus
    $\psij_{\rdtree}(f)$ is a.s. finite. 

    Let $\beta \in \real$ and  $g\in \cb_+([0,1])$ be  such
    that   $\int_0  g(x^{\gamma/(\gamma-1)})   x^\beta   \,   \dd  x   <
    \infty$.
    Recall that $\H(\rdtree)$  has finite moments of  all orders.  Thus,
    by \eqref{expectation functional mass+height}, we have
	\begin{equation*}
	\ex{\psij_{\rdtree}(g(x)u^\beta)} =  \mathfrak{g}(0) \ex{\H(\rdtree)^\beta}\int_0^1 g(x) x^{(\beta+1)(1-1/\gamma)-1}(1-x)^{-1/\gamma} \, \dd x < \infty.
	\end{equation*}
	
	Next, let $\alpha \in \real$ and $h \in \cb_+(\R_+)$ be nonincreasing such that $\int_0 h(x)x^{\alpha \gamma/(\gamma - 1)}\, \dd x < \infty$. Again by \eqref{expectation functional mass+height}, we have
	\begin{equation*}
	\ex{\psij_\rdtree(x^\alpha h(u))} = \mathfrak{g}(0)  \int_0^1 x^{\alpha - 1/\gamma}(1-x)^{-1/\gamma} \ex{h\left(x^{1-1/\gamma}\H(\rdtree)\right)}\, \dd x.
	\end{equation*}
	Now, letting $k$ goes to infinity  in \eqref{stochastic order} and using the continuity of the cdf of $\H(\rdtree)$ (see \cite{duquesne2017decomposition}), we get that
	\begin{equation*}
	\pr{\H(\rdtree)\leqslant y } \leqslant 1 \wedge \left( C_0
      \exp\left(-c_0 y^{-\gamma/(\gamma-1)}\right)\right) \quad\text{for
      all}\quad  y \geqslant 0. 
	\end{equation*}
    We deduce that $\H(\rdtree)  \geqslant_{\mathrm{st}} Y$ where the
    cdf of the random variable $Y$ is
     given by the right-hand side of the inequality above.  Using
    that $h$ is nonincreasing and  repeating the same computations as in
    the    proof   of    Lemma   \ref{mass+height    bounded   in    L1}
    (cf. \eqref{integral test, part 2}), we deduce that
	\begin{align*}
	\ex{\psij_\rdtree(x^\alpha h(u))} &\leqslant \mathfrak{g}(0)\int_0^1 x^{\alpha - 1/\gamma}(1-x)^{-1/\gamma} \ex{h\left(x^{1-1/\gamma}Y\right)}\, \dd x < \infty.
	\end{align*}
	This finishes the proof of the finite case. The infinite case is more delicate and its proof is postponed to Section \ref{section proof phase transition}. 
\end{proof}

We end this section with a complete description of the behavior of
polynomial functionals of the mass and height on the stable Lévy tree,
which is a particular case of Proposition \ref{joint phase transition} (and Lemma \ref{lem:finite moment delta} for $\alpha> 0$ and $\beta> 0$). 
\begin{corollary}\label{joint phase transition poly}
	Let $\rdtree$ be the stable Lévy tree with branching mechanism
    $\psi(\lambda) = \kappa \lambda^\gamma$ with $\kappa>0$ and $\gamma \in (1,2]$, and let $\alpha, \beta \in \real$. Then we have
	\begin{align}\gamma \alpha + (\gamma-1)(\beta+1) > 0 \quad \Longleftrightarrow \quad \psij_\rdtree(x^\alpha u^\beta) < \infty \ \text{a.s.} \quad \Longleftrightarrow \quad\ex{\psij_\rdtree(x^\alpha u^\beta)} <\infty, \label{joint finite poly}\\
	\gamma \alpha + (\gamma-1)(\beta+1) \leqslant 0 \quad \Longleftrightarrow \quad \psij_\rdtree(x^\alpha u^\beta) = \infty \ \text{a.s.} \quad \Longleftrightarrow \quad\ex{\psij_\rdtree(x^\alpha u^\beta)} =\infty.\label{joint infinite poly}
	\end{align}
\end{corollary}

\subsection{Proof of the infinite case in Proposition \ref{joint phase
    transition}}\label{section proof phase transition} 
	Recall that $H$ denotes the height process under the excursion
    measure $\n$.  Recall that $\sigma_{r,s} $ and $\H_{r,s}$ are  the length and height of the excursion of $H$
    above level $r$ that straddles $s$, see Section \ref{sec:frag}.  Let $f\in \cb_+( [0,1]\times\R_+ )$. Set
	\begin{equation}
	Z_f = \int_0^\sigma\dd s \int_0^{H(s)} f(\sigma_{r,s},\H_{r,s})\, \dd r.
	\end{equation}
	Notice that under $\excm{1}$, the random variable $Z_f$ is
    distributed as $\psij_\rdtree (f)$ under $\P$. Using the scaling property
    \eqref{scaling H} of the height process, we have the following more
    general result which is partially given in \cite{delmas2018} (notice
    that there is a misprint in the first line of p.34 therein). 
\begin{lemma}
\label{lemma scaling}
      Let  $\psi(\lambda)=\kappa  \lambda  ^\gamma$ with $\kappa>0$ and $\gamma \in (1,2]$ and let $H$ be the $\psi$-height process. For every $x>0$, the random variable
\[
\big( (\H(s), \, s \in [0,x]), (\sigma_{r,s},H_{r,s}; \, r \in [0,H(s)],
  s \in [0,x])\big)
\]
under
      $\n^{(x)}$ 	is distributed as the following random variable 	under $\n^{(1)}$
\[
\Big( \!\left(x^{1-1/\gamma}H(s/x), \, s \in [0,x]\right),
\left(x\sigma_{x^{-1+1/\gamma}r,s/x},x^{1-1/\gamma}\H_{x^{-1+1/\gamma}r,s/x};
\,  r \in [0,x^{1-1/\gamma}H(s/x)], s \in [0,x]\right)\!\Big).
\]
In particular, the random variable
$\big(  \left(H(s), \, s\in [0,x]\right),Z_f \big)$
under  $\n^{(x)}$ is distributed as the random variable $
\Big( \left(x^{1-1/\gamma} H(s/x), \allowbreak\, s \in
    [0,x]\right),x^{2-1/\gamma} Z_{f_x}\Big)$ 
 under $\n^{(1)}$, where $f_x$ is defined by $f_x(y,u) = f(xy,x^{1-1/\gamma}u)$ for  $x >0$.
	\end{lemma}

	Conditionally on $H$, let $U$ be uniformly distributed on
    $[0,\sigma]$ under $\n[\sigma \bullet]$. Using Bismut's
    decomposition, see \emph{e.g.} \cite[Theorem
    4.5]{duquesne2005probabilistic} or \cite[Theorem
    2.1]{abraham2013forest}, we get that under $\n[\sigma \bullet]$, the
    random variable $H(U)$ has Lebesgue distribution on $(0,\infty)$
    and, conditionally on $\{H(U) =t\}$, the process
    $(\left(\sigma_{t-r,U},\H_{t-r,U}\right),\,\allowbreak 0 \leqslant r
    \leqslant t)$ is distributed as $(\left(\mass_r,\height_r\right),\,
    0\leqslant r \leqslant t)$ where 
	\begin{equation}\label{def subordinator+record}
	\mass_r = \sum_{s\leqslant r} \m(\ppp_s) 
\quad \text{and}\quad 
\height_r = \max_{s\leqslant r}\left(\H(\ppp_s) +r-s\right), \quad \forall
0\leqslant r\leqslant t, 
	\end{equation}
	where $\m(\ppp_s)$  (resp. $\H(\ppp_s)$)  stands for  the mass  (resp. the
    height) of  the real tree  $\ppp_s$, and  $\ppp=(\ppp_s, \, s\geq
    0)$ is a  $\T$-valued Poisson 
    point  process  on  $[0,t]$  whose intensity  is  given  below.   If
    $\gamma  =  2$,   the  Poisson  point  process   $\ppp$  has  intensity
    $2\kappa   \n$.    To   describe   the   intensity    of   $\ppp$   for
    $\gamma  \in  (1,2)$,  we  introduce  the  probability  distribution
    $\mathbf{P}_{\!a}$  on  $\T$ which  is  the  law  of a  random  tree
    obtained by  gluing a family of  trees  $(T_i,  i\in I)$ at
    their root, with $\sum_{i\in I}  \delta_{T_i}( \dd T)$ a $\T$-valued
    Poisson  point  measure  with  intensity $a  \n[\dd  T]$,  see  also
    \cite[Section     2.6]{abraham2013forest}     for  more   details     on
    $\mathbf{P}_{\!a}$. If $\gamma \in (1,2)$, the Poisson point process
    $\ppp$                           has                          intensity
    $\int_0^\infty  a \pi(\dd  a )\mathbf{P}_{\!a}(\dd  T)$ where  $\pi$
    is  the  Lévy  measure associated  with
    $\psi$ given  by  \eqref{eq:def-pi}.  In particular, we get the equality in law
\begin{equation}\label{bismut give H(U)=t}
	\int_0^{H(U)} f(\sigma_{r,U},\H_{r,U})\, \dd r \ \text{under} \
    \n\left[\sigma\bullet\middle| H(U)=t\right] \quad \lawd \quad
    \int_0^t f(\mass_r,\height_r)\, \dd r. 
\end{equation}
	In the  proof of  \cite[Lemma 4.6]{delmas2018}, see Section 8.6 and
    more precisely (8.20) therein, it is proven that 
 $\mass$  is a  stable subordinator  with Laplace  transform
    $\ex{\exp(-\lambda         \mass_1)}          =         \exp(-\gamma
    \kappa^{1/\gamma}\lambda^{1-1/\gamma})$.
    We  shall  determine the  intensity  of  the Poisson  point  process
    $\H(\ppp)=(\H(\ppp_s), \,  0\leqslant s \leqslant  t)$.  If $\gamma  = 2$,
    $\H(\ppp)$  has  intensity   $2\kappa  \n[\H  \in  \dd   x]$.  But,  by
    \cite[Eq.       (14)]{duquesne2005probabilistic},      we       have
    $\n[\H > x] = 1/(\kappa x)$. Differentiating with respect to $x$, we
    get $\n[\Hexc \in \dd  x] = \kappa^{-1}x^{-2}\ind_{\{x>0\}}\,\dd x$,
    so that $\H(\ppp)$ has intensity  $2 x^{-2}\ind_{\{x>0\}}\, \dd x$.  If
    $1< \gamma < 2$, $\H(\ppp)$ has intensity
\[
\int_0^\infty a \pi(\dd a) \, \mathbf{P}_{\! a} (\Hexc \in \dd x).
\]
Using \eqref{density height excursion} and the definition of $\mathbf{P}_{\! a}$, we have
		$\mathbf{P}_{\! a}\left(\H \leqslant x\right)
		= \e^{-a \n\left[\H > x\right]} = \e^{-C a x^{-1/(\gamma-1)}}$
		where $C = (\kappa (\gamma - 1))^{-1/(\gamma-1)}$. Differentiating with respect to $x$, we obtain
\[   
\mathbf{P}_{\!   a}   (\H   \in  \dd   x)   =   \frac{C   a
          x^{-\gamma/(\gamma-1)}}{\gamma-1}\e^{-Ca    x^{-1/(\gamma-1)}}
        \ind_{\{x>0\}}\,              \dd             x.              
\]
        Since $\pi(\dd a) =  C' a^{-1-\gamma} \, \dd a$ where
        $C'  =  \kappa  \gamma  (\gamma  -1)/\Gammaeuler(2-\gamma)$,  
        see \eqref{eq:def-pi}, we
        deduce  that for $x>0$
\begin{align*}
		\int_0^\infty a \pi(\dd a) \, \mathbf{P}_{\! a}(\H \in \dd x) 
&= \frac{C C'}{\gamma - 1}\left(\int_0^\infty
  a^{1-\gamma}x^{-\gamma/(\gamma-1)}\e^{-Cax^{-1/(\gamma-1)}} \, \dd
  a\right) \ind_{\{x>0\}}\,\dd x  \notag\\ 
&= \frac{C^{\gamma - 1} C' \Gammaeuler (2-\gamma)}{\gamma -
          1}\ind_{\{x>0\}}\, \frac{\dd x}{x^2} \\
&= \frac{\gamma}{\gamma - 1}\ind_{\{x>0\}}\,\frac{\dd x}{x^2}\cdot
\end{align*}
		In all cases, for $\gamma \in (1,2]$, we get that $\H(\ppp)$ is a
        Poisson point process with intensity $(\gamma/(\gamma-1))x^{-2}
        \ind_{\{x>0\}}\, \dd x$. Intuitively, this implies that
        $\mass_r$ is of order $r^{\gamma/(\gamma-1)}$ while $\height_r$
        is of order $r$ as $r \to 0$ which, together with \eqref{bismut
          give H(U)=t}, explains the form of the integral test
        \eqref{integral test}.

Our goal now is to show that
	\begin{equation*}
	\int_0 f(x^{\gamma/(\gamma-1)},x)\, \dd x = \infty \quad \implies \quad \int_0 f(\mass_t,\height_t)\, \dd t = \infty \quad \text{a.s.}
	\end{equation*}
	under the assumptions of Proposition \ref{joint phase transition}. To do this, we adapt the proof of Theorem 1 in \cite{erickson2007finiteness}  which gives a necessary and sufficient condition for the divergence of integrals of Lévy processes. We first consider the case $f(x,u) = x^\alpha h(u)$.
	\begin{lemma}\label{mass polynomial}
      Let $\alpha >-1+1/\gamma$ and  $h\in \cb_+(\R_+)$ be nonincreasing
      such                                                          that
      $\int_0  h(x)\allowbreak x^{\alpha  \gamma/(\gamma -  1)}\, \dd  x
      \allowbreak = \infty$. We have that a.s.
		\begin{equation*}
		\int_0 \mass_t^\alpha h(\height_t) \, \dd t = \infty.
		\end{equation*}
	\end{lemma}
	\begin{proof}
		Define the first passage time for $a>0$
		\begin{equation}
		\theight(a) \coloneqq \inf \left\lbrace t >0 \colon \, \height_t \geqslant a \right\rbrace.
		\end{equation}
		Since $t \mapsto \height_t$ is right-continuous, we have
\begin{equation}
\label{eq:T_H-at}
		\left \lbrace \theight(a) > t \right \rbrace = \left \lbrace
          \height_t< a \right \rbrace. 
\end{equation}
		Furthermore, since $\height_0 = 0$, it holds that a.s. $\theight(a)>0$ for every $a>0$.
		
		Set $F(t) = \int_0^t \mass_s^\alpha\, \dd s$. Clearly $F(t) <\infty$ a.s. if $\alpha \geqslant 0$. If $-1+1/\gamma<\alpha<0$, we have
		\begin{equation*}
		\ex{F(t)} = \int_0^t \ex{\mass_s^\alpha}\, \dd s = \ex{\mass_1^\alpha}\int_0^t s^{\alpha\gamma/(\gamma-1)}\, \dd s,
		\end{equation*}
		where we used that $\mass$ is stable with index $1-1/\gamma$. Now the last integral is finite because of the condition on $\alpha$, and
		\begin{equation*}
		\ex{\mass_1^\alpha} = \frac{1}{\Gammaeuler(|\alpha|)}\int_0^\infty \ex{\e^{-\lambda \mass_1}} \lambda^{-1-\alpha}\, \dd \lambda = \frac{1}{\Gammaeuler(|\alpha|)}\int_0^\infty\e^{-\gamma \kappa^{1/\gamma}\lambda^{1-1/\gamma}} \lambda^{-1-\alpha}\, \dd \lambda <\infty.
		\end{equation*}
		Thus, we get  $F(t) <\infty$ a.s. for $\alpha >-1+1/\gamma$. Furthermore, $F$ is nondecreasing and we have
		\begin{equation}
		\int_0^1 \mass_t^\alpha h(\height_t)\, \dd t = \int_0^1 h(\height_t) \, \dd F(t).
		\end{equation}
		
		We shall need the first and second moment of $F(\theight(a))$ for $a>0$.
		Using \eqref{eq:T_H-at}, we have that 
		\begin{equation*}
		\ex{F(\theight(a))} = \int_0^\infty \ex{\mass_t^\alpha \ind_{\lbrace \theight(a) > t \rbrace}} \, \dd t = \int_0^\infty \ex{\mass_t^\alpha \ind_{\lbrace \height_t < a\rbrace}}\, \dd t.
		\end{equation*}
		On the other hand, notice that for every $s \in [0,\sigma]$, it
        holds that $\sigma_{0,s} = \sigma$ is the total mass and
        $H_{0,s} = \H$ is the total height. Thus, using Bismut's
        decomposition, we have  
\begin{equation}
\label{Bismut justif} 
		\n\left[ \sigma^{\alpha+1}\ind_{\lbrace \H< a \rbrace}\right] =
        \int_0^\infty \n\left[\sigma \sigma_{0,U}^\alpha \ind_{\{H_{0,U}
            < a\}}\middle| H(U) = t\right]\,\dd t = \int_0^\infty
        \ex{\mass_t^\alpha \ind_{\{\height_t < a\}}} \, \dd t, 
\end{equation}
where we recall  that conditionally on $H$,  under $\n[\sigma \bullet]$,
$U$     is     uniformly     distributed     on     $[0,\sigma]$     and
$(\sigma_{0, U}, H_{0, U})$ conditionally on $\{H(U)=t\}$ is then distributed
as $(\mass_t, \height_t)$.  We deduce that
\begin{align}
		\ex{F(\theight(a))} 
&= \n\left[ \sigma^{\alpha+1}\ind_{\lbrace \H< a \rbrace}\right]\notag\\ 
&= \mathfrak{g}(0) \int_0^\infty x^{-1-1/\gamma}
         \excm{x}\left[\sigma^{\alpha+1} \ind_{\lbrace \H < a
          \rbrace}\right]\, \dd x \notag\\ 
&= \mathfrak{g}(0)\int_0^\infty x^{\alpha-1/\gamma} \excm{1}\left[
  x^{1-1/\gamma}\H < a \right]\, \dd x \notag\\ 
&= \frac{\gamma \mathfrak{g}(0)}{(\alpha+1) \gamma  - 1}
          \excm{1}\left[\H^{-1-\alpha\gamma/(\gamma-1)}\right]
          a^{1+\alpha\gamma/(\gamma-1) }, 
\label{claim 1}
\end{align}
		where we disintegrated  with respect to $\sigma$  for the second
        equality and used the scaling  property \eqref{scaling H} of the
        height  process  for the  third.  Recall  that $\H$  has  finite
        moments   of    all   orders    under   $\excm{1}$,    so   that
        $\E[F(\theight(a))]$   is   finite    for   all   $a>0$.   Next,
        set 
\begin{equation*}
Z_\alpha^{\mathfrak{m}} = \int_0^\sigma \dd
          s\int_0^{H(s)} \sigma_{r,s}^\alpha \, \dd r.
\end{equation*} 
		It   follows  from   Lemma \ref{lemma scaling}   that  under   $\excm{x}$,
        $(\H,Z_\alpha^\mathfrak{m})$       is       distributed       as
        $(x^{1-1/\gamma}\H,\allowbreak
        x^{\alpha+2-1/\gamma}Z_\alpha^\mathfrak{m})$
        under $\excm{1}$. Recall  that $\alpha>-1+1/\gamma$. Thus, using
        Bismut's decomposition as in \eqref{Bismut justif}, we have
		\begin{align}
		\ex{F(\theight(a))^2} &= 2\ex{ \int_0^\infty \mass_t^\alpha \ind_{\lbrace \height_t < a\rbrace}\, \dd t \int_0^t \mass_s^\alpha \, \dd s} \notag\\
		&=2\n\left[\sigma^{\alpha+1}\ind_{\lbrace\H< a\rbrace}\int_0^{H(U)}\sigma_{r,U}^\alpha \, \dd r\right] \notag\\
		&=2\n\left[\sigma^\alpha \ind_{\lbrace \H < a \rbrace} Z_\alpha^{\mathfrak{m}}\right]\notag\\
		&= 2 \mathfrak{g}(0) \int_0^\infty x^{-1-1/\gamma} \excm{x}\left[\sigma^\alpha \ind_{\lbrace \H < a \rbrace} Z_\alpha^{\mathfrak{m}}\right]\, \dd x\notag\\
		&= 2 \mathfrak{g}(0) \int_0^\infty x^{-1-1/\gamma} \excm{1}\left[x^\alpha \ind_{\lbrace x^{1-1/\gamma}\H < a \rbrace} x^{\alpha +2-1/\gamma}Z_\alpha^{\mathfrak{m}}\right]\, \dd x\notag\\
		&= \frac{\mathfrak{g}(0)}{\alpha+1-1/\gamma}\excm{1}\left[\H^{-2(1+\alpha\gamma/(\gamma-1))}Z_\alpha^{\mathfrak{m}}\right] a^{2(1+\alpha\gamma/(\gamma-1))}, \label{claim 2}
		\end{align}
		where the last term is finite by \eqref{finite moment delta}.
		Combining \eqref{claim 1} and \eqref{claim 2} and using
        Cauchy-Schwartz inequality, we deduce that there exists some
        finite constant $C>0$ such that  for all $a,b>0$
\begin{equation}
\label{claim 3}
		\ex{F(\theight(a))F(\theight(b))} \leqslant
        \ex{F(\theight(a))^2}^{1/2} \ex{F(\theight(b))^2}^{1/2}
        \leqslant C \ex{F(\theight(a))}\ex{F(\theight(b))}.  
\end{equation}

		For $i \in \mathbb{N}$, put $\theight_i = \theight(2^{-i})$,
        $h_i = h(2^{-i})$ and $\Delta h_i = h_{i+1} - h_i$. Notice that
        the sequence $(\theight_i, \, i \in \mathbb{N})$ is
        nonincreasing and $\Delta h_i \geqslant 0$. Set $V_n =
        \sum_{i=1}^n F(\theight_i) \Delta h_{i-1}$. Notice that
        $\E[V_n]$ is finite as $\E[F(\theight(a))]$ is finite for all
        $a>0$. By \eqref{claim 3}, we have
\begin{align*}
\ex{V_n^2} 
&= \sum_{i=1}^n \ex{F(\theight_i)^2} \left(\Delta h_{i-1}\right)^2 + 2
  \sum_{1\leqslant i < j \leqslant n} \ex{F(\theight_i)F(\theight_j)}
  \Delta h_{i-1} \Delta h_{j-1} \\ 
&\leqslant C\sum_{i=1}^n \ex{F(\theight_i)}^2 \left(\Delta
  h_{i-1}\right)^2 + 2C \sum_{1\leqslant i < j \leqslant n}
  \ex{F(\theight_i)}\ex{F(\theight_j)} \Delta h_{i-1}\Delta h_{j-1} \\ 
&=C \left(\sum_{i=1}^n \ex{F(\theight_i)}\Delta h_{i-1}\right)^2 =
  C\ex{V_n}^2. 
\end{align*}
		Therefore, we get that $\limsup_n \ex{V_n}^2/\ex{V_n^2} >0$. By \cite{kochen1964note}, it follows that 
		\begin{equation}\label{borel cantelli}
		\pr{\limsup_n \frac{V_n}{\ex{V_n}} \geqslant 1} >0.
		\end{equation}
Using \eqref{claim 1}, notice that for some finite constant $C>0$, we have
		\begin{align}
		\int_0^1 x^{1+\alpha \gamma/(\gamma-1)} \, |\dd h(x)| 
&\leqslant \sum_{i=1}^\infty (2^{-i+1})^{1+\alpha \gamma/(\gamma-1)}
  \int_{2^{-i}}^{2^{-i+1}} |\dd h(x)| \notag\\ 
&= C\sum_{i=1}^\infty \ex{F(\theight_i)}\Delta h_{i-1} =
          C\lim_{n\to \infty}\ex{V_n}. \label{X_n} 
		\end{align}
		Since $\int_0^1 x^{1+\alpha\gamma/(\gamma-1)} \, |\dd h(x)| \geqslant -h(1) +\left(1+\alpha \gamma/(\gamma-1)\right) \int_0^1 h(x)x^{\alpha \gamma/(\gamma-1)}\, \dd x= \infty$ by assumption, it follows from \eqref{X_n} that $\lim_{n\to \infty}\allowbreak\ex{V_n}\allowbreak = \infty$. Thus, using \eqref{borel cantelli} and the fact that $V_n$ is nondecreasing, we deduce that $\lim_{n\to \infty} V_n = \infty$ with positive probability, that is
		\begin{equation}\label{positive probability}
		\pr{\sum_{i=1}^\infty F(\theight_i)\Delta h_{i-1} = \infty} >0.
		\end{equation}
		
		Since $h$ is nonincreasing, we have
		\begin{equation}\label{int sum}
		\int_0^{\theight_0} h(\height_t)\, \dd F(t) \geqslant \sum_{i=0}^\infty h_{i-1}\left(F(\theight_{i-1}) - F(\theight_i)\right).
		\end{equation}
		A summation by parts gives
		\begin{equation}\label{spp}
		\sum_{i=1}^n h_{i-1}\left(F(\theight_{i-1}) - F(\theight_i)\right) = F(\theight_0)h_0 - F(\theight_n) h_n + \sum_{i=1}^n F(\theight_i)\Delta h_{i-1}.
		\end{equation}
		But, notice that
		\begin{equation*}F(\theight_n) h_n = F(\theight_n) h(2^{-n}) \leqslant \int_0^{\theight_n} h(\height_t) \, \dd F(t) \leqslant \int_0^{\theight_0} h(\height_t) \, \dd F(t).
		\end{equation*}
		Together with \eqref{int sum} and \eqref{spp}, this yields
		\begin{equation*}
		F(\theight_0)h_0+ \sum_{i=1}^\infty F(\theight_i) \Delta h_{i-1}  \leqslant 2\int_0^{\theight_0} h(\height_t)\, \dd F(t).
		\end{equation*}
		It follows from \eqref{positive probability} that $\int_0^{\theight_0}\mass_t^\alpha h(\height_t)\, \dd t = \int_0^{\theight_0} h(\height_t) \, \dd F(t)$ diverges with positive probability.
		
		Finally, since the event $\left\lbrace \int_0 \mass_t^\alpha h(\height_t) \, \dd t = \infty \right \rbrace$ is $\mathcal{F}_{0+}$-measurable where $(\mathcal{F}_t)_{t\geqslant 0}$ is the filtration generated by the Poisson point process $\ppp$, Blumenthal's zero-one law entails that $\int_0^1 \mass_t^\alpha h(\height_t)\, \dd t$ diverges with probability $1$.
	\end{proof}
	
	\begin{lemma}\label{height polynomial}
		Let $\beta > -1$ and $g\in \cb_+( [0,1])$ be nonincreasing such
        that $\int_0 g(x^{\gamma/(\gamma-1)}) x^{\beta}\, \dd x
        =\infty$. We have that a.s. 
		\begin{equation*}
		\int_0 g(\mass_t) \height_t^\beta \, \dd t =\infty.
		\end{equation*}
	\end{lemma}
	\begin{proof}
		The proof is similar to that of Lemma \ref{mass polynomial} and we only highlight the major differences. Define the first passage time $\tmass(a) = \inf \left\lbrace t >0 \colon \, \mass_t > a\right \rbrace$ for every $a >0$. Since $\mass$ is a stable subordinator, we have a.s. $\tmass(a) >0$ for every $a >0$. Set $F(t) = \int_0^t \height_s^\beta\, \dd s$. Notice that $F(t)<\infty$ a.s. if $\beta \geqslant 0$. If $-1<\beta <0$, then using that $\height_s \geqslant s$, we have a.s. $F(t)\leqslant \int_0^t s^\beta \, \dd s < \infty$. To compute the first moment of $F(\tmass(a))$, use Bismut's decomposition as in \eqref{Bismut justif} to get
		\begin{align}
		\ex{F(\tmass(a))} &= \ex{\int_0^\infty \height_t ^\beta \ind_{\lbrace \mass_t < a \rbrace}\, \dd t} \notag\\
		&= \n\left[\sigma \ind_{\lbrace \sigma < a \rbrace} \H^\beta\right] \notag\\
		&= \mathfrak{g}(0) \int_0^a x^{(\beta+1)(1-1/\gamma)-1} \excm{1}\left[\H^\beta\right]\,\dd x\notag \\
		&= \frac{\mathfrak{g}(0)}{(\beta+1)(1-1/\gamma)} \excm{1}\left[\H^\beta\right]a^{(\beta+1)(1-1/\gamma)}. \label{claim 4}
		\end{align}
		Setting
		\begin{equation*}
		Z_\beta^{\mathfrak{h}} = \int_0^\sigma \dd s \int_0^{H(s)} H_{r,s}^\beta\, \dd r
		\end{equation*}
		and using Bismut's decomposition as in \eqref{Bismut justif} and
        the fact that under $\excm{x}$, $(\H,Z_\beta^\mathfrak{h})$ is
        distributed as $(x^{1-1/\gamma}\H,\allowbreak
        x^{(\beta+1)(1-1/\gamma)+1}Z_\beta^\mathfrak{h})$ under
        $\excm{1}$ by Lemma \ref{lemma scaling}, we have
		\begin{align}\label{claim 5}
		\ex{F(\tmass(a))^2} &= 2 \ex{\int_0^\infty \height_t^\beta \ind_{\lbrace \mass_t < a\rbrace}\, \dd t \int_0^t \height_s^\beta \, \dd s}\notag\\
		&= 2 \n\left[\sigma \ind_{\lbrace \sigma < a\rbrace} \H^\beta \int_0^{H(U)} H_{r,U}^\beta\, \dd r \right]\notag\\
		&= 2\n\left[\ind_{\lbrace \sigma < a \rbrace} \H^\beta Z_\beta^\mathfrak{h}\right]\notag\\
		&= 2 \mathfrak{g}(0) \int_0^a x^{-1-1/\gamma}  \excm{x}\left[\H^\beta Z_\beta^\mathfrak{h}\right] \, \dd x\notag\\
		&= \frac{\mathfrak{g}(0)}{(\beta+1)(1-1/\gamma)} \excm{1}\left[\H^\beta Z_\beta^\mathfrak{h}\right]a^{2(\beta+1)(1-1/\gamma)},
		\end{align}
		where $\excm{1}\left[\H^\beta Z_\beta^\mathfrak{h}\right]
        <\infty$ by \eqref{finite moment delta}. Combining \eqref{claim
          4} and \eqref{claim 5}, we see that the estimate \eqref{claim
          3} holds. The rest of the proof is similar to that of Lemma
        \ref{mass polynomial} (with $h_i$ replaced by $g_i=g(2^{-i})$). 
	\end{proof}

	We can now finish the proof of Proposition \ref{joint phase
      transition}. Let $f \in \cb_+( [0,1]\times\R_+)$ be of
    the form $f(x,u) = g(x)u^\beta$ or $f(x,u) = x^\alpha h(u)$ with
    $g,h$ nonincreasing and such that $\int_0 f(x^{\gamma/(\gamma-1)},x)\,
    \dd x = \infty$. By Lemmas \ref{mass polynomial} and \ref{height
      polynomial}, we have that, in the cases $\alpha >-1+1/\gamma$ and $\beta
    >-1$,  a.s. 
	\begin{equation}\label{integral PPP}
	\int_0 f(\mass_t,\height_t)\, \dd t = \infty. 
	\end{equation}

	Now suppose that $\alpha \leqslant -1+1/\gamma$. Since $h$ is nonincreasing and satisfies $\int_0 h(x)x^{\alpha\gamma/(\gamma-1)}\, \dd x = \infty$, there exists a constant $C>0$ such that $h\geqslant C$ on some interval $(0,\epsilon)$. Thus, we have
		\begin{equation*}
		\int_0 \mass_t^\alpha h(\height_t)\, \dd t \geqslant C \int_0 \mass_t^\alpha \, \dd t,
		\end{equation*}
		where  the  last integral  diverges  a.s.  by Lemma  \ref{height
          polynomial}                                                 as
        $\int_0   x^{\alpha   \gamma/(\gamma-1)}\,   \dd   x   =\infty$.
        Similarly,  if $\beta  \leqslant  -1$, there  exists a  constant
        $C'>0$  such that  $g\geqslant C'$  on $(0,\epsilon)$.  Thus, we
        have
\begin{equation*}
		\int_0 g(\mass_t)\height_t^\beta\, \dd t \geqslant C' \int_0 \height_t^\beta \, \dd t,
\end{equation*}
		and the last integral diverges by Lemma \ref{mass polynomial} since $\int_0 x^{\beta} \, \dd x = \infty$. This proves that \eqref{integral PPP} holds for all $\alpha,\beta \in \real$.

		Combining \eqref{bismut give H(U)=t} and \eqref{integral PPP}, we deduce that
		\begin{align*}
		\n\left[\sigma;\, Z_f < \infty\right] &= 	\n\left[\sigma; \, \sigma \int_0^{H(U)} f(\sigma_{r,U},H_{r,U})\, \dd r< \infty\right] \\
		&= \int_0^\infty \n\left[\sigma; \, \sigma \int_0^{H(U)} f(\sigma_{r,U},H_{r,U})\, \dd r< \infty\middle| H(U) = t\right]\, \dd t \\
		&=\int_0^\infty\pr{\mass_t \int_0^{t} f(\mass_r,\height_r)\, \dd r< \infty}\, \dd t = 0.
		\end{align*}
		It follows that $\n$-a.e. $Z_f = \infty$. Disintegrating with
        respect to $\sigma$ and using the scaling property from Lemma
        \ref{lemma scaling}, we get
		\begin{equation*}
		0 = \n\left[Z_f < \infty\right] = \int_0^\infty \excm{x}\left[Z_f < \infty\right]\, \pi_*(\dd x) = \int_0^\infty \excm{1}\left[x^{2-1/\gamma} Z_{f_x} < \infty\right]\, \pi_*(\dd x).
		\end{equation*}
		Consequently, $\dd x$-a.e. on $(0,\infty)$, we have
        $\excm{1}\left[Z_{f_x} < \infty\right] = 0$. Suppose that
        $f(y,u) = g(y)u^\beta$ with $g$ nonincreasing. Then, under $\excm{1}$, $Z_{f_x}$ is
        equal to $x^{\beta(1-1/\gamma)} \int_0^1 \dd s \int_0^{H(s)}
        g(x\sigma_{r,s})H_{r,s}^\beta\, \dd r$ and we get that 
\begin{equation*}
		x\mapsto\excm{1}\left[ \int_0^1 \dd s \int_0^{H(s)}
          g(x\sigma_{r,s})H_{r,s}^\beta\, \dd r< \infty\right] 
\end{equation*}
vanishes  $\dd  x$-a.e.  on  $(0,\infty)$. Moreover,  this  function  is
nonincreasing in  $x$ as $g$  is nonincreasing. Hence it  is identically
zero.       In        particular,       taking        $x=1$       yields
$    \excm{1}\left[Z_f    <    \infty\right]     =    0$,    and    thus
$\psij_{\rdtree}(f)=+\infty  $   a.s.  as  $Z_f$  under   $\excm{1}$  is
distributed as  $\psij_{\rdtree}(f)$.  The  same argument applies  if we
suppose that $f(y,u) = y^\alpha h(u)$ instead. This completes the proof.

\section{Phase transition for functionals of the mass and height}\label{phase transition section}
Recall that $\rddtree^n$ is a BGW($\xi$) conditioned to have $n$
vertices (with $n\in \supp$) and $\xi$ satisfies \ref{xi1} and
\ref{xi3}, with the sequence $(b_n, n\in \N^*) $ in \eqref{stable CLT}, 
 and that $\rdtree$ is a stable Lévy tree with branching mechanism $\psi(\lambda) = \kappa \lambda^\gamma$. In this section, we study the limit of
\[
\mathcal{A}_n^\circ(f)
=\frac{b_n}{n^{2}} \sum_{w \in
  \rddtree^{n,\circ}}|\rddtree^n_{w}|f\left(\frac{b_n}{n}
  \rddtree^n_{w},\frac{b_n}{n}H(w)\right)  
\]
for functions $f\in \cb(\T \times \R_+)$ continuous on $(\T\setminus
\TT)\times \R_+$ but  that may blow up as either the mass or the height goes to $0$. 

\subsection{A general convergence result}
We now give a first convergence result for general functionals that may
blow up. Recall from \eqref{T_0} the definition of $\TT$. Notice that
$\mathcal{A}_n^\circ(\TT\times \R_+) = 0$ and $\Psi_\rdtree(\TT\times \R_+) =
0$. 
\begin{proposition}\label{general functional}
	Assume that $\xi$ satisfies \ref{xi1} and \ref{xi3}. Let $f \in
    \cb(\T\times \R_+) $ be  continuous on $(\T
    \setminus \TT)\times \R_+$ and
    $\alpha,\beta \in \real$ with  $\gamma\alpha +
    (\gamma-1)(\beta+1)>0$ be such that 
\begin{equation}
   \label{eq:f<mh}
\left| f(\tree,r)\right| \leqslant C\, \m(\tree)^{\alpha}
\H(\tree)^{\beta}, \quad  \text{for all } \tree  \in \T\setminus
\TT \text{ and }  r\geqslant 0,
\end{equation}
for some finite constant $C >0$. Then $\Psi_\rdtree(|f|) $ is a.s. finite  and we have the convergence
 in distribution 
\begin{equation}
\label{general functional eq}
\mathcal{A}_n^\circ (f)=\frac{b_n}{n^2}\sum_{w\in \rddtree^{n,\circ}}|\rddtree^n_w|
f\left(\frac{b_n}{n}\rddtree^n_{w},\frac{b_n}{n}H(w)\right)\cvlaw
\Psi_{\rdtree}(f).  
\end{equation}
We also  have the convergence of all moments of order $p\geq 1$ such that
$p(\gamma\alpha + (\gamma-1)\beta)>1-\gamma$. 
\end{proposition}

\begin{proof}
	By Corollary \ref{convergence for continuous functions discrete}, we know that $\mathcal{A}_n^\circ \cvlaw \Psi_\rdtree$ in the space $\mathcal{M}(\T\times \R_+)$. In particular, the sequence $( \mathcal{A}_n^\circ, \, {n \in \supp})$ is tight (in distribution) in $\mathcal{M}(\T\times \R_+)$, and applying \cite[Theorem 4.10]{kallenberg2017random}, we have
	\begin{equation}\label{A_n is tight}
	\inf_{K \in \mathcal{K}} \sup_{n \in \supp} \ex{1\wedge \mathcal{A}_n^\circ(K^c)} = 0,
	\end{equation}
	where $\mathcal{K}$ is the set of all compact subsets of $\T\times \R_+$. We start by showing that
	\begin{equation}\label{A_n is tight without 1}
	\inf_{K \in \mathcal{K}} \sup_{n \in \supp} \ex{ \mathcal{A}_n^\circ(K^c)} = 0.
	\end{equation}
Let $K\in \mathcal{K}$. Using the inequality $x\leq  1 \wedge x + x \sqrt{1\wedge x}$ with
$x=\mathcal{A}_n^\circ(K^c) \geq 
0$ and the Cauchy--Schwartz inequality, we get that 
\begin{equation}
\label{CS} 
\ex{\mathcal{A}_n^\circ(K^c)} \leqslant \ex{1
        \wedge\mathcal{A}_n^\circ(K^c)} + \sqrt{\ex{\mathcal{A}_n^\circ
          (1)^2}\ex{1\wedge\mathcal{A}_n^\circ(K^c)}}. 
\end{equation}
	Since $\mathcal{A}_n^\circ(1) \leqslant \frac{b_n}{n} \H(\rddtree^n)$
	by \eqref{eq:mass-A_n}, Lemma \ref{height finite moments} implies that
\[
\sup_{n \in \supp}\ex{\mathcal{A}_n^\circ(1)^2} ^{1/2}\leqslant \sup_{n
  \in \supp}\ex{\left(\frac{b_n}{n}\H(\rddtree^n)\right)^2}^{1/2}< \infty  .
\] 
	This, in conjunction with \eqref{A_n is tight} and \eqref{CS}, proves \eqref{A_n is tight without 1}.
	
    Let     $\alpha,\beta     \in     \real$       such     that
    $\gamma\alpha  +  (\gamma-1)(\beta+1)>0$.   We  consider  the  space
    $S=\T       \times      \R_+$       with      the       metric
    $\rho((\tree,r),(\tree',r'))   =    \ghp(\tree,\tree')+|r-r'|$   and
    $S_0=\TT\times \R_+$, so that $(S, \rho)$ is a Polish metric
    space  and $S_0$  is  a  closed subset  of  $S$.  We shall  consider
    $0_S=(\{\root\},  0)\in  S_0$  as  a distinguished  point.   We  shall
    construct  a family  of  functions $\mf$  on $S$ satisfying  assumptions
    \ref{H1}--\ref{Hstar}  of Appendix  \ref{append} in  order to  apply
    Proposition       \ref{main       result,       general}.        Let
    $(\delta_k, \, k \in \N)$  be
    a  positive increasing      sequence      such      that
    $(2\gamma-1)\delta_k < (\gamma-1)+\left(\gamma \alpha + (\gamma-1)\beta\right)\wedge 0$ for all $k\in \N$.
    Define for every $k \in \N$
\[
  f_k(\tree,r) 
= \left(\m(\tree)^{\delta_k}\vee
  \m(\tree)^{-\delta_k}\right)\left(\H(\tree)^{\delta_k}\vee
  \H(\tree)^{-\delta_k}\right)
  \quad\text{and}\quad
  g_k(\tree,r) 
= \m(\tree)^{\alpha} \H(\tree)^{\beta} f_k(\tree,r) ,
\]
 for all $\tree \in \T \setminus \TT$
    and $r \geqslant 0$ and $f_k=g_k=+\infty $ on $\TT\times \R_+$. 
The functions $f_k$ and $g_k$ are positive and continuous on
$(\T\setminus \TT)\times \R_+$. We define $\mf=\{\ind\} \cup
\{f_k, g_k\colon\, k\in \N\}$. 
Therefore assumptions \ref{H1} and \ref{Hcont} are satisfied. 
Notice that  $\rho((\tree,r),S_0) = \ghp(\tree,\TT)$. 
Let $\epsilon >0$ and $M >0$. By \eqref{ghp trivial tree}, $\ghp(\tree,
\{\emptyset\})\leqslant M$ implies that $\H(\tree) \leqslant 2M$ and
$\m(\tree)\leqslant M$. Similarly, by Lemma \ref{distance to F},
$\ghp(\tree,\TT) \geqslant \epsilon$ implies that $\H(\tree)\geqslant
\epsilon$ and $\m(\tree)\geqslant \epsilon$. Therefore, we have the
inclusion 
\[
\left\lbrace (\tree,r) \in S\colon 
  \rho((\tree,r),S_0)\geqslant \epsilon, \,
  \rho((\tree,r),0_S)\leqslant M\right\rbrace
	\subset \{\tree \in \T\colon  \H(\tree)\in
    [\epsilon,2M],\,\m(\tree)\in [\epsilon, M]\}\times \R_+.
\]
	Since $f_k$ and $g_k$ are clearly bounded away from zero and infinity on the latter set, assumption
    \ref{Hbound} is satisfied.
Moreover, $f_k/f_{k+1}$ and $g_k/g_{k+1}$ are  continuous and bounded on
$S_0^c=(\T\setminus \TT)\times \R_+$ for every $k \in \N$. Recall that
$\rho((\tree,r),S_0) = \ghp(\tree,\TT)$. Therefore, as
$\rho((\tree,r),S_0) \to 0$, we have $\H(\tree)\wedge\m(\tree) \to0$ by
Lemma \ref{distance to F}. 
It follows that 
$f_k(\tree,r)/f_{k+1}(\tree,r) \to 0$ and $g_k(\tree,r)/g_{k+1}(\tree,r) \to 0$
as $\rho((\tree,r),S_0) \to 0+$. Recall the notation $\mf^\star(f) $ from
\ref{Hstar}. We deduce  that $f_{k+1} \in \mf^\star
(f_k)$ and $g_{k+1} \in \mf^\star
(g_k)$ for $k\in \N^*$. We also have that $1/f_1$ is continuous and bounded on
$S_0^c$ and that  $1/f_{1}(\tree,r) \to 0$
as $\rho((\tree,r),S_0) \to 0+$. This implies that $f_1\in \mf^\star
(\ind)$. Therefore, assumption \ref{Hstar} is satisfied.

 In order to apply Proposition \ref{main result, general} to the
 sequence of measures $(\mathcal{A}_n^\circ, \, {n \in \supp})$ and the
 family $\mf  $, we shall check that the sequence $(\mathcal{A}_n^\circ,
 \, {n \in \supp})$ is tight (in distribution) in the space $\MF$ (see
 Appendix \ref{append} for the definition of $\MF$). Thanks to  Proposition
 \ref{tightness}, the sequence   $(\mathcal{A}_n^\circ,\,
 {n \in \supp})$ is tight  in the space $\MF$ if and
 only if  
 $(\f\mathcal{A}_n^\circ, \, n\in \supp)$ is tight in $\M(S)$ for all
 $\f\in \mf$. 
 Let $\f\in \mf$. 
 Notice
 that for every $\tree \in \T\setminus \TT$ and $r \geqslant 0$, we have 
\[
\f((\tree,r))\leqslant \sum_{1\leq i,j \leq 2}\m(\tree)^{\alpha_i}
\H(\tree)^{\beta_j}
\]
for $\alpha_1, \alpha_2, \beta_1, \beta_2\in \R$ such that
 $\gamma\alpha_i +  (\gamma-1)(\beta_j+1)>0$ holds for every $i,j \in \{1,2\}$.
	Therefore, by Lemma \ref{lem:Amh^p}, we have for some $p>1$ small enough
\begin{equation}
\label{A_n bounded in L1}
\sup_{n \in \supp} \ex{\mathcal{A}_n^\circ(\f)^p} < \infty
\quad\text{and}\quad
\sup_{n \in \supp} \ex{\mathcal{A}_n^\circ(\f^p)} < \infty. 
\end{equation}
The first bound gives that \eqref{tight1} holds for all $\f\in \mf$ by
the Markov inequality. Recall that $\mathcal{K}$ denotes the set of compact
subsets of $\T\times \R_+$. 
Moreover, with $q$  such that $1/p+1/q = 1$ and $K\in \mathcal{K}$,
using Hölder's inequality, we get
\[
\ex{\mathcal{A}_n^\circ(\f \ind_{K^c})}
\leq  \ex{\mathcal{A}_n^\circ( \ind_{K^c})}^{1/q} 
\ex{\mathcal{A}_n^\circ(\f^p)}^{1/p}.
\]
Using the second bound in \eqref{A_n bounded in L1} and \eqref{A_n is
  tight without 1}, we deduce that 
\[
\inf_{K\in \mathcal{K} } \sup_{n \in \supp} \ex{\mathcal{A}_n^\circ (\f
  \ind_{K^c})}
= 0.
\]
Thus \eqref{tight2} holds for all $\f\in
\mf$. 
According to Proposition \ref{tightness}-(i), we get that 
the sequence $(
\mathcal{A}_n^\circ, \, {n \in \supp})$ is tight (in distribution) in
$\MF (\T\times \R_+)$. Now apply Proposition \ref{main result, general}
and Proposition \ref{convergence in distribution} to get that 
\[
\mathcal{A}_n^\circ(\f h) \cvlawd \Psi_\rdtree(\f h)
\]
	for every $h\in \cc_b(\T\times\R_+)$ and every $\f\in \mf$. Let
    $f\in \cb(\T\times \R_+)$ satisfying the assumptions of Proposition
    \ref{general functional}. 
Consider
    $\f = g_1$ and $h = f/g_1$. 
Notice that \eqref{eq:f<mh} implies that $h$ is continuous on $\T\times\R_+$. 
Since $\f h= g_1 h = f$ except possibly on $S_0 = \TT \times \R_+$ and
$\mathcal{A}_n^\circ(S_0) = \Psi_{\rdtree}(S_0) = 0$, we deduce that the
convergence in distribution \eqref{general functional eq} holds.
	
Let $p >1$ such that $p(\gamma\alpha+(\gamma-1)\beta)>1-\gamma$. There exists $q>p$ satisfying the same inequality. Since $|f(\tree,r) |\leqslant C \m(\tree)^\alpha \H(\tree)^\beta$, we get that
	\begin{equation}
	\sup_{n \in \supp} \ex{|\mathcal{A}_n^\circ(f)|^{q}} \leqslant C^q \sup_{n \in \supp}\ex{\left(\frac{b_n^{1+\beta}}{n^{2+\alpha +\beta}} \sum_{w \in \rddtree^{n,\circ}} |\rddtree^n_{w}|^{1+\alpha} \H(\rddtree^n_{w})^\beta\right)^{q}\,},
	\end{equation}
	where the right-hand side is  finite by Lemma \ref{lem:Amh^p}. Thus,
    the  sequence  $(|\mathcal{A}_n^\circ(f)|^p,  \, n  \in  \supp)$  is
    uniformly integrable and the convergence  of the moment of order $p$
    of  $\mathcal{A}_n^\circ(f)$  towards the  moment  of  order $p$  of
    $\Psi_{\rdtree}(f)$ readily  follows from  \eqref{general functional
      eq}.
\end{proof}

\subsection{Phase transition for functionals of the mass and height}
We refine the convergence result given in Proposition \ref{general functional} for functionals depending only on the mass and height and describe a phase transition in that case.

We start with a technical lemma which is a consequence of the well-known de La Vallée Poussin criterion for uniform integrability.
\begin{lemma}\label{existence f*}
  Let   $\nu$  be   a  nonnegative   finite  measure   on  $(0,1]$   and
  $f\in  \cc_+( (0,1])$  be nonincreasing,  belonging to  $L^1(\nu)$ and
  such  that  $\lim_{x \to  0+}f(x)  =  +\infty$.  Then there  exists  a
  positive  function   $\*{f}\in  \cc_+((0,  1])  $   which  belongs  to
  $L^1(\nu)$,   such  that   $f/\*{f}$   is  bounded   on  $(0,1]$   and
  $\lim_{x \to 0+} f(x)/\*{f}(x) = 0$.
\end{lemma}
\begin{proof}
	We may assume without loss of generality that $f$ does not vanish anywhere in $(0,1]$ and that $\nu$ is a probability measure. By the de La Vallée Poussin criterion (see \cite[§22]{dellacherie1975probabilites}), there exists a convex nondecreasing function $F \colon \R_+ \to \R_+$ such that $\lim_{t \to \infty} F(t)/t = \infty$ and $F\circ f \in L^1(\nu)$. In fact, up to considering $F+1$ instead, we can and will assume that $F$ does not vanish anywhere. Since $F$ is convex on $\R_+$, it is continuous on $(0,\infty)$ and it follows that $F \circ f$ is continuous on $(0,1]$. Moreover, $F\circ f$ is clearly nonincreasing by composition. Further, since $\lim_{x\to 0}f(x) = \infty$ and $\lim_{t\to \infty} t/F(t) = 0$, we get $\lim_{x \to 0}f(x)/F\circ f(x) = 0$. The function $f/F\circ f$ being continuous on $(0,1]$ with a finite limit at $0$, it is bounded on $(0,1]$. Setting $\*{f} = F\circ f$, the conclusion readily follows.
	\end{proof}

We now give the main result of this section. Recall that the notation
$\psij_{\rdtree}(g(x)h(u))$ stands for $\psij_\rdtree(f)$ where $f(x,u)
= g(x)h(u)$. For $g\in \cb(\R_+) $, define 
\begin{equation}
g^*(x) \coloneqq \sup_{x\leqslant y\leqslant 1} |g(y)| \quad \text{ for
all } x  \in (0,1].
\end{equation} 
\begin{thm}\label{mass+height convergence}Assume that $\xi$ satisfies \ref{xi1} and \ref{xi3}.
	\begin{enumerate}[label = (\roman*),leftmargin=*]
	\item Let $\beta \in \real$ and $g \in \cb([0,1])$ be  such that $g$ is continuous on $(0,1]$ and satisfies	
	\begin{equation}\int_0 g^*(x^{\gamma/(\gamma-1)}) x^\beta\, \dd x < \infty.
	\end{equation}
	Then we have the convergence in distribution and of the first moment
	\begin{equation}\frac{b_n^{1+\beta}}{n^{2+\beta}}\sum_{w \in \rddtree^{n,\circ}} |\rddtree^n_{w}| \H(\rddtree^n_{w}) ^\beta g\left( \frac{|\rddtree^n_{w}|}{n}\right)\cvlawmean \psij_\rdtree(g(x)u^\beta)
	\end{equation}
	where $\psij_\rdtree(|g(x)|u^\beta)$ is a.s. finite and integrable.
	\item Let $\alpha \in \real$ and $h \in \cb(\R_+)$ be   such that $h$ is
      continuous on $(0,\infty)$ and satisfies $h(u) = O(\e^{u^\eta})$
      as $u \to \infty$ for some  $\eta\in (0,\gamma)$ and  
\begin{equation}
\int_0 x^{\alpha \gamma/(\gamma -1)}  h^*(x)\, \dd x < \infty.
\end{equation}
	Then we have the convergence in distribution and of the first moment
\begin{equation}
\label{eq:a-h}
\frac{b_n}{n^{2+\alpha}}\sum_{w \in \rddtree^{n,\circ}} |\rddtree^n_{w}|^{1+\alpha} h\left(\frac{b_n}{n}\H(\rddtree^n_{w})\right) \cvlawmean \psij_\rdtree(x^\alpha h(u))
\end{equation}
	where $\psij_\rdtree(x^\alpha |h(u)|)$ is a.s. finite and integrable.
	\item Let $f\in \cb_+([0,1]\times \R_+)$  be such that 
	\begin{equation}
	\int_0 f(x^{\gamma/(\gamma-1)},x)\, \dd x =\infty.
	\end{equation}
	Suppose that $f$ is of the form $f(x,u) = g(x)u^\beta$ or $f(x,u) =
    x^\alpha h(u)$ where $\alpha,\beta\in \real$ and $g,h$ are
    nonincreasing and continuous on $(0,1]$ and  on $(0,\infty)$ respectively. Then we have
	\begin{equation}\label{divergence mass+height}
	\frac{b_n}{n^2} \sum_{w\in \rddtree^{n,\circ}} |\rddtree^n_w| f\left(\frac{|\rddtree^n_w|}{n}, \frac{b_n}{n} \H(\rddtree^n_w)\right) \cvlawmean \infty.
	\end{equation}
\end{enumerate}
\end{thm}

\begin{proof}To prove (i), we proceed in three steps.\\
	\textbf{Step 1 in the proof of (i).} Let $g\in \cc_+([0, 1])$ be
    nonincreasing and nonzero. 
 Let
    $(\beta_k, \, k \in \mathbb{N})$ be a decreasing sequence of
    nonpositive real numbers such that $\beta_0 = 0$ and $\lim_{k \to
      \infty} \beta_k = -1$. We define a set of functions
    $\mf = \{h_k\colon \, k \in \mathbb{N}\}$ where  $h_k(u) =
    u^{\beta_k} \vee u^k$ for $u >0$ and $k \in \mathbb{N}$, and
    $h_0(0)=1$ and $h_k(0)=+\infty $ for $k\in \N $. We shall
    prove that $\mathfrak{F}$ satisfies assumptions \ref{H1}--\ref{H5}
    of Appendix \ref{append} with $S = \R_+$ equipped with the
    Euclidean distance and $S_0 = \{0\}$. Notice that $h_0 \equiv 1$ and
    $h_k $ is continuous on $S_0^c$ for every $k
    \in \mathbb{N}$, so \ref{H1} and \ref{Hcont} are
    satisfied. Moreover, for every $k \in \mathbb{N}$, the function
    $h_k/h_{k+1}$ is continuous on $(0,\infty)$ and we have 
\[
\lim_{u \to 0+} \frac{h_k(u)}{h_{k+1}(u)} = \lim_{u \to
      0+}u^{\beta_k - \beta_{k+1}} = 0 \quad \text{and} \quad \lim_{u \to
      +\infty} \frac{h_k(u)}{h_{k+1}(u)} = \lim_{u \to
      +\infty}\frac{1}{u}= 0,
\] 
	so that \ref{Hstar} and \ref{H5} are satisfied. Finally, since the
    set $\{x \in S \colon \, \rho(x,S_0) \geqslant \epsilon, \,
    \rho(x,0)\leqslant M\} = [\epsilon, M]$ is compact and $h_k$ is
    continuous, it is bounded there and \ref{Hbound} is
    satisfied. Define a (random) measure on $\R_+$ by setting 
\begin{equation}
\zeta_n (h) = \frac{b_n}{n^2} \sum_{w \in \rddtree^{n,\circ}}
|\rddtree^n_w|
g\left(\frac{|\rddtree^n_w|}{n}\right)h\left(\frac{b_n}{n}\H(\rddtree^n_w)\right) 
\end{equation}
	for every $h\in \cb_+(\R_+ )$. By \eqref{convergence for continuous functions
      discrete, mass+height}, $\zeta_n$ converges to $\zeta$ in
    distribution in $\M(\R_+)$ and $\ex{\zeta_n(\bullet)}$
    converges to $\ex{\zeta(\bullet)}$ in $\M(\R_+)$ where $\zeta$
    is defined by $\zeta(h) = \psij_\rdtree(g(x)h(u))$. 
	But, since we have $\int_0 g(x) x^{(\beta_k+1)(1-1/\gamma)-1}\, \dd
    x < \infty$ for every $k \in \mathbb{N}$, Lemma \ref{mass+height
      bounded in L1}-(i) gives 
\[
\sup_{n \in \supp} \ex{\zeta_n(h_k)} \leqslant \sup_{n \in \supp}
    \ex{\zeta_n(u^{\beta_k})} + \sup_{n \in \supp}\ex{\zeta_n(u^k)}<
    \infty \quad \text{ for all } k \in \mathbb{N}.
\]
	Thus, Corollary \ref{main result, special} yields the convergence in
    distribution $\zeta_n \cvlaw \zeta$ in $\MF$ as well as the
    convergence of the first moment $\ex{\zeta_n(\bullet)} \to
    \ex{\zeta(\bullet)}$ in $\MF$. By Proposition \ref{convergence in
      distribution}, this implies that for every $g\in \cc_+([0, 1])$ 
    nonincreasing and every $\beta >
    -1$, we have 
	\begin{equation}\frac{b_n^{1+\beta}}{n^{2+\beta}}\sum_{w \in \rddtree^{n,\circ}} |\rddtree^n_{w}| \H(\rddtree^n_{w}) ^\beta g\left( \frac{|\rddtree^n_{w}|}{n}\right)\cvlawmean \psij_\rdtree(g(x)u^\beta). \label{g beta}
	\end{equation}
	
	\noindent
	\textbf{Step 2 in the proof of (i).} Now fix $\beta > -1$ and define the (random) measure $\xi_n$ on $[0,1]$ by
	\begin{equation}\xi_n(g) = \frac{b_n^{1+\beta}}{n^{2+\beta}}\sum_{w \in \rddtree^{n,\circ}} |\rddtree^n_{w}| \H(\rddtree^n_{w}) ^\beta g\left( \frac{|\rddtree^n_{w}|}{n}\right),
	\end{equation}
	for every  $g \in \cb_+([0,1])$. Notice that \eqref{g beta} can be
    rewritten as 
	\begin{equation}\xi_n(g) \cvlawmean \xi(g) \label{xi g}
	\end{equation}
 	for every $g\in \cc_+([0,1])$ nonincreasing, where the measure $\xi$
    is defined by $\xi(g) = \psij_\rdtree(g(x)u^\beta)$. Moreover, Lemma
    \ref{mass+height bounded in L1}-(i) applied  with $g \equiv 1$ gives
    $\sup_{n \in \supp}  \ex{\xi_n(1)} < \infty$.  As  a consequence, by
    the           Markov          inequality,           we          have
    $\lim_{r  \to \infty}  \sup_{n \in  \supp} \pr{\xi_n(1)  >r }  = 0$.
    Since $[0,1]$  is compact,  this means that  the sequence  of random
    measures $(\xi_n,  \, {n  \in \supp})$ is  tight in  distribution in
    $\M([0,1])$,  see \cite[Theorem  4.10]{kallenberg2017random}. Hence,
    it  is  relatively  compact  by Prokhorov's  theorem  as  the  space
    $\M([0,1])$ is  Polish for the  weak topology. Let $\hat{\xi}$  be a
    limit  point. Then  we  have $\xi(g)  \law  \hat{\xi}(g)$ for  every
     $g \in \cc_+([0, 1])$ nonincreasing. Therefore, we get
    that     $\xi      \law     \hat{\xi}$     and      the     sequence
    $(\xi_n, \, {n \in \supp})$ has only one limit point $\xi$. Since it
    is relatively compact, we deduce  that $\xi_n$ converges to $\xi$ in
    distribution in $\M([0,1])$. A  similar deterministic argument shows
    that  $\ex{\xi_n(\bullet)}$  converges   to  $\ex{\xi(\bullet)}$  in
    $\M([0,1])$.
	
	\noindent
	\textbf{Step 3 in the proof of (i).} Let $\beta>-1$ and $g
    \in \cb([0,1]) $ be
 continuous on $(0,1]$, nonzero and such that  $\int_0 g^*(x)
    x^{(\beta+1)(1-1/\gamma)-1} \, \dd x < \infty$. Set $g_0 \equiv
    1$. If $\lim_{x \to 0} g^*(x) = \infty$, set $g_1 = g^*+1$. If $g^*$
    has a finite limit at $0$ (which is then positive), then  there exists $\epsilon >0$ such that $\int_0
    x^{-\epsilon} g^*(x) x^{(\beta+1)(1-1/\gamma)-1}\, \dd x <
    \infty$. We also have $\lim_{x\to 0+} x^{-\epsilon} g^*(x) =
    \infty$ and the function $x \mapsto x^{-\epsilon} g^*(x)$ is
    continuous and nonincreasing. In that case, we set $g_1(x) =
    x^{-\epsilon} g^*(x)+1$ for $x\in [0, 1]$.

    Define a set of functions $\mf  = \{g_k\colon \, k \in \mathbb{N}\}$
    as follows: for every $k \geqslant 1$, set $g_{k+1} = \*{g}_k$ which
    is given by Lemma \ref{existence f*} applied with the finite measure
    $\nu(\dd x)  = x^{(\beta+1)(1-1/\gamma)-1}\dd x$.   By construction,
    the      sequence      $\mathfrak{F}$     satisfies      assumptions
    \ref{H1}--\ref{Hstar}  of Appendix  \ref{append} with  $S =  [0,1]$,
    $S_0  =  \{0\}$  and  $\mf^\star(g_k)=\{g_j\colon\,  j>k\}$  (notice
    \ref{Hbound}    is   automatically    satisfied   as    $[0,1]$   is
    compact). Notice  that, by  Lemma \ref{existence f*},  for every
    $k   \in  \mathbb{N}$,   the  function   $g_k$  is   continuous  and
    nonincreasing         on        $(0,1]$         and        satisfies
    $\int_0 g_k(x) x^{(\beta+1)(1-1/\gamma)-1}\, \dd x < \infty$. So, by
    Lemma \ref{mass+height bounded in L1}, we get that
\[
\sup_{n \in \supp} \ex{\xi_n(g_k)} < \infty \quad \text{ for all } k \in
\mathbb{N}.
\]
	Now, Corollary \ref{main result, special} applies and yields, in
    conjunction with Proposition \ref{convergence in distribution}, the
    convergence in distribution and of the first moment 
\[
\xi_n(g_k \ell) \cvlawmean \xi(g_k \ell)
\]
	for every $k \in \mathbb{N}$ and  $\ell \in \cc([0, 1])$.  Now apply this with $k = 1$ and $\ell =
    g/g_1$. Notice that $g_1 \ell = g$ except possibly on $S_0 =
    \{0\}$. Since $\xi_n(S_0) = \xi(S_0) =0$, we deduce that 
\[
\xi_n(g) \cvlawmean \xi(g).
\]
This, together with Proposition \ref{joint phase transition}, proves
 (i).

	The proof of (ii) is quite similar so we only indicate the changes compared with (i).\\
	\textbf{Step 1 in the proof of (ii).}  Let $h\in \cc_+(\R^+)
    $ be  nonincreasing and nonzero.

 Taking a decreasing sequence $(\alpha _k , \, k \in \mathbb{N})$ of
 nonpositive real numbers such that $\alpha_0 = 0$ and $\lim_{k \to
   \infty} \alpha_k = -1+ 1/\gamma$
 and defining a set of functions $\mf = \{g_k\colon \, k \in
 \mathbb{N}\}$ by $g_k(x) = x^{\alpha_k}$, we can show that for every
 $h\in \cc_+(\R_+)$  nonincreasing  and every $\alpha >
 -1+1/\gamma$, 
 we have
	\begin{equation}\frac{b_n}{n^{2+\alpha}}\sum_{w \in \rddtree^{n,\circ}} |\rddtree^n_{w}|^{1+\alpha} h\left(\frac{b_n}{n}\H(\rddtree^n_{w})\right) \cvlawmean \psij_\rdtree(x^\alpha h(u)). \label{h alpha}
	\end{equation}
	
	\noindent
	\textbf{Step 2 in the proof of (ii).} Fix $\alpha > -1+1/\gamma$ and define the (random) measure $\xi_n$ on $\R_+$ by
	\begin{equation}\xi_n(h) = \frac{b_n}{n^{2+\alpha}}\sum_{w \in \rddtree^{n,\circ}} |\rddtree^n_{w}|^{1+\alpha} h\left(\frac{b_n}{n}\H(\rddtree^n_{w})\right),
	\end{equation}
	for every  $h\in \cb_+( \R_+)$. Notice that \eqref{h alpha} can be rewritten as
	\begin{equation}
\label{xi h}
\xi_n(h) \cvlawmean \xi(h) 
	\end{equation}
	for every $h\in \cc_+(\R_+)$  nonincreasing, where the measure $\xi$
    is  defined by  $\xi(h) =  \psij_\rdtree(x^\alpha h(u))$.  Moreover,
    Lemma \ref{mass+height bounded in L1}-(ii) applied with $h \equiv 1$
    gives   $\sup_{n  \in   \supp}  \ex{\xi_n(1)}   <  \infty$.    As  a
    consequence,     by    the     Markov     inequality,    we     have
    $\lim_{r \to \infty} \sup_{n \in \supp} \pr{\xi_n(1) >r } = 0$.  Fix
    $\beta   >0$  and   let   $r  >0$.   Then,   using  the   inequality
    $\ind_{[r,\infty)}(u)    \leqslant   (u/r)^\beta    $   for    every
    $u \geqslant 0$, we get
\[
\sup_{n \in \supp}\E\left[\xi_n([r,\infty))\right] \leqslant \frac{1}{r^\beta}\sup_{n
  \in \supp} \E\left[\costj_n(x^\alpha u^\beta)\right].
\]
	Notice that the right-hand side is finite by Lemma \ref{lem:Amh^p} since $\gamma \alpha + (\gamma-1)(\beta+1) >0$. We deduce that
	\[\inf_{K \subset \R_+} \sup_{n \in \supp} \ex{\xi_n(K^c)} = 0,\]
	where the infimum is taken over all compact subsets $K \subset
    \R_+$. By \cite[Theorem 4.10]{kallenberg2017random}, this means that
    the sequence of random measures $(\xi_n, \, {n \in \supp})$ is tight
    in distribution in $\M(\R_+)$. Following the end of step 2 for
    property (i), we are then able to show that $\xi_n$
    converges to $\xi$ in distribution in $\M([0,\infty ))$ and
    $\ex{\xi_n(\bullet)}$ converges to $\ex{\xi(\bullet)}$ in
    $\M([0,\infty ))$. 
	
	\noindent
	\textbf{Step  3  in  the  proof  of  (ii).}  Let  $h \in  \cb(\R_+)$  be
    continuous   on   $(0,\infty)$   such  that   $h^*$   is   non-zero,
    $\int_0 h^*(u) u^{\alpha \gamma/(\gamma - 1)} \, \dd u < \infty$ and
    $h(u)   =    O(\e^{u^\eta})$   as   $u   \to    \infty$   for   some
    $\eta  \in (0,\gamma)$.  Set $h_0  \equiv 1$  and define  a positive
    function  $h_1  \in  \cb_+((0,\infty))  $ in  the  following  way.  If
    $\lim_{u \to  0}h^*(u) = \infty$, set  $h_1 = h^*+1$ on  $(0,1]$. If
    $h^*$  has a  finite limit  at $0$  (which is  positive as  $h^*$ is
    non-zero),  then  $\alpha  >-1+1/\gamma$,   and  thus  there  exists
    $\epsilon                >0$                such                that
    $\int_0  u^{-\epsilon}h^*(u)u^{\alpha\gamma/(\gamma-1)}\,  \dd  u  <
    \infty$.
    Moreover, we have $\lim_{u\to 0}  u^{-\epsilon} h^*(u) = \infty$ and
    the  function  $u  \mapsto u^{-\epsilon}h^*(u)$  is  continuous  and
    nonincreasing. In that case,  we set $h_1(u)= u^{-\epsilon}h^*(u)+1$
    for $u  \in (0,1]$.  Now extend  $h_1$ to  a continuous  function on
    $(0,\infty)$   such    that   $h_1(u)   =    \exp(u^{\eta_1})$   for
    $u \geqslant 2$ for some $\eta_1 \in (\eta,\gamma)$. Define a set of
    functions  $\mf  =   \{h_k\colon  \,  \,  k   \in  \mathbb{N}\}$  as
    follows. Let $(\eta_k, \, k  \geqslant 2)$ be an increasing sequence
    in $ (\eta_1,\gamma)$. Recall that  $\alpha > -1+1/\gamma$ so that the
    measure
    $\nu(\dd u) = \ind_{(0,  1]} (u)\, u^{\alpha\gamma/(\gamma-1)}\dd u$
    is      finite.       For      every     $k\geq      1$,      define
    $h_{k+1}  \in   \cb_+([0,\infty  ))$  continuous  and   positive  on
    $(0,\infty )$  and such  that $h_{k+1} =  \*{h}_k$ on  $(0,1]$, with
    $\*{h}_k$    defined    in    Lemma    \ref{existence    f*},    and
    $h_{k+1}(u)  =  \exp(u^{\eta_{k+1}})$  for   $u  \geqslant  2$.   In
    particular,                          we                         have
    $\lim_{x   \rightarrow  0+}h_k(x)/h_{k+1}(x)=   \lim_{x  \rightarrow
      +\infty                                     }h_k(x)/h_{k+1}(x)=0$.
    Then, it is easy to check that the sequence $\mathfrak{F}$ satisfies
    assumptions   \ref{H1}--\ref{H5}  of   Appendix  \ref{append}   with
    $S = \R_+$, $S_0 =  \{0\}$ and $\mf^\star(h_k)=\{h_j\colon \, j>k\}$
    for $k\in \N$.  Notice that, by Lemma \ref{existence  f*}, for every
    $k   \in  \mathbb{N}$,   the  function   $h_k$  is   continuous  and
    nonincreasing         on        $(0,1]$         and        satisfies
    $\int_0 h_k(u)u^{\alpha\gamma/(\gamma - 1)} \,  \dd u < \infty$. So,
    by Lemma \ref{mass+height  bounded in L1} (i) and (ii),  we get that
    for all $k\in \N$, there exists a finite constant $C_k>0$ such that
\[
\sup_{n \in \supp} \ex{\xi_n(h_k)} \leqslant \sup_{n \in \supp} 
    \ex{\xi_n({h_k}\ind_{(0,1]})}+ C_k\sup_{n \in
      \supp}\ex{\xi_n(\exp(u^{\eta_k})\ind_{\lbrace u\geqslant
        1\rbrace})}< \infty.
\] 
 Now, Corollary \ref{main result, special} applies and yields, in
 conjunction with Proposition \ref{convergence in distribution}, the
 convergence in distribution and of the first moment 
	\[\xi_n(h_k f) \cvlawmean \xi(h_k f)\]
	for every $k \in \mathbb{N}$ and every $f\in \cc( \R_+
    )$. Taking $k = 1$ and $f = h/h_1$ proves \eqref{eq:a-h}
 as $\xi_n(S_0) = \xi(S_0) = 0$. This, together with Proposition \ref{joint phase transition}, proves
 (ii).  
	
	To prove (iii), notice that by \eqref{convergence for continuous functions discrete, mass+height} we have the convergence in distribution $\costj_n \cvlaw \psij_\rdtree$ in the space $\M([0,1]\times \R_+)$. Thanks to Skorokhod's representation theorem, we may assume that we have a.s. convergence. Thus, we get that a.s. for every $k \in \mathbb{N}$,
	\begin{equation*}
		\lim_{n\to \infty}\frac{b_n}{n^2} \sum_{w\in \rddtree^{n,\circ}} |\rddtree^n_w| \left(f\left(\frac{|\rddtree^n_w|}{n}, \frac{b_n}{n} \H(\rddtree^n_w)\right)\wedge k\right)  = \psij_{\rdtree}(f\wedge k).
	\end{equation*}
	Therefore, we have for $k\in \N$
	\begin{equation}\label{divergence liminf}
	\liminf_{n \to \infty} 	\frac{b_n}{n^2} \sum_{w\in \rddtree^{n,\circ}} |\rddtree^n_w| f\left(\frac{|\rddtree^n_w|}{n}, \frac{b_n}{n} \H(\rddtree^n_w)\right) \geqslant \psij_{\rdtree}(f\wedge k).
	\end{equation}
	But by the monotone convergence theorem and Proposition \ref{joint
      phase transition}, we have that a.s. $\lim_{k \to \infty}
    \psij_{\rdtree}(f\wedge k) = \psij_\rdtree(f) = \infty$. Thus,
    \eqref{divergence mass+height} follows from \eqref{divergence
      liminf} by letting $k $ go to infinity. 
\end{proof}

Recall from \eqref{A general} that we
excluded the leaves to be able to consider functions taking infinite
values on trees whose  height vanishes. In the particular case
where the function only blows up as the mass goes to zero, one can get
rid of this restriction. 
\begin{remark}\label{with vs without leaves}Recall the definition of the random measure $\costj_n \in \M([0,1]\times \R_+)$:
\[
\costj_n(f) = \frac{b_n}{n^2}\sum_{w \in \rddtree^{n,\circ}}
    |\rddtree_w^n| f\left(
      \frac{|\rddtree^{n}_w|}{n},\frac{b_n}{n}\H(\rddtree_w^n)\right). 
\]
Similarly to the measure $\costj_n$, we define the measure 
$\costmh_n\in \M([0,1]\times \R_+)$, where the sum is over all the
vertices (the internal vertices 
and the leaves): for $f\in \cb_+([0, 1] \times \R_+)$
\[
\costmh_n(f) = \frac{b_n}{n^2}\sum_{w \in \rddtree^{n}}
    |\rddtree_w^n| f\left(
      \frac{|\rddtree^{n}_w|}{n},\frac{b_n}{n}\H(\rddtree_w^n)\right). 
\]
	Let $\beta \geqslant 0$ and $g\in \cb([0, 1])$ 
 such that $g$ is continuous on $(0,1]$ and $\int_0 g^*(x^{\gamma/(\gamma-1)})x^\beta\, \dd x < \infty$. By Theorem \ref{mass+height convergence}-(i), we have 
	\begin{equation}\label{with leaves eq}
	\costj_n(g(x)u^\beta) \cvlawmean \psij_\rdtree(g(x)u^\beta). 
	\end{equation}
	Now note that
	\begin{equation*}
\costmh_n (g(x)u^\beta) = \frac{b_n^{1+\beta}}{n^{2+\beta}}\sum_{w \in
  \rddtree^n} |\rddtree^n_w|\H(\rddtree^n_w)^\beta
g\left(\frac{|\rddtree^n_w|}{n}\right)
\end{equation*}
	makes sense when the function $g$ blows up at $0$. If $\beta >0$, we
    have $\mathcal{A}_n^{\mathfrak{mh}}(g(x)u^\beta) = \costj_n (g(x)u^\beta)$ since $\H(\rddtree^n_w) = 0$ for every leaf $w \in \operatorname{Lf}(\rddtree^n)$. Thus we only need to consider the case $\beta = 0$. Then, using \eqref{b_n bounded} and the fact that $|\operatorname{Lf}(\rddtree^n)| \leqslant n$ and that $|\rddtree^n_w| = 1$  for every $w \in \operatorname{Lf}(\rddtree^n)$, we have 
	\begin{equation*}
\left|\costmh_n(g(x)) - \costj_{n}(g(x)) \right| = \frac{b_n}{n^2}
\left|\sum_{w \in \mathrm{Lf}(\rddtree^n)}|\rddtree^n_{w}| g\left(
    \frac{|\rddtree^n_{w}|}{n}\right)\right| \leqslant  \overline{b}
\:\! n^{-1+1/\gamma}g^*\left(\frac 1 n \right).
\end{equation*}
	Since $g^*$ is nonincreasing and satisfies $\int_0
    g^*(x^{\gamma/(\gamma-1)})\, \dd x < \infty$, it is straightforward
    to check that $g^*(x) = o(x^{1/\gamma - 1})$ as $x \to 0$. Thus, we
    deduce that $\lim_{n\to
      \infty}\costmh_n(g(x)u^\beta) -
    \costj_{n}(g(x)u^\beta) =0$ a.s. and in $L^1(\P)$. 
	As a consequence, the convergence \eqref{with leaves eq} still holds
    if we replace $\costj_n(g(x)u^\beta)$ by
    $\costmh_n(g(x)u^\beta)$.

Similarly, let $\alpha > -1+1/\gamma$ and $h \in \cc(\R_+)$ such that $h(u) = O(\e^{u^\eta})$ as $u \to \infty$ for some $\eta \in (0,\gamma)$. Then $h^*$ is bounded near $0$ and necessarily $\int_0 x^{\alpha\gamma/(\gamma-1)}h^*(x)\, \dd x < \infty$. Thus, by Theorem \ref{mass+height convergence}, we have
\begin{equation}\label{with leaves eq2}
\costj_n(x^\alpha h(u)) \cvlawmean \psij_\rdtree(x^\alpha h(u)). 
\end{equation}
Furthermore, using \eqref{b_n bounded} we have
\begin{equation*}
\left|\costmh_n(x^\alpha h(u)) - \costj_{n}(x^\alpha h (u)) \right| =
\frac{b_n}{n^{2+\alpha}}\left|\operatorname{Lf}(\rddtree^n)
\right|\left|h(0)\right|\leqslant  \overline{b} \:\!
n^{-\alpha-1+1/\gamma}|h(0)|. 
\end{equation*}
Thus, we deduce that $\lim_{n\to \infty}\costmh_n(x^\alpha h(u)) -
\costj_{n}(x^\alpha h(u)) =0$ a.s. and in $L^1(\P)$ and the convergence
\eqref{with leaves eq2} holds for $\costmh_n(x^\alpha h(u))$. 
\end{remark}

\begin{example}\label{power log}
	Fix $\alpha >-1+1/\gamma$ and set $g(x) = |\log (x)| x^{\alpha}$. It is clear that $\int_0 g(x^{\gamma/(\gamma-1)})  \, \dd x < \infty$, so by Theorem \ref{mass+height convergence} we have  the convergence in distribution
	\[\costj_n(g(x)) \cvlawd \psij_\rdtree(g(x)).\]
	But notice that
	\begin{align*}
\costj_n(g(x))  
&= \frac{b_n \log (n)}{n^{2+\alpha}} \sum_{w \in \rddtree^{n,\circ}}
  \left|\rddtree^n_w\right|^{1+\alpha}- \frac{b_n}{n^{2+\alpha}} \sum_{w
  \in \rddtree^{n,\circ}} \left|\rddtree^{n}_w\right|^{1+\alpha} \log
  \left|\rddtree^n_w\right| \notag \\
	&=  \log (n) \,\costj_n(x^{\alpha})-\frac{b_n}{n^{2+\alpha}} \sum_{w
      \in \rddtree^{n,\circ}} \left|\rddtree^{n}_w\right|^{1+\alpha}
      \log \left|\rddtree^n_w\right|. 
	\end{align*}
Again Theorem \ref{mass+height convergence} gives the convergence in distribution $\costj_n(x^{\alpha}) \cvlaw \psij_{\rdtree}(x^{\alpha})$. Therefore, we get the following asymptotic expansion in distribution
	\[\frac{b_n}{n^{2+\alpha}} \sum_{w \in \rddtree^{n,\circ}} \left|\rddtree^{n}_w\right|^{1+\alpha} \log \left|\rddtree^n_w\right| \lawd \log (n)\, \psij_{\rdtree}(x^{\alpha}) - \psij_{\rdtree}(|\log (x)| x^{\alpha}) + o(1).\]
	Furthermore, since
	\[\lim_{n\to \infty}\ex{\costj_n(g(x)) }=\ex{\psij_{\rdtree}(g(x)) }\quad \text{and} \quad \lim_{n\to \infty}\ex{\costj_n(x^{\alpha})}= \ex{\psij_{\rdtree}(x^{\alpha})},\]
	we get the corresponding asymptotic expansion for the first moment
	\[\frac{b_n}{n^{2+\alpha}} \ex{\sum_{w \in \rddtree^{n,\circ}} \left|\rddtree^{n}_w\right|^{1+\alpha} \log \left|\rddtree^n_w\right|} = \log (n)\ex{ \psij_{\rdtree}(x^{\alpha})} - \ex{\psij_{\rdtree}(|\log (x)| x^{\alpha})} + o(1).\]
\end{example}

\appendix
\section{A space of measures}\label{append}

Let $(S,\rho)$ be a Polish metric space, $S_0 \subset S$ be a closed set
in $S$ and $0 \in S_0$ be a distinguished point. Denote by $\mathcal{K}$
the  class  of  compact  sets  $K\subset  S$. For  any  $x  \in  S$  and
$A\subset   S$,  the   distance  from   $x$   to  $A$   is  defined   by
$\rho(x,A)  = \inf\{\rho(x,y)  \colon  \,  y \in  A\}$.
Let $\mf$ be a countable set of  measurable
$[0,+\infty]$-valued   functions   on   $S$  satisfying   the   following
assumptions:
\begin{enumerate}[label = (H\arabic*)]
	\item \label{H1} The constant function $\ind$ belongs to $\mf$.
	\item \label{Hcont} All  $f\in \mf$ are  continuous
      on $S_0^c$. 
\item \label{Hbound} All  $f\in \mf$  are bounded away from zero and infinity on $\{x \in S\colon
  \, \rho(x,S_0) \geqslant \epsilon, \, \rho(x,0) \leqslant M\}$ for
  every $0 <\epsilon < M<+\infty $. 
\item \label{Hstar} For all $f\in \mf$, the set $\mf^\star(f)\subset \mf$
  of functions $f^\star \in \mf$ such that $f/f^\star$ is bounded on $S_0^c$
  and $\lim_{\rho(x,S_0) \to 0+}f(x)/f^\star(x) = 0$ is non-empty. 
\end{enumerate}
Note that assumption \ref{Hbound} is automatically satisfied when $S$ is
compact and every $f \in \mf$ is positive on $S_0^c$.  Notice   that  \ref{Hstar}  implies  that   $\mf^\star(f)$  is infinitely
countable for any $f\in \mf$.  We  shall write $f^\star$ for any element
of   $\mf^\star   (f)$.    By   \ref{H1}  and   \ref{Hstar},   we   have
$\lim_{\rho(x,S_0) \to 0+} \ind^\star(x) =+  \infty$. By convention, we take
$\ind^\star \equiv + \infty $ on $S_0$ and $f/f^\star \equiv 0$ on $S_0$
for  every  $f\in  \mf$.   We   will  occasionally  need  the  following
additional assumption:
\begin{enumerate}[resume, label = (H\arabic*)]
	\item \label{H5} $S$ is compact or $\inf_{K \in \mathcal{K}} \sup_{x
        \in K^c} f(x)/f^\star (x) = 0$ for every $f\in \mf$ (and some
      $f^\star \in \mf^\star(f)$). 	
\end{enumerate}

Denote by $\M = \mathcal{M}(S)$ the space of nonnegative finite measures
on    $S$    endowed   with    the    weak    topology.   Recall    that
$(\M,  d_{\mathrm{BL}})$, with  $d_{\mathrm{BL}}$ the  bounded Lipschitz
distance is a Polish metric space.   If $\mu \in \M$ and $f\in \cb_+(S)$,
we write $f\mu$ for the measure $f(x) \mu(\dd x)$. Set
\begin{equation} 
\MF =  \MF(S) \coloneqq \left\{ \mu \in \M \colon \, \mu(f) < \infty
  \, \text{ for all } f \in \mathfrak{F}\right\}.
\end{equation}
For $\mu \in \MF$, we have $\mu(S_0) = 0$ (as $\ind^\star \equiv +
\infty $ on $S_0$) and $f\mu  \in \M$ for every $f\in \mf$. In
particular, since $(f/f^\star )f^\star  =f$ on $S_0^c$, we  have
$(f/f^\star )f^\star \mu = f \mu$ for every $f\in \mf$ (and $f^\star \in
\mf^\star(f)$). We say  a sequence $(\mu_n, \, n \in \mathbb{N})$
of elements of $\MF$ converges to $\mu \in \MF$ if and only if
$(f \mu_n, \, n  \in \mathbb{N})$ converges to $ f  \mu$ in $\M$ for
every    $f\in \mf$. We consider the following distance $\dF$
on $\MF$ which defines the same topology: 
\begin{equation} 
\dF(\mu,\nu) = \sum_{k\in \N} \frac{1}{2^k}\left(1\wedge d_{\mathrm{BL}}\left(f_k
    \mu,f_k \nu\right)\right)
\quad\text{for} \quad \mu,
\nu\in \MF,
\end{equation}
where $\{f_k\colon\, k\in \N\}$ is an  enumeration of $\mf$. (The choice of the
enumeration is  unimportant, as  the corresponding distances  all define
the
same topology on $\MF$.)  
Notice that the   mapping
$\mu \mapsto  f\mu$ is continuous  from $\MF$ to $\M$.  In particular,
taking $f=\ind$ gives that every  sequence which converges in  $\MF$ also
converges in $\M$ to the same limit.

We shall see that the space $(\MF,\dF)$ is complete and separable (Proposition \ref{Polish space}) and give a complete description of its compact subsets (Proposition \ref{compactness}). The main goal of this section is to give conditions which allow to strengthen a convergence in $\M$ to a convergence in $\MF$ for deterministic measures (Corollary \ref{from M to MF}) and then to extend this result to random measures (Proposition \ref{main result, general} and Corollary \ref{main result, special}).
\begin{proposition}\label{Polish space}
	The space $(\MF,\dF)$ is complete and separable.
\end{proposition}

	
\begin{proof}  
	Let $(\mu_n, \, n \in \mathbb{N})$ be a Cauchy sequence in
    $\MF$. Then, by definition of $\dF$, the sequence $(f \mu_n, \, n
    \in \mathbb{N})$ is Cauchy in $\M$ for every $f \in \mf$. By
    completeness of $\M$, for every $f\in \mf$, there exists a measure $\nu_f \in \M$ such
    that $\lim_{n\to \infty}f\mu_n= \nu_f$ in $\M$. We claim that
    $\nu_f(S_0) = 0$ for every $f\in \mf$. Indeed, fix $f\in \mf$ and
    $f^\star\in \mf^\star(f)$. As $f^\star\in \mf$, we have $\lim_{n\to\infty}f^\star\mu_n=
    \nu_{f^\star}$ in $\M$. By \ref{Hstar}, the function $f/f^\star$ is continuous
    and bounded on $S$, so that the mapping $\pi \mapsto (f/f^\star)
    \pi$ is continuous on $\M$. In particular, we have $\lim_{n\to
      \infty}f \mu_n =  (f/f^\star) \nu_{f^\star}$ in $\M$. On the other
    hand, we have $\lim_{n\to \infty}f \mu_n =\nu_f$ in $\M$. We deduce
    that $\nu_f = (f/f^\star)\nu_{f^\star}$. It follows that $\nu_f(S_0)
    = 0$  since $f/f^\star = 0$ on $S_0$. 

    We set $\mu = \nu_\ind$ so
    that $\lim_{n\to \infty}\mu_n = \mu$ in $\M$. Let $f\in \mf$. We shall prove that
    $f\mu=\nu_f$.  Consider the closed set $F_k=\{f\geq  1/k\}$ for
    $k\in \N^*$. Notice
    that  $F_k \subset \operatorname{int}(F _ {k+1})$. Therefore,
    by Urysohn's lemma, there exists, for $k\in \N^*$,  a continuous function $\chi_k
    \colon S \to [0,1]$ such that $\chi_k = 1$ on $F_ k$ and
    $\operatorname{supp}(\chi_k) \subset
    \operatorname{int}(F_{k+1})$. Notice that
    $(\chi_k f/f)\mu_n = \chi_k \mu_n$ since $(f/f)=1$ on $ S_0^c$ and $\mu_n(S_0) =
    0$.   Since $\chi_k$ and $\chi_k/f$ are continuous and bounded, the
    mappings $\nu \mapsto \chi_k \nu$ and $\nu
    \mapsto (\chi_k/f) \nu$ are   continuous  from  $\M$  to  itself. We
    deduce that $\chi_k \mu= \lim_{n\to \infty} \chi_k \mu_n=
\lim_{n\to \infty} (\chi_k/f) f \mu_n=(\chi_k/f) \nu_f$ in $\M$. 
Letting $k$ go to infinity, as $\chi_k \uparrow \ind$ on $S_0^c$ since $f$ is positive
   on $S_0^c$, and $\mu(S_0)=\nu_f(S_0)=0$, we deduce (using the monotone
 convergence theorem) that $\mu=(1/f) \nu_f$ and thus  $f\mu=
 \nu_f$. Since this holds for all $f\in \mf$, this proves that $\mu
    \in \MF$ and that $\lim_{n\to \infty}f\mu_n = f \mu$ in $\M$ for
    every $f\in \mf$. Thus $\MF$ is complete.

	Next,                                                         define
    $F'_n =  \{x \in S  \colon \, \rho(x,S_0)\geqslant 1/n,  \, \rho(x,0)
    \leqslant                                                      n\}$.
    We will  identify the  space $\mathcal{M}(F'_n)$  with the  subset of
    $\M$ consisting of the measures  whose support lies in $F'_n$. Notice
    that $F'_n$ is a Polish space (when endowed with the topology induced
    by  $\rho$)  as  a  closed  subset  of  the  Polish  space  $S$.  In
    particular, the  set $\mathcal{M}(F'_n)$ endowed with  the bounded Lipschitz distance
    is a Polish space. Let $f\in \mf$. By
    \ref{Hbound}, the functions 
    $f$ and $1/f$ are both continuous and bounded on $F'_n$, so it is
    easy to check that the topology  induced by $\dF$ on $\M(F'_n)$ coincides with the
    topology  of  weak  convergence,  \emph{i.e.}  the  one  induced  by
    $d_{\mathrm{BL}}$. Therefore, the space $(\mathcal{M}(F'_n), \dF)$ is separable. To
    prove that $\MF$ is separable, it suffices to show that $\MF$ is
    equal to the completion of $\bigcup_{n\geqslant1} \mathcal{M}(F'_n)$
    with respect to  $d_{\mathrm{BL}}$.
	Notice that $F'_n \subset \operatorname{int}(F'_ {n+1})$. Therefore,
    by Urysohn's lemma, there exists a continuous function $\chi'_n
    \colon S \to [0,1]$ such that $\chi'_n = 1$ on $F'_ n$ and
    $\operatorname{supp}(\chi'_n) \subset
    \operatorname{int}(F'_{n+1})$. Let $\mu \in \MF$ and set $\mu_n =
    \chi'_n\mu$. Then it is clear that $\mu_n$ has support in $F'_{n+1}$
    and thus $\mu\in \M(F'_{n+1})$. Moreover, for every $f\in \mf$ and
    every nonnegative $h\in \cc_b(S)$, we have
\[
\mu_n(hf) =\mu(hf \chi'_n)\xrightarrow[n\to \infty]{} \mu(hf )
\]
	by the monotone convergence theorem, since $\chi'_n \uparrow
    \ind_{S_0^c}$ and $\mu(S_0) = 0$. This proves that
    $(f\mu_n,  n \in \mathbb{N})$ converges to
    $f\mu$ in $\M$ for every $f\in \mf$, thus $\dF(\mu_n,\mu)
   \to 
    0$. This  concludes the proof. 
\end{proof}

A set of measures $A \subset \M$ is said to be bounded if $\sup_{\mu \in A} \mu(\ind) < \infty$. We now give a characterization of compactness in $\MF$.
\begin{proposition}\label{compactness}Let $A \subset \MF$.
	\begin{enumerate}[label=(\roman*),leftmargin=*]
		\item $A$ is relatively compact if and only if for every $f\in
          \mf$, the family $\{f\mu\colon \, \mu \in A\}$ of finite
          measures is bounded and tight. 
        \item If \ref{H5}  holds, then $A$ is relatively  compact if and
          only    if    for    every     $f\in    \mf$,    the    family
          $\{f \mu\colon \, \mu \in A\}$ is bounded.
\end{enumerate}
\end{proposition}

\begin{proof}
  To prove  (i), start by assuming  that $A$ is relatively  compact. For
  every $\mu \in \MF$  and every $f\in \mf$, set $F_f(\mu)  = f \mu$. This
  defines a  continuous mapping $F _f\colon  \MF \to \M$. It  follows that
  the set
\[
F_f(A) = \{ f \mu\colon \mu \in A\}
\]
is relatively  compact in $\M$, \emph{i.e.}  it is bounded and  tight by
Prokhorov's theorem.
	
	Conversely,  let us  assume that  $\{ f  \mu\colon \mu  \in A\}$  is
    bounded   and   tight   in   $\M$   for   all   $f\in   \mf$.    Let
    $(\mu_n,  \, n  \in \mathbb{N})$  be a  sequence in  $A$. Since  the
    sequence of measures $(f \mu_n, \, n \in \mathbb{N})$ is bounded and
    tight,  it is  relatively  compact  in $\M$  for  every $f\in  \mf$.
    Therefore, by diagonal extraction,  there exists a subsequence still
    denoted   by  $(f   \mu_n,   \,  n   \in   \mathbb{N})$  which converges   in  $\M$ for  every
    $f\in \mf$.   By the same  argument as  in the proof  of Proposition
    \ref{Polish space}, it  follows that $(\mu_n, \,  n \in \mathbb{N})$
    converges in $\MF$.  This proves that $A$ is relatively compact.
	
    To  prove (ii),  assume that  \ref{H5}  holds. The  statement for  a
    compact $S$ follows immediately since a family of finite measures on
    a compact space is always tight.  Now assume that $S$ is not compact
    and    let    $A    \subset    \MF$    such    that    the    family
    $\{f  \mu\colon \mu  \in A\}$  is bounded  for every  $f\in \mf$. 
    To prove that $A \subset \MF$ is relatively compact, it is enough to
    show that $\{f \mu  \colon \mu \in A\}$ is tight  and to apply the
    first  point. Let $f^\star\in \mf^\star(f)$, which appears in
    \ref{H5}, 
 and  $K  \subset S$  be a  compact  subset. For  every
    $\mu \in A$, since $\mu(S_0) = 0$, we have
\begin{align*}
\int_{K^c} f(x)\,\mu(\dd x) 
&= \int_{K^c} f(x)\ind_{S_0 ^c}(x)\,\mu(\dd x)\\
&=\int_{K^c} \frac{f(x)}{f^\star(x)} \ind_{S_0^c}(x)f^\star(x)\, \mu(\dd
  x) \\
&\leqslant \mu(f^\star)\, \sup_{ K^c}\frac{f}{f^\star}\cdot
	\end{align*}
It follows that
\[
\sup_{\mu \in A} \int_{K^c} f(x)\,\mu(\dd x)
\leqslant \sup_{\mu \in
  A}\mu(f^\star)\,\,  \sup_{ K^c}\frac{f}{f^\star},
\]
	and taking the infimum over all compact subsets $K \in \mathcal{K}$
    yields, thanks to \ref{H5}
\[
\inf_{K\in \mathcal{K}} \, \sup_{\mu \in A} \int_{K^c} f(x)\,\mu(\dd x) =
0,
\]
	\emph{i.e.} the family $\{f \mu \colon \mu \in A\}$ is tight. This  completes the proof.
\end{proof}

The next result gives sufficient conditions allowing to strengthen convergence in $\M$ to convergence in $\MF$.
\begin{corollary}\label{from M to MF}
	Let $(\mu_n, \, n \in \mathbb{N})$ be a sequence of elements of $\MF$ converging in $\M$ to some $\mu \in \M$. Then $\mu \in \MF$ and $\lim_{n\to \infty}\mu_n = \mu$ in $\MF$ under either of the following conditions: 
	\begin{enumerate}[label=(\roman*),leftmargin=*]
		\item $(f\mu_n, \, n \in \mathbb{N})$ is bounded and tight for
          every $f\in \mf$. 
		\item \ref{H5} holds and $(f\mu_n, \, n \in \mathbb{N})$ is
          bounded for every $f\in \mf$. 
	\end{enumerate} 
\end{corollary}
\begin{proof}
	Either condition guarantees that the sequence $(\mu_n, \, n \in \mathbb{N})$ is relatively compact in $\MF$ by Proposition \ref{compactness}. Let $\hat{\mu}\in \MF$ be a limit point of $(\mu_n, \, n \in \mathbb{N})$. Then there exists a subsequence, still denoted by $(\mu_n, \, n \in \mathbb{N})$ such that $\lim_{n\to \infty}\mu_n = \hat{\mu}$ in $\MF$. In particular, we have $\lim_{n\to \infty}\mu_n = \hat{\mu}$ in $\M$. Since $\lim_{n\to \infty}\mu_n = \mu$ in $\M$ by assumption, it follows that $\hat{\mu} = \mu$. This proves that $\mu \in \MF$ and that $\lim_{n\to \infty}\mu_n = \mu$ in $\MF$ since the sequence $(\mu_n, \, n \in \mathbb{N})$ is relatively compact in $\MF$ and has only one limit point $\mu$.
\end{proof}
The compactness criterion of Proposition \ref{compactness} yields a tightness criterion for random measures in $\MF$.
\begin{proposition}\label{tightness}Let $\Xi$ be a family of $\MF$-valued random variables.
	\begin{enumerate}[label=(\roman*),leftmargin=*]
		\item The family $\Xi$ is tight (in distribution) in $\MF$ if
          and only if for every $f\in \mf$, the family $\{f\xi\colon \,
          \xi \in \Xi \}$ is tight (in distribution) in $\M$, \emph{i.e.} if and only if 
\begin{equation}
\label{tight1}
\lim_{r\to \infty} \sup_{\xi \in \Xi}
      \pr{\xi(f) > r} = 0
\end{equation}
	and
\begin{equation}
\label{tight2}
\inf_{K \in \mathcal{K}} \sup_{\xi \in \Xi} \ex{1\wedge \int_{K^c}f(x)
  \xi(\dd x)} = 0.
\end{equation}
\item If \ref{H5} holds, then $\Xi$  is tight (in distribution) in $\MF$
  if and only if \eqref{tight1} holds for every $f\in \mf$.
\end{enumerate}
\end{proposition}
\begin{proof}
  To prove (i),  assume that $\Xi$ is tight in  $\MF$. Since the mapping
  $F_f \colon \mu \mapsto f \mu$  is continuous from $\MF$ to $\M$ for
  every $f\in \mf $ and since tightness is  preserved by continuous
  mappings,        it        follows       that        the        family
  $F_f(\Xi) = \{f\xi \colon \, \xi  \in \Xi\}$ is tight in $\M$ for
  every $f\in \mf$. The
  result now follows from Theorem 4.10 in \cite{kallenberg2017random}.
	
	Conversely, assume that \eqref{tight1} and \eqref{tight2} hold for
    all $f\in \mf$ and
    let $\epsilon >0$. 
Let $\{f_k\colon\, k\in \N^*\}$ be an enumeration of $\mf$. We set for $k\in \N^*$:
\[
C_k= k \left(1+  \sup_{j\leq  k,\, f_k\in \mf^\star(f_j)} \norm{f_j/f_k}_\infty
  \right),
\]
with the convention that $\sup \emptyset=0$. 
For every $k \in \N^*$, there exists $r_k > 0$ and a compact set $K_k
\in \mathcal{K}$ such that 
\[
\sup_{\xi \in \Xi} \pr{\xi(f_k) > r_k} \leqslant
\frac{\epsilon}{2^{k}}
\quad\text{and}\quad
\sup_{\xi \in \Xi} \ex{1\wedge \int_{K_k ^c} f_k(x) \xi(\dd x)}
\leqslant \frac{\epsilon}{C_k 2^{k}}\cdot
\]
	Set 
\[
A_\epsilon =\bigcap _{k\in \N^*} \left\{ \mu \in \MF \colon \, \mu(f_k)
  \leqslant r_k \text{ and } \int_{K_k ^c} f_k(x) \mu(\dd x) \leqslant
  \frac{1}{C_k} \right\}. 
\]
	Then for every $\xi \in \Xi$, we have
\begin{align*}
	\pr{\xi \in A_\epsilon^c} 
&= \pr{\exists k \in \N ^*, \, \xi(f_k) > r_k \text{ or } \int_{K_k
  ^c} f_k(x) \xi(\dd x) > \frac 1{C_ k} } \\ 
&\leqslant \sum_{k\in \N^*} \pr{\xi(f_k) > r_k} + \sum_{k\in \N^*}
  \pr{\int_{K_k ^c} f_k(x) \xi(\dd x) > \frac 1{C_ k}} \leqslant 2\epsilon, 
	\end{align*}
	where in the last inequality we used that
\[
\pr{\int_{K_k ^c} f(x) \xi(\dd x) > \frac 1 {C_k}} 
= \pr{1 \wedge \int_{K_k ^c} f_k(x) \xi(\dd x) > \frac 1 {C_k}} 
\leqslant C_k \ex{1\wedge \int_{K_k ^c} f_k(x) \xi(\dd x)} 
\leqslant \frac{\epsilon}{2^{k}}\cdot
\]
	
Thus, to  prove that $\Xi$  is tight in $\MF$,  it remains to  show that
$A_\epsilon\subset    \MF$    is    relatively    compact.    We    have
$\sup_{\mu \in A_\epsilon} \mu(f_k) \leqslant  r_k < \infty$ so that the
family $\{f_k\mu\colon  \, \mu  \in A_\epsilon\}$  is bounded  for every
$k \in \N^*$.  Moreover, 
for every $i \geqslant k$ such that $f_i\in \mf^\star(f_k)$, we have 
\[
\sup_{\mu \in A_\epsilon}\int_{K_i ^c} f_k(x) \mu(\dd x) \leqslant
\norm{f_k/f_i}_{\infty} \, \sup_{\mu \in A_\epsilon}\int_{K_i
  ^c} f_i(x) \mu(\dd x) \leqslant \frac 1 i\cdot
\]
This implies that 
$\inf_{K \in  \mathcal{K}} \sup_{\mu  \in A_\epsilon}  \int_{K^c} f_k(x)
\mu(\dd                 x)                \leqslant                 1/i$
for  $i \geqslant  k$ such  that  $f_i\in \mf^\star(f_k)$.  Since there  are
infinitely many such $i$, we deduce
that 
\[
\inf_{K \in \mathcal{K}} \sup_{\mu \in A_\epsilon} \int_{K^c} f_k(x)
\mu(\dd x) = 0,
\]
\emph{i.e.} the  family $\{f_k \mu \colon  \, \mu \in A_\epsilon  \}$ is
tight.  As  this holds  for  all  $k\in  \N^*$,  we get  by  Proposition
\ref{compactness} that  $A_\epsilon$ is relatively compact  in $\MF$ (in
fact, $A_\epsilon$  is compact as  it is  closed). This proves  (i). The
proof of (ii) is similar.
\end{proof}
We now give a sufficient condition for tightness in the space $\MF$.
\begin{corollary}\label{sufficient condition for tightness}
	Assume that \ref{H5} holds. Let $\Xi$ be a family of $\MF$-valued
    random variables such that for every $f\in \mf$, 
\begin{equation}
\sup_{\xi \in \Xi} \ex{\xi(f)} < \infty.
\end{equation}
	Then $\Xi$ is tight (in distribution) in $\MF$.
\end{corollary}
\begin{proof}
	By the Markov inequality, we have for every $f\in \mf$,
\[
\sup_{\xi \in \Xi}\pr{\xi(f) > r} \leqslant \frac{1}{r} \sup_{\xi
      \in \Xi} \ex{\xi(f)} \xrightarrow[r \to \infty]{} 0.
\]
	This proves that $\Xi$ is tight in $\MF$ by Proposition \ref{tightness}-(ii).
\end{proof}
We  denote   by  $\B$   (resp.  $\BF$)   the  Borel   $\sigma$-field  on
$(\M,d_{\mathrm{BL}})$  (resp.  on  $(\MF,\dF)$).   We  also  denote  by
$\Btr = \{A \cap  \MF \colon \, A \in \B\}$  the trace $\sigma$-field of
$\B$ on $\MF$.

\begin{lemma}\label{sigma algebras}
	We have $\BF = \Btr$.
\end{lemma}
\begin{proof}
  \textbf{Step 1.} We first prove that $\MF$ is a Borel subset in $\M$.
For $g\in \cb_+(S)$, we consider the function $\Theta_g$ defined on $\M
$ by   $\Theta_g(\mu) = g\mu$. 
Denote $\cb_{b+}=\cb_b(S) \cap  \cb_+(S)$ the set of bounded
nonnegative measurable functions defined on $S$. We follow the proof of
\cite[Theorem 15.13]{aliprantis} to prove that,  for every $g\in
\cb_{b+}$, 
$\Theta_g$  is a measurable function from  $\M$ to $\M$. 
Denote by $\cf=\{g\in \cb_{b+} \colon \, \Theta_g \text{ is
  measurable}\}$. The function $\Theta_g$ is continuous for $g$ belonging to
$\cc_{b+}= \cc_b(S) \cap \cc_+(S)$. Furthermore, the set $\cf$ is
closed under bounded pointwise  convergence: if $g_n \rightarrow g$ pointwise, with $g\in
\cb_{b+}$ and $(g_n, \,n\in \N)$ a bounded sequence of elements of $\cf$
(\emph{i.e.} $\sup_{n\in \N} \norm {g_n}_\infty <\infty $), then
$\Theta_g(\mu)=\lim_{n\rightarrow \infty } \Theta_{g_n}(\mu)$ by
dominated convergence and thus  $g$ belongs to $\cf$. An immediate
extension of  \cite[Theorem 4.33]{aliprantis} gives that
$\cb_{b+}\subset \cf$. 

We then deduce that the function  $\theta_g \colon \M \to [0, +\infty ]$
  defined by $\theta_g(\mu) = g\mu(\ind)=\mu(g)$ is measurable for every $g\in
  \cb_{b+}$, and as $g\in \cb_+(S)$ is the limit of $g\wedge n \in
  \cb_{b+}$ as $n$ goes to infinity, we deduce  by monotone
  convergence that
  $\theta_g=\lim_{n\rightarrow \infty } \theta_{g\wedge n}$, and thus $\theta_g$ is
  measurable for every $g\in \cb_+(S)$.   By  definition of
  $\MF$,                  we                  have                  that
  $\MF = \bigcap_{f\in \mf}\theta_f  ^{-1} (\R_+),$ and thus 
 $\MF$ is a Borel subset in $\M$.
	
	\noindent
	\textbf{Step 2.} We prove that for every $\mu \in \MF$, the mapping
    $\nu \mapsto \dF(\mu,\nu)$ defined on $\MF$ is $\Btr$-measurable. 
    Let  $g\in \cb_{b+}$.  Since  the  function $\Theta_g$  is
    measurable   from  $\M$  to  itself by step 1, it  is  $\B/\B$-measurable.   By
    definition of the trace $\sigma$-field,  it follows that the mapping
    $\Theta_g$  from $\MF$ to $\M$  is $\Btr/\B$-measurable. Let
    $f\in  \mf$.   By monotone convergence we get that
    $\Theta_f=\lim_{n\rightarrow \infty } \Theta_{f\wedge n}$, and thus 
    $\Theta_{f}$ is $\Btr/\B$-measurable.

	Since  $\mu \in  \MF$, we  have  $f \mu  \in \M$  and the  mapping
    $\pi \mapsto  d_{\mathrm{BL}}(f\mu, \pi)$ from $\M$  to $\real$ is
    continuous hence  $\B$-measurable. Thus, by composition  we get that
    the mapping $\nu \mapsto d_{\mathrm{BL}}(f\mu, f\nu)$ from $\MF$
    to    $\real$   is    $\Btr$-measurable.   Finally,    the   mapping
    $\nu   \mapsto    \dF(\mu,\nu)$   from    $\MF$   to    $\real$   is
    $\Btr$-measurable as a sum of $\Btr$-measurable mappings.

	\noindent
	\textbf{Step 3.} We conclude the proof of the lemma. For every $\mu \in \MF$ and every $\epsilon > 0$, we have
	\[B(\mu, \epsilon) = \{ \nu \in \MF \colon \, \dF(\mu,\nu) < \epsilon\} \in \Btr\]
	by Step 2. Since $\MF$ is a Polish space, every open set is the countable union of open balls and it follows that every open set lies in $\Btr$. Hence we get $\BF \subset \Btr$.
	
	Conversely, notice that the identity mapping from $ (\MF, \dF)$ to
    $(\MF, d_{\mathrm{BL}})$ is continuous. Therefore, if $V \subset \M$
    is an open set, $V \cap \MF$ is open in $(\MF, d_{\mathrm{BL}})$ hence also in
    $(\MF, \dF)$. In particular, we have $V \cap \MF \in \BF$. Since
    this is true for every open set $V \subset \M$, we deduce that $\Btr
    \subset \BF$. 
\end{proof}

The following two results are a direct consequence of Lemma \ref{sigma algebras}.
\begin{corollary}\label{measurability}
  Let  $\xi$   be  a   $\M$-valued  random   variable  such   that  a.s.
  $\xi(f) < \infty$  for every $f\in \mf$. Then $\xi$  is a $\MF$-valued
  random  variable.  Conversely,  if  $\xi$  is  a  $\MF$-valued  random
  variable then $\xi$ is also a $\M$-valued random variable.
\end{corollary}
\begin{corollary}\label{equality in distribution}Let $\xi$ and $\zeta$ be $\MF$-valued random variables. Then the following conditions are equivalent:
	\begin{enumerate}[label = (\roman*),leftmargin=*]
		\item $\xi \law \zeta$ when viewed as $\MF$-valued random variables.
		\item $\xi \law \zeta$ when viewed as $\M$-valued random variables.
		\item $\xi(h) \law \zeta(h)$ for every $h\in \cc_b(S)$. 
		\item $\xi(fh) \law \zeta(fh)$ for every $h\in \cc_b(S)$ and
          $f\in \mf$. 
	\end{enumerate}
\end{corollary}
We now characterize convergence in distribution of random measures in
$\MF$. Recall that
\ref{H1}--\ref{Hstar} are in force. 

\begin{proposition}\label{convergence in distribution}
  Let $\xi_n$ and  $\xi$ be $\MF$-valued random  variables. Then $\xi_n$
  converges  in  distribution   to  $\xi$  in  $\MF$  if   and  only  if
  $\xi_n(f  h)  \cvlawd  \xi(f  h)$  for  every  $h\in \cc_b(S) $ and
  every $f\in \mf$. 
\end{proposition}

\begin{proof}
  Assume that $\xi_n$  converges in distribution to $\xi$  in $\MF$. Let
  $f\in \mf$.  Since $F  \colon \mu  \mapsto f  \mu$ is  continuous from
  $\MF$ to  $\M$ and  $\nu \mapsto  \nu(h)$ is  continuous from  $\M$ to
  $\real$  for  every  $h\in  \cc_b(S)$, it  follows  that  the  mapping
  $\mu  \mapsto \mu(fh)$  is continuous  from $\MF$  to $\real$.  By the
  continuous mapping theorem, we get $\xi_n(fh) \cvlaw \xi(fh)$.

Conversely,  for  every  $f\in\mf $,  $f\xi_n$  and  $f\xi$  are
$\M$-valued       random       variables,        and       we       have
 $\xi_n(fh) \cvlaw \xi(fh)$
for  every $h\in \cc_b(S)$.   By
\cite[Theorem    4.11]{kallenberg2017random},    this    implies    that
$f\xi_n  \cvlawd  f   \xi$  in  the  space   $\M$.   In  particular,
$(f \xi_n,  \, n \in \mathbb{N})$  is tight (in distribution)  in $\M$
for every $f\in \mf$.  By Proposition  \ref{tightness}, it follows
that $(\xi_n, \,  n \in \mathbb{N})$ is tight in  $\MF$.  Since $\MF$ is
Polish, Prokhorov's theorem ensures that  $(\xi_n, \, n \in \mathbb{N})$
is relatively compact  (in distribution) in $\MF$. Let  $\hat{\xi}$ be a
limit point (in  distribution) of $(\xi_n, \, n  \in \mathbb{N})$. There
exists   a   subsequence,   still   denoted  by   $\xi_n$,   such   that
$\xi_n \cvlaw  \hat{\xi}$ in $\MF$.   Let $h\in \cc_b(S)$.  Applying the
first      part      of      the       proof,      we      get      that
$\xi_n(fh)  \cvlawd  \hat{\xi}(fh)$  for   every  $f\in \mf$.
Therefore,   we   have   $\hat{\xi}(fh)    \law   \xi(fh)$   for   every
$h\in   \cc_b(S)$.   It   follows   from   Corollary  \ref{equality   in
  distribution} that  $\hat{\xi} \law \xi$  in $\MF$. Thus  the sequence
$(\xi_n, \,  n \in \mathbb{N})$ is  relatively compact and has  only one
limit point $\xi$ in $\MF$. This proves the result.
\end{proof}

We  state now the main result of this section.   Recall that
\ref{H1}--\ref{Hstar} are in force. 
\begin{proposition}\label{main result, general}
	Let $(\xi_n, \, n \in \mathbb{N})$ be a sequence of $\MF$-valued random variables and $\xi$ be a $\M$-valued random variable such that $\xi_n \cvlaw \xi$ in $\M$ and $(\xi_n, n \in \mathbb{N})$ is tight (in distribution) in $\MF$. Then $\xi$ is a $\MF$-valued random variable and we have the convergence in distribution $\xi_n \cvlaw \xi$ in $\MF$.
\end{proposition}

\begin{proof}
  By  assumption,  the  sequence  $(\xi_n,  \,  n  \in  \mathbb{N})$  is
  relatively  compact   (in  distribution)  in  the   space  $\MF$.  Let
  $\hat{\xi}  \in  \MF$  be  a  limit  point  in  distribution  and  let
$h\in \cc_b(S)$.  On the  one hand,
  Proposition  \ref{convergence in  distribution} applied  with $f=\ind $
  yields the  convergence $\xi_n(h)  \cvlaw \hat{\xi}(h)$. On  the other
  hand,   since   $\xi_n  \cvlaw   \xi$   in   $\M$  it   follows   that
  $\xi_n(h)  \cvlaw \xi(h)$.  Therefore $\hat{\xi}(h)  \law \xi(h)$  for
  every  $h\in \cc_b(S)$,  \emph{i.e.}
  $\hat{\xi} \law \xi$ in $\M$. Since the distribution of $\hat{\xi}$ is
  concentrated on  $\MF$, the  same is  true for  $\xi$. In  other words
  $\xi \in \MF$ a.s., and so  $\xi$ is a $\MF$-valued random variable by
  Corollary \ref{measurability}.  Now, applying  Corollary \ref{equality
    in  distribution}  we   get  $\hat{\xi}  \law  \xi$   in  the  space
  $\MF$. Thus the sequence $(\xi_n,  \, n \in \mathbb{N})$ is relatively
  compact   in  $\MF$   and  has   only  one   limit  point   $\xi$,  so
  $\xi_n \cvlaw \xi$ in $\MF$.
\end{proof}

The following special case is particularly useful. Recall that
\ref{H1}--\ref{Hstar} are in force.
\begin{corollary}\label{main result, special}
	Assume that \ref{H5} holds. Let $(\xi_n, \, n \in \mathbb{N})$ and
    $\xi$ be $\M$-valued random variables such that $\xi_n \cvlaw \xi$
    in $\M$ and for every $f\in \mf$,
	\begin{equation}\sup_{n} \ex{\xi_n(f)} < \infty.\label{boundedness F}
	\end{equation}
	Then $(\xi_n, \, n \in \mathbb{N})$ and $\xi$ are $\MF$-valued
    random variables and we have the convergence in distribution $ \xi_n
    \cvlaw \xi$ in $\MF$. Moreover, for every $f\in \mf$, we have
\[
 \ex{\xi(f)} \leqslant \liminf_{n \to \infty} \ex{\xi_n(f)} <
 \infty.
\]
	Furthermore, if $(\ex{\xi_n(\bullet)}, \, n\in \mathbb{N})$ converges to $ \ex{\xi(\bullet)}$ in $\M$ then the convergence actually holds in $\MF$.
\end{corollary}

\begin{proof}
	The random variable $\xi_n$ is  $\M$-valued and satisfies $\xi_n(f)
    < \infty$ a.s. since $\ex{\xi_n(f)} < \infty$ for every $f\in \mf$,
    so by Corollary \ref{measurability}, $\xi_n$ is a $\MF$-valued
    random variable. By Corollary \ref{sufficient condition for
      tightness}, the assumption \eqref{boundedness F} 
	implies that $(\xi_n, \, n \in \mathbb{N})$ is tight (in
    distribution) in $\MF$.  Therefore Proposition \ref{main result,
      general} applies and gives the convergence in distribution $\xi_n
    \cvlaw \xi$ in $\MF$. Moreover, Skorokhod's representation theorem
    in conjunction with Fatou's lemma implies that for every $f \in \mf$,
	\[\ex{\xi(f)} \leqslant \liminf_{n \to \infty}\ex{\xi_n(f)} < \infty.\]
	
	Now set  $\mu_n = \ex{\xi_n(\bullet)}$ and  $\mu= \ex{\xi(\bullet)}$
    and assume that $\mu_n \to \mu$  in $\M$. Notice that the assumption
    \eqref{boundedness  F}  implies  that  $\mu_n  \in  \MF$  for  every
    $n   \in   \mathbb{N}$   and   that   the   sequence   of   measures
    $(f\mu_n,   \,   n  \in   \mathbb{N})$   is   bounded  for   every
    $f\in \mf$. Thus Corollary  \ref{from M to MF}  gives the
    convergence $\lim_{n\to \infty}\mu_n = \mu$ in $\MF$.
\end{proof}

\section{Sub-exponential tail bounds for the height of conditioned BGW trees}\label{Kortchemski}
Assume  that  $\xi$ satisfies  \ref{xi1}  and  \ref{xi2} and  denote  by
$\rddtree^n$ a BGW($\xi$) tree conditioned to have $n$ vertices. Then by
\cite[Theorem 1]{kortchemski2017sub}  which is stated for  the aperiodic
case  but  is  trivially  extended   to  the  general  case,  for  every
$\alpha  \in  (0,\gamma/(\gamma  -   1))$,  there  exist  two  constants
$C_0,  c_0  >0$   such  that  for  every  $y  \geqslant   0$  and  every
$n \in \supp$
\begin{equation}
\label{subgaussian height zero}
\pr{\frac{b_n}{n}\H(\rddtree^n)\leqslant y} \leqslant C_0 \exp\left(-c_0 y^{-\alpha}\right).
\end{equation}
We will show that under the stronger assumption \ref{xi3}, the previous inequality holds with $\alpha = \gamma/(\gamma -1)$. Since the finite variance case has already been treated in \cite{addario2013sub}, we assume henceforth that $\xi$ has infinite variance.

Recall that $L$ is a  slowly varying function such that $\ex{\xi^2
  \ind_{\{\xi \leqslant n\}}} = n^{2-\gamma}L(n)$. On the other hand,
the slowly varying function appearing in the appendix of
\cite{kortchemski2017sub}, which we denote by $K$, satisfies
$\operatorname{Var}\left(\xi \ind_{\{\xi \leqslant n\}}\right) =
n^{2-\gamma}K(n)$. Since $\operatorname{Var} (\xi)  = +\infty$, we have
as $n$ goes to infinity that 
\[
\ex{\xi^2 \ind_{\{\xi \leqslant n\}}} \sim n^{2-\gamma} K(n)+1 \sim
n^{2-\gamma} K(n),
\]
see the appendix in \cite{kortchemski2017sub}. Therefore, we get $K(n) \sim L(n)$ and $K$ is bounded above. 

Following the proof of   \cite[Theorem 1]{kortchemski2017sub} to get
\eqref{subgaussian height zero}  holds for $\alpha=\gamma/(\gamma-1)$,
it is enough to prove the analogue of Proposition 8 therein with
$\alpha= \gamma/(\gamma-1)$, that is Proposition \ref{prop:B} below.  Let
$(W_n, \, n\in \N)$  be a random walk with starting point  $W_0 = 0$ and
jump distribution $\xi-1$.
\begin{proposition}
\label{prop:B}
  Assume that $\xi$  satisfies \ref{xi1} and \ref{xi3}.  There exist two
  constants $C_0, c_0 >0$ such that  for every $u \geqslant 0$ and every
  $n \geqslant 1$,
\begin{equation}
\label{eq:b2}
\pr{\min_{1 \leqslant i \leqslant n} W_i \leqslant -u\:\!b_n} \leqslant
C_0 \exp\left(-c_0 u^{\gamma/(\gamma-1)}\right).
\end{equation}
\end{proposition}

\begin{proof}
	Note that $\pr{\min_{1 \leqslant i \leqslant n} W_i \leqslant -u\:
      \!b_n} = 0$ if $u\: \!b_n > n$, so that it is enough to prove
    \eqref{eq:b2} for  $1\leqslant u \leqslant n/b_n$. Write, for $h >0$
\begin{equation}
\label{doob}
\pr{\min_{1 \leqslant i \leqslant n}W_i \leqslant -u\: \!b_n } =
\pr{\max_{1\leqslant i \leqslant n} \e^{-hW_i} \geqslant \e^{hu\:
    \!b_n}} \leqslant \e^{-hu\: \! b_n} \ex{\e^{-hW_n}} = \e^{-hu \: \!
  b_n} \ex{\e^{-hW_1}}^n,  
\end{equation}
	where the inequality follows  from Doob's maximal inequality applied
    to the submartingale $(\e^{-hW_n}, \allowbreak\, n \in \N)$. We
    shall apply  \eqref{doob} with  $h =  \epsilon u^{\eta}/  b_n$ where
    $\eta = 1/(\gamma - 1)$ and $\epsilon >0$ is a constant to be chosen
    later. Note that $\gamma/(\gamma - 1)=\eta \gamma = 1+\eta$. Observe
    that   $\epsilon    u^\eta/   b_n$    is   bounded    uniformly   in
    $1 \leqslant u  \leqslant n/b_n$ and $n \geqslant  1$. Indeed, since
    $b_n \geqslant \underline{b} \: \! n^{1/\gamma}$, we have
	\[\frac{u^\eta}{b_n} \leqslant \left(\frac{n}{b_n}\right)^\eta \frac{1}{b_n}\leqslant \frac{1}{\underline{b}^{1+\eta}}\cdot\]
	Therefore, by \cite[Eq. (42)]{kortchemski2017sub}, we have for every $n \geqslant 1$ and every $1\leqslant u \leqslant n/b_n$
	\[\ex{\e^{-\epsilon \frac{u^\eta}{b_n}W_1}}\leqslant \exp\left\{Cn \left(\epsilon \frac{u^\eta}{b_n}\right)^\gamma K\left(\frac{b_n}{\epsilon u^\eta}\right)\right\} \leqslant \exp\left(C'\epsilon^\gamma u^{\eta \gamma}\right),\]
	as $K$ is bounded from above and $b_n \geqslant \underline{b} \: \!
    n^{1/\gamma}$. Thus, we deduce from \eqref{doob} that for $1\leq
    u\leq  b_n/n$  
\[
\pr{\min_{1 \leqslant i \leqslant n}W_i \leqslant -u\: \!b_n } \leqslant
\exp\left(-\left(\epsilon  - C' \epsilon^\gamma\right)
  u^{1+\eta}\right).
\]
The conclusion  readily  follows by  choosing $\epsilon  >0$
small enough such that $\epsilon - C' \epsilon^\gamma >0$.
\end{proof}
\begin{remark}
	In fact, this proof is valid if we only assume that the slowly
    varying function $L$ of \ref{xi3} is bounded from above, in which
    case $n^{-1/\gamma}b_n$ is bounded below. 
\end{remark}

\medskip
\bibliographystyle{amsplain}
\bibliography{costfunctionals}

\end{document}